\documentclass[11pt,twoside,a4paper]{article}
\usepackage{amsmath,amssymb,amsthm}
\usepackage{hyperref}
\usepackage{graphicx,psfrag}
\usepackage{tikz}

\newtheorem{thm}{Theorem}[section]
\newtheorem{lem}[thm]{Lemma}
\newtheorem{pro}[thm]{Proposition}
\newtheorem{cor}[thm]{Corollary}

\theoremstyle{definition}

\theoremstyle{remark}
\newtheorem{rem}[thm]{Remark}

\newcommand{\R}{\mathbb{R}}

\newcommand{\N}{\mathbb{N}}
\newcommand{\bO}{\mathbb{O}}
\newcommand{\C}{\mathbb{C}}

\newcommand{\bH}{\mathbb{H}}

\newcommand{\cF}{\mathcal{F}}

\newcommand{\cM}{\mathcal{M}}

\newcommand{\cR}{\mathcal{R}}

\newcommand{\cU}{\mathcal{U}}

\newcommand{\al}{\alpha}
\newcommand{\be}{\beta}
\newcommand{\ga}{\gamma}

\newcommand{\de}{\delta}

\newcommand{\ep}{\varepsilon}
\newcommand{\om}{\omega}

\newcommand{\si}{\sigma}

\newcommand{\la}{\lambda}
\newcommand{\La}{\Lambda}
\renewcommand{\phi}{\varphi}

\newcommand{\crt}{\operatorname{crt}}

\newcommand{\CAT}{\operatorname{CAT}}
\newcommand{\hyp}{\operatorname{H}}
\newcommand{\id}{\operatorname{id}}

\newcommand{\pr}{\operatorname{pr}}

\newcommand{\semi}{semi$\text{-}\C$}

\newcommand{\rot}{\operatorname{rot}}

\newcommand{\im}{\operatorname{Im}}

\newcommand{\harm}{\operatorname{Harm}}
\newcommand{\fix}{\operatorname{Fix}}

\newcommand{\es}{\emptyset}
\renewcommand{\d}{\partial}
\newcommand{\di}{\d_{\infty}}
\newcommand{\set}[2]{\{#1:\,\text{#2}\}}
\newcommand{\sm}{\setminus}
\newcommand{\sub}{\subset}

\newcommand{\ov}{\overline}
\newcommand{\wt}{\widetilde}
\newcommand{\wh}{\widehat}

\begin{document}

\title{Incidence axioms for the boundary at infinity
of complex hyperbolic spaces}
\author{Sergei Buyalo\footnote{Supported by RFFI Grant
14-01-00062 and SNF Grant}
\ \& Viktor Schroeder\footnote{Supported by Swiss National
Science Foundation}}

\date{}
\maketitle

\begin{abstract}
We characterize the boundary at infinity of a complex
hyperbolic space as a compact Ptolemy space that satisfies 
four incidence axioms.
\end{abstract}


\medskip


\section{Introduction}
\label{sect:introduction}

We characterize the boundary at infinity of a complex
hyperbolic space
$\C\hyp^k$, $k\ge 1$,
as a compact Ptolemy space that satisfies four incidence
axioms.

Recall that a M\"obius structure on a set 
$X$,
or a M\"obius space over
$X$,
is a class of M\"obius equivalent metrics on
$X$, 
where two metrics are equivalent if they have the same cross-ratios.
A Ptolemy space is a M\"obius space with the property that 
the metric inversion operation preserves the M\"obius structure,
that is, the function
\begin{equation}\label{eq:metric_inversion}
d_\om(x,y)=\frac{d(x,y)}{d(x,\om)\cdot(y,\om)} 
\end{equation}
satisfies the triangle inequality for every metric
$d$
of the M\"obius structure and every
$\om\in X$.
A basic example of a Ptolemy space is the boundary at infinity
of a rank one symmetric space
$M$
of noncompact type taken with the canonical M\"obius structure,
see sect.~\ref{subsect:model_moebius_structure}. In the case
$M=\hyp^{k+1}$
is a real hyperbolic space, the boundary at infinity
$\di M$
taken with the canonical M\"obius structure is M\"obius
equivalent to an extended Euclidean space
$\wh\R^k=\R^k\cup\{\infty\}$.

\subsection{Incidence axioms}
\label{subsect:incidence_axioms}

Basic objects of our axiom system are 
$\R$-circles, $\C$-circles
and harmonic 4-tuples in a M\"obius space
$X$. 
Any
$\R$-circle
is M\"obius equivalent to the extended real line
$\wh\R$
and any 
$\C$-circle
is the square root of
$\wh\R$.
In other words, an
$\R$-circle
can be identified with
$\di\hyp^2$,
and a
$\C$-circle with
$\di(\frac{1}{2}\hyp^2)$,
where the hyperbolic plane
$\frac{1}{2}\hyp^2$
has constant curvature 
$-4$.

An (ordered) 4-tuple
$(x,z,y,u)\sub X$
of pairwise distinct points is {\em harmonic} if
$$d(x,z)d(y,u)=d(x,u)d(y,z)$$
for some and hence for any metric
$d$
of the M\"obius structure.\footnote[1]{This notion in the case of 
$\wh\R$
was introduced by Karl von Staudt  in his book ``Geometrie der Lage, 1847.} 
We use the notation
$\harm_A$
for the set of harmonic 4-tuples in
$A\sub X$.
In the case
$A=X$
we abbreviate to
$\harm$.
In the case
$X=\di M$
for
$M=\hyp^2$
or
$M=\frac{1}{2}\hyp^2$
a 4-tuple
$(x,z,y,u)$
is harmonic if and only if the geodesic lines
$xy$, $zu\sub M$
intersect each other orthogonally.  

We consider the following axioms.  

\begin{itemize}
 \item [(E)] Existence axioms

 \begin{itemize}
 \item [($\rm E_\C$)]  Through every two distinct points in
$X$
there is a unique
$\C$-circle.
 \item [($\rm E_\R$)]  For every
$\C$-circle $F\sub X$, $\om\in F$
and 
$u\in X\sm F$
there is a unique $\R$-circle
$\si\sub X$
through
$\om$, $u$
that intersects
$F_\om=F\sm\{\om\}$.
\end{itemize}

\item [(O)] Orthogonality axioms: 

For every
$\R$-circle $\si$
and every
$\C$-circle $F$
with common distinct points
$o$, $\om$
the following holds

\begin{itemize}
 \item [($\rm O_\C$)]  given
 $u$, $v\in\si$
 such that
 $(o,u,\om,v)\in\harm_\si$, 
the 4-tuple
$(w,u,o,v)$
is harmonic for every
$w\in F$,
 \item [($\rm O_\R$)] given
 $x$, $y\in F$
 such that
 $(o,x,\om,y)\in\harm_F$,
the 4-tuple
$(w,x,\om,y)$
is harmonic for every
$w\in\si$.
\end{itemize}
\end{itemize}


\subsection{Main result}
\label{subsect:main_result}

Our main result is the following

\begin{thm}\label{thm:complex_hyperbolic} Let
$X$
be a compact Ptolemy space that satisfies axioms
(E) and (O). Then
$X$
is M\"obius equivalent to the boundary at infinity of a complex
hyperbolic space
$\C\hyp^k$, $k\ge 1$,
taken with the canonical M\"obius structure.
\end{thm}

In \cite{BS2} we have obtained a M\"obius characterization of the boundary
at infinity of any rank one symmetric space assuming existence of
sufficiently many M\"obius automorphisms. In contrast to \cite{BS2},
there is no assumption in Theorem~\ref{thm:complex_hyperbolic} on
automorphisms of
$X$, 
and one of the key problems with Theorem~\ref{thm:complex_hyperbolic} 
is to establish existence of at least one nontrivial automorphism. 
In this respect, Theorem~\ref{thm:complex_hyperbolic} is an analog 
for complex hyperbolic spaces of a result obtained in \cite{FS1}
for real hyperbolic spaces: every compact Ptolemy space such that 
through any three points there exists an
$\R$-circle 
is M\"obius equivalent to
$\wh\R^k=\di\hyp^{k+1}$.

The model space
$Y=\di M$, $M=\C\hyp^k$,
serves for motivation and illustration of various notions
and constructions used in the proof of Theorem~\ref{thm:complex_hyperbolic}.
These are reflections with respect to 
$\C$-circles,
pure homotheties, orthogonal complements to
$\C$-circles,
suspensions over orthogonal complements, M\"obius joins.
In sect.~\ref{sect:model_space}, we describe the canonical
M\"obius structure on
$Y$,
show that this structure satisfies axioms (E), (O),
and study in Proposition~\ref{pro:holonomy_normal_bundle} 
holonomy of the normal bundle to a complex hyperbolic plane in
$\C\hyp^k$. 
This Proposition plays an important role in sect.~\ref{sect:moebius_join_equivalence}, 
see Lemma~\ref{lem:unique_representative_cfiber}.

Now, we briefly describe the logic of the proof. The first
important step is Proposition~\ref{pro:conjugate_pole} which
leads to existence for every
$\C$-circle $F\sub X$
of an involution
$\phi_F:X\to X$
whose fixed point set is 
$F$,
and such that every
$\R$-circle $\si\sub X$
intersecting 
$F$
at two points is 
$\phi_F$-invariant, $\phi_F(\si)=\si$.
The involution
$\phi_F$
is called the reflection with respect to
$F$.

In the second step, crucial for the paper, we obtain
a distance formula, see Proposition~\ref{pro:explicit_distance},
which is our basic tool. Using the distance formula, we establish
existence of nontrivial M\"obius automorphisms
$X\to X$
beginning with vertical shifts, see sect.~\ref{subsect:vertical_shifts}, 
and then proving that reflections with respect to
$\C$-circles
and pure homotheties are M\"obius, sect.~\ref{subsect:reflection_ccircle}
and \ref{subsect:pure_homothety}.

Next, in sect.~\ref{sect:fiber_filling_map} we introduce the notion of 
an orthogonal complement 
$A=(F,\eta)^\perp$
to a
$\C$-circle $F\sub X$
at a M\"obius involution
$\eta:F\to F$
without fixed points. In the case of the model space
$Y$, $A$
is the boundary at infinity of the orthogonal complement
to the complex hyperbolic plane
$E\sub M$
with
$\di E=F$
at
$a\in E$, $\eta=\di s_a$,
where
$s_a:E\to E$
is the central symmetry with respect to
$a$.

In the third step, we show that
for any mutually orthogonal 
$\C$-circles $F$, $F'\sub X$
the intersection
$A\cap A'$
of their orthogonal complements, when nonempty,
satisfies axioms (E), (O), see sect.~\ref{subsect:induced_argument}. 
This gives a possibility to proceed by induction on dimension.

To do that, we introduce in sect.~\ref{sect:moebius_join} the notion of 
a M\"obius join 
$F\ast F'\sub X$
of a 
$\C$-circle $F$
and its canonical orthogonal subspace
$F'\sub (F,\eta)^\perp$.
It turns out that
$X=F\ast A$
for any 
$A=(F,\eta)^\perp$.
All what remains to prove is that whenever
$F'$
is M\"obius equivalent to
$\di\C\hyp^{k-1}$,
the M\"obius join
$F\ast F'$
is equivalent to
$\di\C\hyp^k$.
The distance formula works perfectly in the case
$\dim F'=1$,
i.e.,
$F'\sub (F,\eta)^\perp$
is a
$\C$-circle,
however, if
$\dim F'>1$,
it does not work directly. Thus we proceed in two steps. 
First, we prove that
$F\ast F'=\di\C\hyp^2$
for any mutually orthogonal
$\C$-circles $F$, $F'\sub X$,
in particular, this proves Theorem~\ref{thm:complex_hyperbolic}
in the case
$\dim X=3$.
Then using this fact as a power tool, we establish that the base 
$B_\om$
of the canonical foliation
$X_\om\to B_\om$, $\om\in X$,
see sect.~\ref{sect:canonical_foliation},
is Euclidean. Using this, we finally prove the general case for
$\dim F'>1$.

{\em Acknowledgments.} The first author
is grateful to the University of Z\"urich for 
hospitality and support.

\tableofcontents

\section{The model space $\C\hyp^k$}
\label{sect:model_space}
Let
$M=\C\hyp^k$, $k\ge 1$,
be the complex hyperbolic space,
$\dim M=2k$.
For
$k=1$
we identify
$M=\frac{1}{2}\hyp^2$
with the hyperbolic plane of constant curvature
$-4$.
For
$k\ge 2$
we choose a normalization of the metric so that 
the sectional curvatures of
$M$
are pinched as
$-4\le K_\si\le -1$.

We use the standard notation
$TM$
for the tangent bundle of
$M$
and
$UM$
for the subbundle of the unit vectors.
For every unit vector
$u\in U_aM$, $a\in M$,
the eigenspaces
$E_u(\la)$
of the {\em curvature operator}
$\cR(\cdot,u)u:u^\bot\to u^\bot$,
where
$u^\bot\sub T_oM$
is the subspace orthogonal to
$u$,
are parallel along the geodesic
$\ga(t)=\exp_a(tu)$, $t\in\R$, and the respective
eigenvalues
$\la=-1,-4$
are constant along
$\ga$.
The dimensions of the eigenspaces are
$\dim E_u(-1)=2(k-1)$, $\dim E_u(-4)=1$,
$u^\bot=E_u(-1)\oplus E_u(-4)$.

Any 
$u\in U_aM$
and a unit vector
$v\in E_u(-4)$
span a 2-dimensional subspace
$L=L(u,v)\sub T_aM$
for which
$\exp_aL\sub M$
is a totally geodesic subspace isometric to
$\frac{1}{2}\hyp^2$
called a {\em complex} hyperbolic plane,
while for every unit
$v\in E_u(-1)$
the totally geodesic subspace
$\exp_aL(u,v)\sub M$
is isometric to
$\hyp^2$
and called a {\em real} hyperbolic plane.

At every point
$a\in M$
there is an isometry
$s_a:M\to M$
with unique fixed point
$a$, $s_a(a)=a$,
such that its differential
$ds_a:T_aM\to T_aM$
is the antipodal map,
$ds_a=-\id$.
The isometry
$s_a$
is called the {\em central symmetry} at
$a$.
For different 
$a$, $b\in M$
the composition
$s_a\circ s_b:M\to M$
preserves the geodesic
$\ga\sub M$
through
$a$, $b$
and acts on 
$\ga$
as a shift whose differential is a parallel 
translation along
$\ga$.
Any such isometry is called a {\em transvection}.

Furthermore, there is a complex structure
$J:TM\to TM$,
where
$J_a:T_aM\to T_aM$
is an isometry with
$J_a^2=-\id$.
Then for every
$u\in U_aM$,
the vectors
$u$, $v=J_a(u)$
form an orthonormal basis of the tangent space
to a complex hyperbolic plane
$E\ni a$.

\begin{rem}\label{rem:3dim_basis} 
Given a complex hyperbolic plane
$E\sub M$
and a point
$a\in E$,
the orthogonal complement
$E^\bot\sub M$
to
$E$
at 
$a$
is a totally geodesic subspace isometric to
$\C\hyp^{k-1}$.
\end{rem}

\subsection{Holonomy of the normal bundle to a complex plane}
\label{subsect:holonomy_normal_bundle}

Let
$E\sub M$
be a complex hyperbolic plane. Here, we describe the holonomy 
of the normal bundle 
$E^\bot$
along
$E$.
The result of this section plays an important role in
sect.~\ref{sect:moebius_join_equivalence}. For every
$a\in E$
we have a 1-dimensional fibration
$\cF_a$
of the unit sphere
$U_aE^\bot$
tangent to the orthogonal complement
$E_a^\bot$
to
$E$.
Fibers of
$\cF_a$
are circles
$T_aL\cap U_aE^\bot$,
where
$L\sub E_a^\bot$
are complex hyperbolic planes through
$a$.
By a {\em holonomy}
of
$E^\bot$
we mean a parallel translation along 
$E$
with respect to the Levi-Civita connection of
$M$.

\begin{pro}\label{pro:holonomy_normal_bundle} For every
$a\in E$
the holonomy
of
$E^\bot$
along any (piece-wise smooth) loop in
$E$
with vertex
$a$
preserves every fiber of
$\cF_a$,
though
$E^\bot$
is not flat.
\end{pro}

\begin{proof} Let
$x$, $y=J_a(x)\in U_aE$
be an orthonormal basis of the tangent space
$T_aE$.
For an arbitrary
$z\in U_aE^\bot$
we put
$u=J_a(z)$.
We actually only show that
$\cR(x,y)z=2u$,
where
$\cR$
is the curvature tensor of
$M$.
This implies the Proposition.

For
$k=1$,
where
$M=\C\hyp^k$,
there is nothing to prove, and for 
$k=2$
the first assertion is trivial.
Thus assume that 
$k\ge 3$.
We take
$v\in U_aE^\bot$
that is orthogonal to
$z$, $u$.

We use the following expression of
$\cR$ 
via
$k(x,y):=\langle\cR(x,y)y,x\rangle$
for arbitrary
$x,y,z,w\in T_aM$,
see \cite{GKM}, where
$\langle x,y\rangle$
is the scalar product on
$T_aM$,

\begin{align*}
 6\langle\cR(x,y)z,w\rangle&=k(x+w,y+z)-k(y+w,x+z)\\
 &-k(x+w,y)-k(x+w,z)-k(x,y+z)-k(w,y+z)\\
 &+k(y+w,x)+k(y+w,z)+k(y,x+z)+k(w,x+z)\\
 &+k(x,z)+k(w,y)-k(y,z)-k(w,x).
\end{align*}

Coming back to our choice of
$x,y,z,u,v$,
we obtain
$k(x,z)=k(y,z)=-1$, $k(w,y)=k(w,x)=-1$
for 
$w=u,v$.

Next, using the Euler formula
$$K_\al=K_0\cos^2\al+K_{\pi/2}\sin^2\al$$
for values of a quadratic form on a 2-dimensional vector
space, we obtain
$k(x,y+z)=k(y,x+z)=-5$, $k(x+w,y)=k(y+w,x)=-5$
for 
$w=u,v$,
$$k(x+w,z)=k(y+w,z)=\begin{cases}
                     -5,\ w=u\\
                     -2,\ w=v
                    \end{cases},
$$
and
$$k(w,y+z)=k(w,x+z)=\begin{cases}
                     -5,\ w=u\\
                     -2,\ w=v
                    \end{cases}.
$$
Therefore,
$$6\langle\cR(x,y)z,w\rangle=k(x+w,y+z)-k(y+w,x+z)$$
for 
$w=u,v$.

Now, we set 
$\wt x=\frac{1}{\sqrt 2}(x+u)$, $\wt y=J_a(\wt x)=\frac{1}{\sqrt 2}(y-z)$,
$\wt z=\frac{1}{\sqrt 2}(x+z)$, $\wt u=J_a(\wt z)=\frac{1}{\sqrt 2}(y+u)$.
For the sectional curvature in the 2-directions
$(\wt x,\wt y)$
and
$(\wt z,\wt u)$
we have
$K(\wt x,\wt y)=K(\wt z,\wt u)=-4$.
Thus
$k(y+u,x+z)=4k(\wt u,\wt z)=-16$
because
$|\wt x|=|\wt y|=|\wt z|=|\wt u|=1$.

For
$\wt y'=\frac{1}{\sqrt 2}(y+z)$,
we have
$\langle\wt y',\wt x\rangle=0=\langle\wt y',\wt y\rangle$.
Hence
$K(\wt y',\wt x)=-1$,
and we obtain
$k(x+u,y+z)=4k(\wt x,\wt y')=-4$.
Thus
$6\langle\cR(x,y)z,u\rangle=-4+16=12$,
and
$\langle\cR(x,y)z,u\rangle=2$.

For the unit vectors
$\wt v=\frac{1}{\sqrt 2}(x+v)$, $\wt v'=\frac{1}{\sqrt 2}(y+v)$,
$\wt x'=J_a(\wt y')=\frac{1}{\sqrt 2}(-x+u)$
we have
$\langle\wt v,\wt y'\rangle=0=\langle\wt v',\wt z\rangle$,
$\langle\wt v,\wt x'\rangle=-\frac{1}{2}=-\langle\wt v',\wt u\rangle$.
Thus
$K(\wt v,\wt y')=K(\wt v',\wt z)
 =-4\cos^2\frac{\pi}{3}-\sin^2\frac{\pi}{3}=-\frac{7}{4}$,
and
$k(x+v,y+z)=4k(\wt v,\wt y')=4k(\wt v',\wt z)=k(y+v,x+z)=-7$.
Therefore,
$\langle\cR(x,y)z,v\rangle=0$.

Since the vector
$\cR(x,y)z\in T_aE^\bot$
is orthogonal to
$z$,
we finally obtain
$\cR(x,y)z=2u$.
\end{proof}

\subsection{The canonical M\"obius structure on $\di\C\hyp^k$}
\label{subsect:model_moebius_structure}

We let
$Y=\di M$
be the geodesic boundary at infinity of a complex hyperbolic space
$M$.
For every
$a\in M$
the function
$d_a(\xi,\xi')=e^{-(\xi|\xi')_a}$
for
$\xi$, $\xi'\in Y$
is a (bounded) metric on
$Y$,
where
$(\xi|\xi')_a$
is the Gromov product based at
$a$.
For every
$\om\in Y$
and every Busemann function
$b:M\to\R$
centered at
$\om$
the function
$d_b(\om,\om):=0$
and
$d_b(\xi,\xi')=e^{-(\xi|\xi')_b}$,
except for the case
$\xi=\xi'=\om$,
is an extended (unbounded) metric on
$Y$
with infinitely remote point
$\om$,
where
$(\xi|\xi')_b$
is the Gromov product with respect to
$b$,
see \cite[sect.3.4.2]{BS3}.
Since
$M$
is a
$\CAT(-1)$-space, 
the metrics
$d_a$, $d_b$
satisfy the Ptolemy inequality and furthermore
all these metrics are pairwise M\"obius equivalent,
see \cite{FS2}.

We let
$\cM$
be the {\em canonical} M\"obius structure on
$Y$
generated by the metrics of type
$d_a$, $a\in M$,
i.e. any metric
$d\in\cM$
is M\"obius equivalent to some metric
$d_a$, $a\in M$.
Then
$Y$
endowed with
$\cM$
is a compact Ptolemy space. Every extended metric
$d\in\cM$
is of type
$d=d_b$
for some Busemann function
$b:M\to\R$,
while a bounded metric
$d\in\cM$
does not necessary coincide with
$\la d_a$,
for some 
$a\in M$
and 
$\la>0$,
see \cite{FS2}. We emphasize that metrics of
$\cM$
are neither Carnot-Carath\'eodory metrics nor length metrics.

\subsection{Axioms (E) and (O) for $\di\C\hyp^k$}

\begin{pro}\label{pro:axioms_model} The boundary at infinity
$Y=\di M$
of any complex hyperbolic space
$M=\C\hyp^k$, $k\ge 1$,
taken with the canonical M\"obius structure,
satisfies axioms (E), (O).
\end{pro}

\begin{proof} Every 
$\C$-circle $F\sub Y$
is the boundary at infinity of a uniquely
determined complex hyperbolic plane
$E\sub M$, 
and every
$\R$-circle $\si\sub Y$
is the boundary at infinity of a uniquely
determined real hyperbolic plane
$R\sub M$.
Axioms (E), (O) are trivially satisfied for 
$\C\hyp^1=\frac{1}{2}\hyp^2$.
Thus we assume that
$k\ge 2$.

Axiom ($\rm E_\C$): given distinct
$x$, $y\in Y$
there is a unique geodesic
$\ga\sub M$
with the ends 
$x$, $y$
at infinity. This 
$\ga$
lies in the unique complex hyperbolic plane
$E\sub M$.
Then 
$F=\di E\sub Y$
is a uniquely determined
$\C$-circle
containing
$x$, $y$.
 
Axiom ($\rm E_\R$): given a
$\C$-circle $F\sub Y$, $\om\in F$, $u\in Y\sm F$,
we let
$E\sub M$
be the complex hyperbolic plane with
$\di E=F$, $a\in E$
the orthogonal projection of
$u$
to 
$E$,
i.e. the geodesic ray
$[au)\sub M$
is orthogonal to
$E$
at 
$a$.
There is a unique geodesic 
$\ga\sub E$
through
$a$
with 
$\om$
as one of the ends at infinity. Then 
$[au)$
and
$\ga$
span a real hyperbolic plane
$R\sub M$,
for which
$\si=\di R$
is a uniquely determined 
$\R$-circle
in
$Y$
through
$\om$, $u$
that hits
$F_\om$.

Axioms (O): given an
$\R$-circle $\si\sub Y$
and a 
$\C$-circle $F\sub Y$
with common distinct points
$o$, $\om$,
there are real hyperbolic plane
$R\sub M$,
complex hyperbolic plane
$E\sub M$
with
$\di R=\si$, $\di E=F$.
Then
$R$
and
$E$
are mutually orthogonal along the geodesic
$\ga=R\cap E$.

($\rm O_\C$): given
$u$, $v\in\si$
such that
$(o,u,\om,v)\in\harm_\si$,
we can assume that
$u\neq v$.
Then the geodesic
$uv\sub R$
is orthogonal to
$\ga$
at
$a=uv\cap\ga$.
For every
$w\in F$
the geodesic rays
$[aw)\sub E$, $[au)\sub R$
span a sector
$wau$
in a real hyperbolic plane
$R'$, 
which is isometric to the sectors
$wav\sub R'$, $oau,oav\sub R$.
Thus
$d_a(o,u)=d_a(o,v)=d_a(w,u)=d_a(w,v)$
with respect to the metric
$d_a(p,q)=e^{-(p|q)_a}$
on
$Y$.
It follows that
$(w,u,o,v)\in\harm$. 

($\rm O_\R$): given
$x$, $y\in F$
such that
$(o,x,\om,y)\in\harm_F$,
we can assume that
$x\neq y$.
Then the geodesic
$xy\sub E$
is orthogonal to
$\ga$
at
$a=xy\cap\ga$
For 
$w\in\si$,
let
$b\in E$
be the orthogonal projection of
$w$
on
$E$.
Then
$b\in o\om$,
and the geodesic rays
$[bw)\sub R$, $[bx)\sub E$
span a sector
$wbx$
in a real hyperbolic plane
$R'$, 
which is isometric to the sector
$wby$
in another real hyperbolic plane
$R''$,
while the sectors
$xb\om$, $yb\om\sub E$
are isometric. Thus
$d_b(\om,x)\cdot d_b(w,y)=d_b(\om,y)\cdot d_b(w,x)$
with respect to the metric
$d_b(p,q)=e^{-(p|q)_b}$
on
$Y$.
It follows that
$(w,x,\om,y)\in\harm$. 
\end{proof}

\section{Spheres between two points}

\subsection{Briefly about M\"obius geometry}

Here we briefly recall basic notions of M\"obius geometry.
For more detail see \cite{BS2}, \cite{FS1}. A quadruple
$Q=(x,y,z,u)$
of points in a set
$X$
is said to be {\em admissible} if no entry occurs three or
four times in 
$Q$.
Two metrics 
$d$, $d'$
on 
$X$ 
are {\em M\"obius equivalent} if for any admissible quadruple
$Q=(x,y,z,u)\sub X$
the respective {\em cross-ratio triples} coincide,
$\crt_d(Q)=\crt_{d'}(Q)$,
where
$$\crt_d(Q)=(d(x,y)d(z,u):d(x,z)d(y,u):d(x,u)d(y,z))\in\R P^2.$$
We consider {\em extended} metrics on
$X$
for which existence of an {\em infinitely remote} point
$\om\in X$
is allowed, that is,
$d(x,\om)=\infty$
for all
$x\in X$, $x\neq\om$.
We always assume that such a point is unique if exists, and that
$d(\om,\om)=0$.
We use notation
$X_\om:=X\sm\om$
and the standard conventions for the calculation with 
$\om=\infty$.

A {\em M\"obius structure} on a set
$X$,
or a M\"obius space over 
$X$,
is a class 
$\cM=\cM(X)$
of metrics on
$X$
which are pairwise M\"obius equivalent. A map
$f:X\to X'$
between two M\"obius spaces
is called {\em M\"obius}, if 
$f$ 
is injective and for all admissible quadruples
$Q\sub X$
$$\crt(f(Q))=\crt(Q),$$
where the cross-ratio triples are taken with respect to
some (and hence any) metric of the M\"obius structure
of
$X$
and of 
$X'$.
If a M\"obius map
$f:X\to X'$
is bijective, then 
$f^{-1}$
is M\"obius,
$f$
is homeomorphism, and the M\"obius
spaces
$X$, $X'$
are said to be {\em M\"obius equivalent}. 

If two metrics 
of a M\"obius structure have the same infinitely remote
point, then they are homothetic, see \cite{FS1}.
We always assume that for every
$\om\in X$
the set 
$X_\om$
is endowed with a metric of the structure having
$\om$
as infinitely remote point, and use notation
$|xy|_\om$
for the distance between
$x$, $y\in X_\om$.
Sometimes we abbreviate to
$|xy|=|xy|_\om$.

A M\"obius space
$X$
is Ptolemy, if it satisfies the {\em Ptolemy inequality}
$$|xz|\cdot|yu|\le|xy|\cdot|zu|+|xu|\cdot|yz|$$
for any admissible 4-tuple
$(x,y,z,u)\sub X$
and for every metric of the M\"obius structure. This is
equivalent to the definition of a Ptolemy space given in 
sect.~\ref{sect:introduction}.

An 
$\R$-circle $\si\sub X$
satisfies the {\em Ptolemy equality}
$$|xz|\cdot|yu|=|xy|\cdot|zu|+|xu|\cdot|yz|$$
for every 4-tuple
$(x,y,z,u)\sub\si$
(in this order) and for every metric of the M\"obius structure.
In particular, for a fixed
$u\in\si$
we have
$|xz|_u=|xy|_u+|yz|_u$,
i.e. 
$\si$
is a geodesic in the space
$X_u$
called an 
$\R$-{\em line}.

A 
$\C$-circle $F\sub X$
satisfies the squared Ptolemy equality
$$|xz|^2|yu|^2=|xy|^2|zu|^2+|xu|^2|yz|^2$$
for every 4-tuple
$(x,y,z,u)\sub\si$
(in this order) and for every metric of the M\"obius structure.
In particular, for a fixed
$u\in F$
we have
$|xz|_u^2=|xy|_u^2+|yz|_u^2$,
and 
$F$
is called a 
$\C$-{\em line}
in 
$X_u$.

\subsection{Harmonic 4-tuples and spheres between two points}

An admissible 4-tuple
$(x,z,y,u)\sub X$
is {\em harmonic} if
$$\crt(x,z,y,u)=(1:\ast:1)$$
for some and hence any metric
$d$
of the M\"obius structure. Note that if
$x=y$
or
$z=u$
for an admissible
$Q=(x,z,y,u)$,
then
$Q$
is harmonic.

We say that 
$x$, $x'\in X$
lie on a sphere between distinct
$\om$, $\om'\in X$
if the 4-tuple
$(\om,x,\om',x')$
is harmonic. For a fixed
$\om$, $\om'$
this defines an equivalence relation on
$X\sm\{\om,\om'\}$,
and any equivalence class
$S\sub X\sm\{\om,\om'\}$
is called a {\em sphere between}
$\om$, $\om'$.
M\"obius maps preserve spheres between two points: if
$f:X\to X$
is M\"obius, then
$f(S)$
is a sphere between
$f(\om)$, $f(\om')$.
In the metric space
$X_\om$
the sphere
$S$
is a metric sphere centered at
$\om'$,
$$S=\set{x\in X_\om}{$|x\om'|_\om=r$}$$
for some
$r>0$.
The points
$\om$, $\om'$
are {\em poles} of
$S$.

The orthogonality axioms~(O) tell us that if
$(o,u,\om,v)$
is harmonic on an 
$\R$-circle, 
then the 
$\C$-circle 
through
$o$, $\om$
lies in a sphere between
$u$, $v$,
and vice versa, if
$(x,o,y,\om)$
is harmonic on a
$\C$-circle,
then any
$\R$-circle
through
$x$, $y$
lies in a sphere between
$o$, $\om$.

If we take some point on a sphere as infinitely remote, then
the sphere becomes the {\em bisector} between its poles,

\begin{lem}\label{lem:bisector} Let
$S\sub X$
is a sphere between distinct
$u$, $v\in X$.
Then for every
$\om\in S$
the set 
$S_\om$
is the bisector in
$X_\om$
between
$u$, $v$,
that is,
$$S_\om=\set{x\in X_\om}{$|xu|_\om=|xv|_\om$}.$$ 
\end{lem}

\begin{proof} For every
$x\in X_\om$
we have
$\crt(x,u,\om,v)=(|xu|_\om:\ast:|xv|_\om)$.
Thus
$x\in S_\om$
if and only if
$|xu|_\om=|xv|_\om$. 
\end{proof}

\begin{cor}\label{cor:bisector_cline} Let
$\si$, $F$
be 
$\R$-circle, $\C$-circle
respectively with distinct common points
$o$, $\om$.
Then
\begin{itemize}
 \item [(i)] for any
$(u,o,v,\om)\in\harm_\si$
the 
$\C$-line $F\sub X_\om$
through
$o$
lies in the bisector between
$u$, $v$,
that is, 
$|wu|_\om=|wv|_\om$
for every
$w\in F$,
 \item [(ii)] for any
$(x,o,y,\om)\in\harm_F$
any 
$\R$-line $\si\sub X_\om$
through
$o$
lies in the bisector between
$x$, $y$,
that is, 
$|wx|_\om=|wy|_\om$
for every
$w\in\si$.
\end{itemize}
In particular, if
$x,y\in F$, $u,v\in\si$
with
$|xo|_\om=|yo|_\om$
and
$|uo|_\om=|vo|_\om$,
then
$$|ux|_\om=|vx|_\om=|uy|_\om=|vy|_\om.$$
\end{cor}

\begin{proof} We have
$\{o,\om\}=F\cap\si$.
Thus by axiom~($\rm O_\C$)
the 
$\C$-line $F$
lies in the bisector between
$u$, $v$,
and by axiom~($\rm O_\R$)
the 
$\R$-line $\si$
lies in the bisector between
$x$, $y$.
\end{proof}

\begin{lem}\label{lem:intersect_rcircle_ccircle} Any
$\C$-circle $F$
and any
$\R$-circle $\si$
have at most two points in common. 
\end{lem}

\begin{proof} Assume
$o,\om\in F\cap\si$.
Choose
$u,v\in\si$
such that
$(o,u,\om,v)\in\harm_\si$.
By Corollary~\ref{cor:bisector_cline} the 
$\C$-line $F\sub X_\om$
lies in the bisector between
$u$, $v$.
Thus
$\R$-line $\si\sub X_\om$
intersects
$F_\om$
once (at
$o$).  
\end{proof}

\begin{lem}\label{lem:three_points_rcircles} If two 
$\R$-circles $\si$, $\ga$
have in common three distinct points, then they coincide,
$\si=\ga$. 
\end{lem}

\begin{proof} Assume
$x,y,\om\in\si\cap\ga$.
Then for the 
$\C$-circle $F$
through
$x,\om$
we have
$y\not\in F$
by Lemma~\ref{lem:intersect_rcircle_ccircle}.
By ($\rm E_\R$), there is at most one
$\R$-circle $\si$
with
$x,y,\om\in\si$. 
\end{proof}

\begin{cor}\label{cor:cline_projection} Given a 
$\C$-circle $F\sub X$
and
$\om\in F$,
there is a retraction
$\mu_{F,\om}:X\to F$
(continuous on 
$X_\om$), $\mu_{F,\om}(u)=u$
for
$u\in F$
and 
$$\mu_{F,\om}(u)=\si\cap F_\om$$
for 
$u\in X\sm F$,
where
$\si$
is the 
$\R$-circle
through
$\om,u$
that hits
$F_\om$.
\end{cor}

\begin{proof} By Lemma~\ref{lem:intersect_rcircle_ccircle},
the intersection
$\si\cap F_\om$
consists of a unique point, thus
$\mu_{F,\om}(u)\in F_\om$
is well defined. Assume
$u_i\to u\in X\sm F$.
We show that 
$p_i:=\mu_{F,\om}(u_i)$
is bounded. This then implies
$\mu_{F,\om}(u_i)\to\mu_{F,\om}(u)$
by uniqueness. Assume to the contrary that 
$|op_i|\to\infty$
for some basepoint
$o\in F$
in the metric of
$X_\om$.
Let
$q_i\in F$
such that
$(o,p_i,q_i,\om)\in\harm_F$.
Then
$|oq_i|\to\infty$.
But this would imply that also
$|ou_i|\to\infty$
because
$p_i$
and
$u_i$
lie on the same sphere between
$o$
and
$q_i$.
\end{proof}

\section{Involutions associated with complex circles}
\label{sect:ccircle_involutions}

\subsection{Reflections with respect to complex circles}

Let
$F\sub X$
be a 
$\C$-circle.
Then by Corollary~\ref{cor:cline_projection}, every
$u\in X\sm F$
defines an involution
$\eta_u:F\to F$
without fixed points by
$$\eta_u(\om)=\mu_{F,\om}(u).$$

\begin{pro}\label{pro:conjugate_pole} Given a
$\C$-circle $F\sub X$
and
$u\in X\sm F$,
there exists a unique
$v\in X\sm F$, $v\neq u$,
such that
$(x,u,\eta_u(x),v)\in\harm_\si$
for every
$x\in F$
and for the 
$\R$-circle $\si=\si_x$
through
$x,u,\eta_u(x)$.
\end{pro}

This results defines for a 
$\C$-circle $F$
an involution
$\phi_F:X\to X$
with fixed point set 
$F$
in the following way: for
$u\in X\sm F$
let
$v=\phi_F(u)$
be the unique point defined by 
Proposition~\ref{pro:conjugate_pole}.
For
$u\in F$
define
$\phi_F(u)=u$.
The  involution
$\phi_F$
is called the {\em reflection} with respect to
$F$.

In the case 
that
$u\in X\sm F$
we say that
$u$, $v$
are {\em conjugate poles} of
$F$.
Note that
$\eta_u=\eta_v$
for conjugate poles 
$u,v=\phi_F(u)$
of
$F$.
Thus in a simplified way one can visualize
$F$
as equator of a 2-sphere, 
$u,v$
as the poles,
the map
$\eta_u=\eta_v$
as antipodal map on the equator
such that for all
$x$
on the equator the points
$(x,u,\eta_u(x),v)$
lie in harmonic position on a circle as in the following picture,
where
$x'=\eta_u(x)=\eta_v(x)$.

\begin{figure}[htbp]\label{fi:conj_poles}
\centering
\psfrag{u}{$u$}
\psfrag{v}{$v$}
\psfrag{x}{$x$}
\psfrag{x'}{$x'$}
\psfrag{F}{$F$}
\includegraphics[width=0.4\columnwidth]{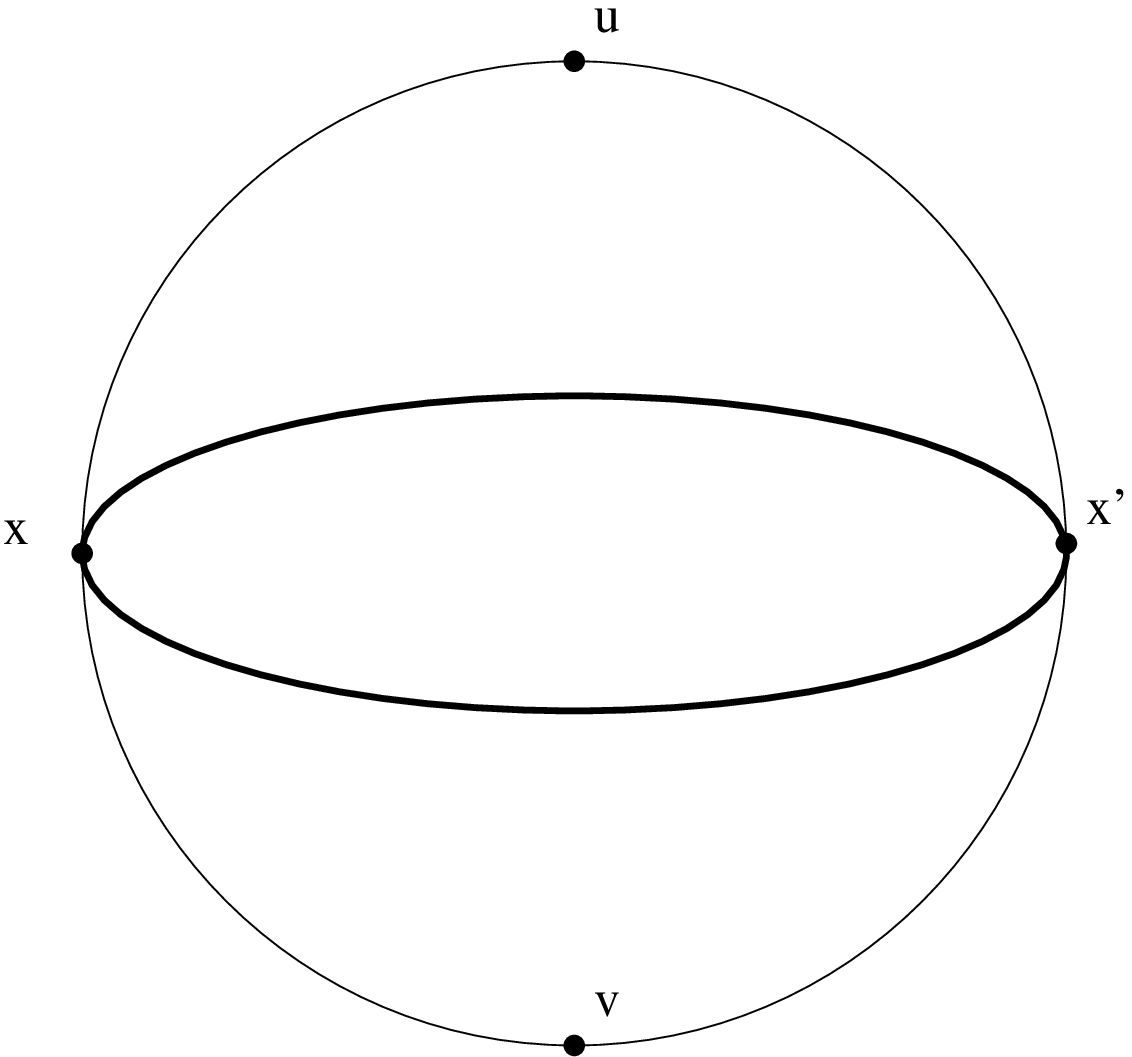}
\end{figure}

For the proof of Proposition~\ref{pro:conjugate_pole} we need 
Lemmas~\ref{lem:bisector_rcircle}--\ref{lem:concatenation}.

\begin{lem}\label{lem:bisector_rcircle} Given a
$\C$-line $F_\om\sub X_\om$, $\R$-circle $\si\sub X_\om$
and
$(x,u,y,v')\in\harm_\si$
with
$x,y\in F_\om$,
we have
$$|wu|_\om=|wv'|_\om$$
for all
$w\in F_\om$. 
\end{lem}

\begin{proof} By Corollary~\ref{cor:bisector_cline}(i), the 
$\C$-line $F_y$
lies in the bisector in
$X_y$
between
$u$, $v'$,
thus
$|wu|_y=|wv'|_y$
for every
$w\in F_y$.
Taking the metric inversion (\ref{eq:metric_inversion}) with respect to
$\om$
we obtain
$$|wu|_\om=\frac{|wu|_y}{|w\om|_y\cdot|u\om|_y}
  =\frac{|wv'|_y}{|w\om|_y\cdot|v'\om|_y}=|wv'|_\om$$  
for every
$w\in F_\om$.
\end{proof}

\begin{lem}\label{lem:sym_dist} Given
$u\in X_\om$
and a
$\C$-line $F_\om\sub X_\om$,
the distance function
$d_u:F_\om\to\R$, $d_u(w)=|wu|_\om$,
is symmetric with respect to
$o=\mu_{F,\om}(u)$,
that is,
$d_u(w)=d_u(w')$
for every
$w,w'\in F_\om$
with
$|wo|_\om=|w'o|_\om$. 
\end{lem}

\begin{proof} If
$|wo|_\om=|w'o|_\om$,
then
$(w,o,w',\om)\in\harm_F$.
Thus by Corollary~\ref{cor:bisector_cline}(ii), the 
$\R$-line $\si\sub X_\om$
through
$u$, $o$
lies in the bisector in
$X_\om$
between
$w$, $w'$.
\end{proof}

\begin{lem}\label{lem:symmetric_functions} Assume for 
$u,v'\in X_\om$
the distance functions
$d_u,d_{v'}:F_\om\to\R$
along a
$\C$-line $F_\om\sub X_\om$
coincide,
$d_u(w)=d_{v'}(w)$
for all
$w\in F_\om$,
and
$d_u$
is symmetric with respect to
$o\in F_\om$,
while
$d_{v'}$
is symmetric with respect to
$o'\in F_\om$.
Then
$o=o'$.
\end{lem}

\begin{proof} Let
$\al,\al':F_\om\to F_\om$
be isometric reflections with respect to
$o,o'$
respectively,
$\be=\al'\circ\al$.
Using
$d_u(\al x)=d_u(x)$
and
$d_{v'}(\al' x)=d_{v'}(x)$
for all
$x\in F_\om$,
we obtain
\begin{align*}
d_u(\be^{i+1}(o))=&d_{v'}(\al'\circ\al\circ\be^i(o))\\
                =&d_{v'}(\al\circ\be^i(o))\\
                =&d_u(\be^i(o)).
\end{align*}

Thus by induction, the function
$d_u$
is constant along the sequence
$\be^i(o)$.
If one admits that
$o\neq o'$,
then
$\be^i(o)\to\infty$
in
$X_\om$
and thus
$d_u(\be^i(o))\to\infty$
as
$i\to\infty$ 
being a distance function, a contradiction. Hence
$o=o'$.
\end{proof}

\begin{lem}\label{lem:concatenation} Assume that
$\mu(u)=\mu(v')=o\in F_\om$
for some
$u,v'\in X_\om$,
where
$\mu=\mu_{F,\om}:X_\om\to F_\om$.
If
$|uv'|=|uo|+|ov'|$,
then
$u,o,v'$
are contained in an
$\R$-line.
\end{lem}

\begin{proof} We can assume that
$u,v'\neq o$.
Let
$\si,\si':\R\to X_\om$
be the unit speed parametrizations of the 
$\R$-lines
through
$u,o$
and
$v',o$
respectively such that
$\si(0)=o=\si'(0)$, $\si(|ou|)=u$, $\si'(|ov'|)=v'$.
Then by the triangle inequality
$$d(s,t):=|\si(s)\si'(t)|\le s+t$$
for all
$s,t\ge 0$.
We put
$s_0=|uo|$, $t_0=|ov'|$.
The function
$d_{t_0}(s)=d(s,t_0)$
is convex and 1-Lipschitz,
$d_{t_0}(0)=t_0$, $d_{t_0}(s_0)=|uv'|=s_0+t_0$
by the assumption. Hence
$$d_{t_0}(s)=s+t_0$$
for all
$s\ge 0$.
Similarly, for a fixed
$s\ge 0$
the function
$d_s(t)=d(s,t)$
is convex and 1-Lipschitz,
$d_s(0)=s$, $d_s(t_0)=d_{t_0}(s)=s+t_0$.
Hence
$d_s(t)=s+t$
for all
$t\ge 0$.
Therefore,
$|\si(s)\si'(t)|=s+t$
for all 
$s,t\ge 0$.
Thus the concatenation
$\si|[0,\infty)\cup\si'|[0,\infty)$
is a geodesic line, which implies that
$u,o,v'$
are contained in an
$\R$-line.
\end{proof}

\begin{proof}[Proof of Proposition~\ref{pro:conjugate_pole}] Let
$o=\mu_{F,\om}(u)\in F_\om$, $\ga$
the 
$\R$-line
in
$X_\om$
through
$u$, $o$.
We take
$v\in\ga$, $v\neq u$,
such that
$|vo|=|ou|$.

Given
$x\in F$,
we put
$y=\eta_u(x)\in F$,
and let
$\si$
be the 
$\R$-circle
through
$x,u,y$.
We show that
$v'=v$
for 
$v'\in\si$, $v'\neq u$,
such that
$(x,u,y,v')\in\harm_\si$.

By Lemma~\ref{lem:bisector_rcircle} the distance functions
$d_u,d_{v'}:F_\om\to\R$, $d_u(w)=|wu|$, $d_{v'}(w)=|wv'|$,
coincide,
$|wu|=|wv'|$
for all
$w\in F_\om$.
By Lemma~\ref{lem:sym_dist},
$d_u$
is symmetric with respect to
$o$,
while
$d_{v'}$
is symmetric with respect to
$o'=\mu_{F,\om}(v')$.
Thus
$o'=o$
by Lemma~\ref{lem:symmetric_functions}.

By the triangle inequality
$|uv'|\le|uo|+|ov'|=2|uo|=|uv|$.
On the other hand, by the Ptolemy equality we have
$$|uv'|\cdot|xy|=|xu|\cdot|v'y|+|yu|\cdot|xv'|.$$
By Lemma~\ref{lem:bisector_rcircle},
$|xv'|=|xu|$,
and by Corollary~\ref{cor:bisector_cline}(i),
$|xu|=|xv|$.
Similarly,
$|yv'|=|yu|=|yv|$.
By the Ptolemy inequality
$$|xu|\cdot|vy|+|yu|\cdot|xv|\ge|xy|\cdot|uv|.$$
We conclude that
$|uv'|\ge|uv|$.
Hence,
$|uv'|=|uv|=|uo|+|ov'|$.
From Lemma~\ref{lem:concatenation}
we conclude that
$v'\in\ga\cap\si$,
thus
$v'=v$.
\end{proof}

\subsection{A M\"obius involution of a complex circle}

\begin{lem}\label{lem:harmonic_midpoint} Given a
$\C$-circle $F$, $x$, $y\in F$, $u\in X\sm F$,
there exists 
$\om\in F$
such that
$(x,\eta_u(\om),y,\om)\in\harm_F$. 
\end{lem}

\begin{proof} We put
$x'=\eta_u(x)$, $y'=\eta_u(y)$.
Then for every
$\om\in x'y'$
we have
$o=\eta_u(\om)\in xy$,
where
$xy\sub F$
is the arc between
$x,y$
that does not contain  
$x'$, $y'$,
and
$x'y'\sub F$
is the arc between
$x',y'$
that does not contain  
$x$, $y$.
The cross-ratio function
$f:x'y'\to\R$,
defined by
$$f(\om)=\frac{|xo|\cdot|y\om|}{|x\om|\cdot|oy|}$$
via a metric of the M\"obius structure, is continuous and takes values
$f(x')=0$, $f(y')=\infty$.
Then
$\om$
with
$f(\om)=1$
does the job. 
\end{proof}

\begin{lem}\label{lem:filling_sphere_rcirles} Assume a sphere
$S\sub X$
between
$\om$, $w'$
intersects the 
$\C$-circle $F$
through
$\om$, $w'$
at
$x$, $x'$.
Then for every
$u\in S$
there is an
$\R$-circle $\ga'$
with 
$x,u,x'\in\ga'$.
\end{lem}

\begin{proof} By Lemma~\ref{lem:harmonic_midpoint} there is
$\wt x\in F$
such that
$(w',\wt x,\om,\wt x')\in\harm_F$,
where
$\wt x'=\eta_u(\wt x)$,
in particular, there is an
$\R$-circle $\ga'\sub X$
with
$\wt x,u,\wt x'\in\ga'$.
By axiom $\rm O_\R$, $\ga'$
lies on a sphere
$S'$
between
$w',\om$.
Since
$u\in\ga'\cap S$,
we conclude that
$S'=S$
and therefore
$\{\wt x,\wt x'\}=\{x,x'\}$. 
\end{proof}

\begin{pro}\label{pro:moebius_involution} For every
$x$, $z\in F$, $u\in X\sm F$
we have
$$\crt(x,u,z,v)=\crt(\eta_u(x),v,\eta_u(z),u),$$
where
$v\in X\sm F$
is the conjugate to
$u$
pole of
$F$.
\end{pro}

\begin{proof} We put
$z'=\eta_u(z)\in F$.
By Lemma~\ref{lem:harmonic_midpoint} there is
$\om\in F$
such that
$(x,o,z',\om)\in\harm_F$,
where
$o=\eta_u(\om)$.
By definition of
$\eta_u$,
there is an
$\R$-circle $\si$
through
$\om$, $u$
and
$o$.
Then
$F$, $\si$
are 
$\C$-line, $\R$-line
respectively in  the metric space
$X_\om$,
and
$o$
is the midpoint between
$x$, $z'$, $|xo|=|oz'|$,
where
$|xo|=|xo|_\om$
is the distance in
$X_\om$.

We denote by
$\psi:F\to F$
the metric reflection with respect to
$o$.
Then 
$\psi$
is an isometry with
$\psi(o)=o$, 
in particular,
$\psi$
is M\"obius. Note that
$F$
is in the bisector in
$X_\om$
between
$u$
and
$v$,
that is
$|wu|=|wv|$
for every
$w\in F$.
Furthermore,
$\si$
is in the bisector between any
$w\in F$
and
$\psi(w)$.
Thus
$\psi(x)=z'=\eta_u(z)$.
We show that
$x':=\psi(z)=\eta_u(x)$.

Since
$z'=\eta_u(z)$,
there is an 
$\R$-circle
$\ga\sub X$
such that
$(z,u,z',v)\in\harm_\ga$.
Let
$w\in F$
be the midpoint between
$z$, $z'$, $|zw|=|wz'|=:r$.
The circle
$\ga$
is a metric circle in
$X_\om$
centered at
$w$,
in particular,
$|wu|=|wv|=r$.
We put
$w'=\psi(w)$.
Then
$|w'u|=|w'v|=r=|w'x'|$
and
$|xw'|=|wz'|=r$.
It follows that
$x,u,x',v$
lie on the metric sphere of radius
$r$
centered at
$w'$.
Thus by Lemma~\ref{lem:filling_sphere_rcirles} there is a circle
$\ga'\sub X$
through
$x$, $x'$
with
$u\in\ga'$.
This means that
$x'=\eta_u(x)$.

It follows
\begin{align*}
 \crt(x,u,z,v)&=(|xu|\cdot|zv|:|xz|\cdot|uv|:|xv|\cdot|uz|)\\
              &=\crt(\psi(x),u,\psi(z),v)\\
              &=\crt(\eta_u(z),u,\eta_u(x),v)=\crt(\eta_u(x),v,\eta_u(z),u).
\end{align*}
\end{proof}

\begin{cor}\label{cor:phi_u_moebius} For every 
$\C$-circle $F$
and every 
$u\in X\sm F$
the involution
$\eta_u:F\to F$
is M\"obius.
\end{cor}

\begin{proof} Note that
$\crt(x,u,y,v)=(1:\ast:1)$
for each
$x$, $y\in F$,
where
$v$
is the conjugate to
$u$
pole of
$F$.
Thus in the space
$X_v$
we have
$|x'y'|_v=|xy|_v$
by Proposition~\ref{pro:moebius_involution}, where
$x'=\eta_u(x')$, $y'=\eta_u(y)$.
Hence
$\eta_u$
preserves the cross-ratio triple of any admissible 4-tuple in
$F$.
\end{proof}

\section{A distance formula}

Using Proposition~\ref{pro:moebius_involution}, we derive here
a distance formula (see Proposition~\ref{pro:explicit_distance})
which plays a very important role in the paper.

\begin{lem}\label{lem:wharm_basics} Let
$u$, $v\in X$
be conjugate poles of a
$\C$-circle $F$.
Then for every
$x$, $z\in F$
we have
$$|xz|_\om\cdot|zy|_\om=|zu|_\om^2,$$
where
$y=\eta_u(x)$, $\om=\eta_u(z)$.
\end{lem}

\begin{proof} We consider a background metric
$|\ |_v$
with infinitely remote point
$v$
and we define
$|\ |_\om$
and
$|\ |_z$
as the metric inversions, i.e.
$$|ef|_\om=\frac{|ef|_v}{|e\om|_v\cdot|f\om|_v}$$
and
$$|ef|_z=\frac{|ef|_v}{|ez|_v\cdot|fz|_v}.$$
Since
$F$
lies on a sphere between
$u$, $v$,
we have
$|uw|_v=:\rho$
for every
$w\in F$
and observe that
$|z\om|_v=2\rho$.
Thus we have
$|uz|_\om=|u\om|_z=\frac{1}{2\rho}$.

Furthermore,
$$|xz|_\om=\frac{|xz|_v}{|x\om|_v\cdot|z\om|_v}=\frac{|xz|_v}{2\rho|x\om|_v}$$
and
$$|y\om|_z=\frac{|y\om|_v}{|yz|_v\cdot|z\om|_v}=\frac{|y\om|_v}{2\rho|yz|_v}.$$

By Proposition~\ref{pro:moebius_involution}, 
$\crt(x,u,z,v)=\crt(y,v,\om,u)$
and also
$\crt(y,u,z,v)=\crt(x,v,\om,u)$.
Each of these cross-ratios is of the form
$(1:\ast:1)$.
Hence
$|xz|_v=|y\om|_v$
and
$|yz|_v=|x\om|_v$.
It follows that
$|xz|_\om=|y\om|_z$
and similarly
$|yz|_\om=|x\om|_z$.
 
We denote by
$a=|xz|_\om$, $d=|yz|_\om$,
and consider the metric inversion
$$d_z(p,q)=\frac{ad|pq|_\om}{|zp|_\om\cdot|zq|_\om}$$
with respect to
$z$.
We calculate
$d_z(x,\om)=\frac{ad}{|xz|_\om}=d=|x\om|_z$,
$d_z(y,\om)=\frac{ad}{|yz|_\om}=a=|y\om|_z$.
Thus two metrics
$d_z(p,q)$
and
$|pq|_z$
with the common infinitely remote point
$z$
coincide, in particular,
$d_z(u,\om)=|u\om|_z=|uz|_\om$.

On the other hand,
$d_z(u,\om)=\frac{ad}{|uz|_\om}$.
Therefore
$\frac{ad}{|uz|_\om}=|uz|_\om$
and hence
$|xz|_\om\cdot|yz|_\om=|uz|_\om^2$. 
\end{proof}

\begin{pro}\label{pro:explicit_distance} Given
$o$, $u\in X_\om$,
let
$F\sub X_\om$
be the 
$\C$-line
through
$o$
(and 
$\om$),
$z=\mu_{F,\om}(u)$,
and let
$a=|zu|_\om$, $b=|oz|_\om$, $r=|ou|_\om$.
Then
$r^4=a^4+b^4$.
\end{pro}

\begin{proof} We take
$x$, $y\in F$
with
$|xo|_\om=|oy|_\om=r$.
Without loss of generality we can assume that
$o<z<x$.
Then
$|xz|_\om^2=r^2-b^2$, $|yz|_\om^2=r^2+b^2$.
We have
$\om=\eta_u(z)$.
The points
$x$, $y$, $u$
lie on the metric sphere of radius
$r$
centered at
$o$.
Thus by Lemma~\ref{lem:filling_sphere_rcirles} there exists an
$\R$-circle $\ga\sub X_\om$
through
$x$, $y$, $u$,
and therefore
$y=\eta_u(x)$.
Using that 
$|xz|_\om^2\cdot|zy|_\om^2=a^4$
by Lemma~\ref{lem:wharm_basics}, we obtain
$r^4=a^4+b^4$.
\end{proof}

\begin{rem}\label{rem:koranyi_gauge} As it is explained in
\cite{BS1}, the distance formula
$r^4=a^4+b^4$
gives rise to the Kor\'anyi gauge on the Heisenberg group
$\bH^{2k-1}$,
see \cite[Remark~14.4]{BS1}.
\end{rem}

\section{The canonical foliation of $X_\om$}
\label{sect:canonical_foliation}

We fix 
$\om\in X$
and denote by
$B_\om$
the set of all the 
$\C$-circles
in
$X$
through
$\om$.
By axiom~($\rm E_\C$), for every
$x\in X_\om$
the is a unique
$\C$-circle
$F\in B_\om$
with
$x\in F$.
This defines a map 
$\pi_\om:X_\om\to B_\om$, $\pi_\om(x)=F$,
called the {\em canonical} projection, the set 
$B_\om$
is called the {\em base} of
$\pi_\om$,
and the foliation of
$X_\om$
by the fibers of 
$\pi_\om$
is said to be {\em canonical}. In this section we use notation
$|xy|$
for the distance between
$x$, $y\in X_\om$
in the metric space
$X_\om$.

\subsection{Busemann functions on $X_\om$}

\begin{lem}\label{lem:busemann_constant} Let
$\si\sub X_\om$
be an
$\R$-line.
Any Busemann function 
$b:X_\om\to\R$
associated with 
$\si$
is constant along any
$\C$-line $F\sub X_\om$
which meets
$\si$.
\end{lem}

\begin{proof} We parametrize
$\si:\R\to X_\om$
by arclength such that
$\si(0)=\si\cap F$
and 
$b$
decreases along
$\si$.
By Proposition~\ref{pro:explicit_distance}
we have
$|x\si(t)|^4=t^4+|x\si(0)|^4$
for every
$x\in F$
and
$t\in\R$.
Therefore
$b(x)=\lim_{t\to\infty}(|x\si(t)|-t)=0$.
\end{proof}

\begin{lem}\label{lem:distance_clines} For every
$\C$-line $F\sub X_\om$
the function
$x\mapsto|x\mu_{F,\om}(x)|$
is constant along any 
$\C$-line $F'\in X_\om$.
\end{lem}

\begin{proof} For
$x'\in F'$
let
$\si:\R\to X_\om$
be the arclength parametrization of the uniquely
determined
$\R$-line
with
$\si(0)=\mu_{F,\om}(x')$, $\si(a)=x'$,
where
$a=|x'\mu_{F,\om}(x')|$.
The Busemann function of
$\si$
is 1-Lipschitz and by Lemma~\ref{lem:busemann_constant}
it is constant on
$F$
as well as on
$F'$.
It follows that
$|yy'|\ge a$
for every
$y\in F$, $y'\in F'$.
In particular,
$|y'\mu_{F,\om}(y')|\ge a$
for all 
$y'\in F'$.
By symmetry we get equality.
\end{proof}

\subsection{Canonical metric on the base}

Given
$F$, $F'\in B_\om$
we put
$|FF'|:=|x'\mu_{F,\om}(x')|$
for some
$x'\in F'$.
By Lemma~\ref{lem:distance_clines} this is well defined
and moreover, it is a metric on
$B_\om$.
Indeed, given
$F$, $F'$, $F''\in B_\om$,
we take
$x'\in F'$
and put
$x=\mu_{F,\om}(x')\in F$, $x''=\mu_{F'',\om}(x')\in F''$.
Then
$|FF''|\le|xx''|\le|xx'|+|x'x''|=|FF'|+|F'F''|$. 
This metric on
$B_\om$
is said to be {\em canonical}.

The projection
$\pi_\om:X_\om\to B_\om$
is 1-Lipschitz and isometric if restricted to any
$\R$-line
in
$X_\om$.
This implies that every two points in
$B_\om$
lie on a geodesic line.

\begin{lem}\label{lem:base_unique_geodesic} The base
$B_\om$
with the canonical metric is uniquely geodesic, i.e.
between any two points there is a unique geodesic.
In particular,
$B_\om$
is contractible.
\end{lem}

\begin{proof} Given pairwise distinct
$F$, $F'$, $F''\in B_\om$
with
$|FF''|=|FF'|+|F'F''|$,
we take
$x'\in F'$
and put
$x=\mu_{F,\om}(x')\in F$, $x''=\mu_{F'',\om}(x')\in F''$,
$y=\mu_{F,\om}(x'')\in F$.
Then
$$|yx''|\le|xx''|\le|xx'|+|x'x''|=|FF''|=|yx''|.$$
Therefore
$y=x$
by Proposition~\ref{pro:explicit_distance}.
Since
$\mu(x)=\mu(x'')=x'$
for 
$\mu=\mu_{F',\om}$
and
$|xx''|=|xx'|+|x'x''|$,
the points
$x$, $x'$, $x''$
are contained in an
$\R$-line by Lemma~\ref{lem:concatenation}. This line
is unique by Lemma~\ref{lem:three_points_rcircles}.
The claim follows.
\end{proof}

\begin{lem}\label{lem:mu_isometric} For every
$\C$-line $F\sub X_\om$
the retraction
$\mu_{F,\om}:X_\om\to F$
if restricted to any
$\C$-line $F'\sub X_\om$
is isometric.
\end{lem}

\begin{proof} Given
$x'$, $y'\in F'$
we let
$x=\mu_{F,\om}(x')$, $y=\mu_{F,\om}(y')$,
$a=|FF'|$, $b=|xy|$.
By Proposition~\ref{pro:explicit_distance},
$|x'y|^4=a^4+b^4$.
On the other hand,
$y'=\mu_{F',\om}(y)$.
Thus again by Proposition~\ref{pro:explicit_distance},
$|x'y|^4=a^4+|x'y'|^4$.
Hence,
$|xy|=|x'y'|$. 
\end{proof}

\subsection{Busemann functions are affine}

\begin{lem}\label{lem:limit_circle} Assume that
$\om_i\to\om$
in
$X$,
and a point
$x\in X$
distinct from
$\om$
is fixed. Then any $\R$-circle
$l\sub X$
through
$\om$, $x$
is the (pointwise) limit of a sequence of $\R$-circles
$l_i\sub X$
through
$\om_i$, $x$.
\end{lem}

\begin{proof} Let
$o\in l$
be a point different from
$x$, $\om$,
and let 
$F=F(o,\om)\sub X$
be the 
$\C$-circle 
through
$o$, $\om$.
For 
$i$
we let 
$F_i=F(o,\om_i)$
be a
$\C$-circle
through
$o$, $\om_i$,
which is uniquely determined for sufficiently
large 
$i$
as soon as
$\om_i\neq o$.
Since
$X$
is compact,
$F_i\to F$
pointwise and thus for 
$i$
large enough
$x\not\in F_i$.
By axiom~($\rm E_\R$) there is uniquely determined
$\R$-circle $l_i\sub X$
through
$\om_i$, $x$
which hits
$F_i\sm\om_i$.
Then
$l_i$
subconverges to an
$\R$-circle $\ov l$
through
$\om$, $x$
that hits
$F\sm\om$.
By axiom~($\rm E_\R$),
$\ov l=l$
and
$l_i\to l$
follows.
\end{proof}

The following result has been obtained in \cite[Proposition~4.1]{FS1},
\cite[Corollary~3.19]{BS2} under different assumptions.

\begin{pro}\label{pro:busemann_affine} Given two 
$\R$-lines $l$, $l'\sub X_\om$,
the Busemann functions of
$l$
are affine on
$l'$.
\end{pro}

\begin{proof} Let
$b:X_\om\to\R$
be a Busemann function of
$l$.
We can write
$b(x)=\lim_{i\to\infty}(|x\om_i|-|o\om_i|)$
for every
$x\in X_\om$,
where
$o\in l$
is some fixed point, $l\ni\om_i\to\om$.

Let
$m\in l'$
be the midpoint between
$x$, $y\in l'$, $|xm|=|my|=\frac{1}{2}|xy|$.
We have to show that
$b(m)=\frac{1}{2}(b(x)+b(y))$.
Busemann functions in any Ptolemy space are convex,
see \cite[Proposition~4.1]{FS1}, thus
$b(m)\le\frac{1}{2}(b(x)+b(y))$.

By Lemma~\ref{lem:limit_circle}, for every sufficiently large
$i$
there is an
$\R$-circle $l_i\sub X_\om$
through
$x$
and
$\om_i$
such that the sequence
$l_i$
converges pointwise to
$l'$.
Thus there are points
$y_i\in l_i$
with
$y_i\to y$.
The points
$x$, $y_i$
divide
$l_i$
into two segments. Choose a point
$m_i$
in the segment that does not contain
$\om_i$
such that
$|xm_i|=|m_iy_i|$.
One easily sees that
$m_i\to m$.

The points
$x$, $m_i$, $y_i$, $\om_i$
lie on the 
$\R$-circle $l_i$
in this order. Thus
$$|xy_i|\cdot|m_i\om_i|=|xm_i|\cdot|y_i\om_i|+|m_iy_i|\cdot|x\om_i|,$$
and since
$|xm_i|=|m_iy_i|\ge\frac{1}{2}|xy_i|$,
we see 
$|m_i\om_i|\ge\frac{1}{2}(|x\om_i|+|y_i\om_i|)$.
This implies in the limit
$b(m)\ge\frac{1}{2}(b(x)+b(y))$. 
\end{proof}

\begin{lem}\label{lem:busemann_constant_fibers} Busemann
functions are constant on the fibers of
$\pi_\om:X_\om\to B_\om$.
\end{lem}

\begin{proof} Let
$b:X_\om\to\R$
be a Busemann function associated with an
$\R$-line $l\sub X_\om$,
and let
$F\sub X_\om$
be a
$\C$-line, $x\in F$, $c=b(x)\in\R$.
Using Lemma~\ref{lem:busemann_constant} we can assume that
$l\cap F=\es$.
We take
$y\in l$
with
$b(y)=c$
and consider the 
$\C$-line $F'\sub X_\om$
through
$y$.
Then 
$F'\cap F=\es$.
By Lemma~\ref{lem:busemann_constant}, the function
$b$
takes the constant value
$c$
on
$F'$.
Let
$z=\mu_{F',\om}(x)\in F'$.
Then
$b(z)=b(x)=c$,
and there is a uniquely determined
$\R$-line $\si\sub X_\om$
through
$x$, $z$.
By Proposition~\ref{pro:busemann_affine} the function
$b$
takes the constant value
$c$
along
$\si$.

Given
$x'\in F$,
for the 
$\R$-line $\si'\sub X_\om$
through
$x'$, $z'=\mu_{F',\om}(x')\in F'$
we have
$\pi_\om(\si')=\pi_\om(\si)\sub B_\om$
by Lemma~\ref{lem:base_unique_geodesic}.
Thus the values of
$b$
along
$\si'$
are uniformly bounded because
$b$
is Lipschitz and the distance of any
$u\in\si'$
to
$\si$
equals
$|x'x|=|z'z|$
by Lemma~\ref{lem:mu_isometric}. Since
$b|\si'$
is affine, we conclude that
$b|\si'\equiv b(z')=c$,
in particular,
$b(x')=c$. 
\end{proof}

%
%
%
%
%
%

\subsection{Properties of the base}

\begin{pro}\label{pro:base_normed_space} The base
$B_\om$
is isometric to a normed vector space of a finite dimension
with a strictly convex norm.
\end{pro}

\begin{proof} By Lemma~\ref{lem:base_unique_geodesic},
$B_\om$
is a geodesic metric space such that through any two distinct
points there is a unique geodesic line. We show that affine
functions on
$B_\om$
separate points. Any Busemann function
$b:X_\om\to\R$
is affine on 
$\R$-lines 
by Proposition~\ref{pro:busemann_affine}. 
By Lemma~\ref{lem:busemann_constant_fibers},
$b$
is constant on the fibers of
$\pi_\om$,
thus it determines a function
$\ov b:B_\om\to\R$
such that
$\ov b\circ\pi_\om=b$.
Every geodesic line
$\ov l\sub B_\om$
is of the form
$\ov l=\pi_\om(l)$
for some 
$\R$-line
$l\sub X_\om$,
and each unit speed parameterization
$c:\R\to X_\om$
of
$l$
induces the unit speed parameterization
$\ov c=\pi_\om\circ c$
of
$\ov l$.
Then
$\ov b\circ\ov c=\ov b\circ\pi_\om\circ c=b\circ c$
is an affine function on
$\R$.
Through any
$x$, $x'\in B_\om$,
there is a geodesic line
$\ov l=\pi_\om(l)$.
Let
$b$
be a Busemann function on
$X_\om$
associated with
$l$.
Then
$b$
takes different values on the fibers of
$\pi_\om$
over
$x$, $x'$
respectively. Thus the affine function
$\ov b$
separates the points
$x$, $x'$, $\ov b(x)\neq\ov b(x')$.
Then by \cite{HL},
$B_\om$
is isometric to a convex subset of a normed vector space
with a strictly convex norm. Since
$B_\om$
is geodesically complete, i.e., through any two points
there is a geodesic line, this subset is a subspace,
and therefore
$B_\om$
is isometric to a normed vector space
$E$.
The Ptolemy space
$X$
is compact, thus
$B_\om$
is locally compact, and the dimension of
$E$
is finite. 
\end{proof}

\begin{cor}\label{cor:homeo_sphere} The space
$X$
is homeomorphic to sphere
$S^{k+1}$, $k\ge 0$.
\end{cor}

\begin{proof} By Proposition~\ref{pro:base_normed_space},
$B_\om$
is homeomorphic to an Euclidean space
$\R^k$,
and since
$\pi_\om:X_\om\to B_\om$
is a fibration over
$B_\om$
with fibers homeomorphic to
$\R$,
the space
$X_\om$
is homeomorphic to
$\R^{k+1}$.
Thus
$X$
is homeomorphic to
$S^k$.
\end{proof}

\begin{pro}\label{pro:property_u} The Ptolemy space
$X$
has the following property: Any 4-tuple
$Q\sub X$
of pairwise distinct points lies on an 
$\R$-circle $\si\sub X$
provided three of the points of
$Q$
lie on 
$\si$
and the Ptolemy equality holds for
the cross-ratio triple
$\crt(Q)$.
\end{pro}

\begin{proof} We assume that
$Q=(x,y,z,u)$, $x,y,z\in\si$
and
$$|xz|\cdot|yu|=|xy|\cdot|zu|+|xu|\cdot|yz|.$$
Choosing
$y$
as infinitely remote, we have
$|xz|_y=|xu|_y+|uz|_y$
and 
$\si$
is an
$\R$-line 
in
$X_y$.
Recall that the canonical projection
$\pi_y:X_y\to B_y$ 
is 1-Lipschitz and isometric on every
$\R$-line.
Thus
$$|\ov x\,\ov z|=|xz|_y=|xu|_y+|uz|_y\ge|\ov x\ov u|+|\ov u\ov z|,$$
where
$\ov x=\pi_y(x)$.
The triangle inequality in
$B_y$
implies
$|\ov x\,\ov z|=|\ov x\,\ov u|+|\ov u\,\ov z|$.
By Proposition~\ref{pro:base_normed_space},
$B_y$
is isometric to a normed vector space with a strictly convex norm.
Thus we conclude that
$\ov u$
lies between
$\ov x$
and
$\ov z$
on a line in
$B_y$.
This means that the 
$\C$-line $F\sub X_y$
through
$u$
hits
$\si$,
and for 
$o=\si\cap F$
we have
$|xz|_y=|xo|_y+|oz|_y$.
By Proposition~\ref{pro:explicit_distance},
$|xo|_y<|xu|_y$
and
$|oz|_y<|uz|_y$
unless
$u=o$.
We conclude that
$u=o\in\si$.
\end{proof}

\section{M\"obius automorphisms}

\subsection{Vertical shifts}
\label{subsect:vertical_shifts}

We fix 
$\om\in X$
and change notation using the letter
$b$
for elements of the base
$B_\om$
and
$F_b$
for the respective fiber of
$\pi_\om$.
For any two 
$b$, $b'\in B_\om$
we have by Lemma~\ref{lem:mu_isometric} the isometry
$\mu_{bb'}:F_b\to F_{b'}$.

\begin{lem}\label{lem:continuity_project_fiber} The isometries
$\mu_{bb'}:F_b\to F_{b'}$
depend continuously of
$b$, $b'\in B_\om$,
that is, for
$b_i\to b'$
and for any
$x\in F_b$,
we have
$\mu_{bb_i}(x)\to\mu_{bb'}(x)$.
\end{lem}

\begin{proof} If a sequence of geodesic segments in a
metric space pointwise converges, then the limit is
also a geodesic segment. Together with
uniqueness of $\R$-lines in
$X_\om$
and compactness of
$X$,
this implies the claim.
\end{proof}

We fix an orientation of
$F_b$
and define the orientation of
$F_{b'}$
via the isometry
$\mu_{bb'}$.
This gives a simultaneously
determined orientation
$O$
on all the fibers of
$\pi_\om$.

\begin{lem}\label{lem:fiber_orient} The orientation
$O$
is well defined and independent of the choice of
$b\in B_\om$.
\end{lem}

\begin{proof} By Lemma~\ref{lem:base_unique_geodesic}, the base
$B_\om$
is contractible. Using Lemma~\ref{lem:continuity_project_fiber},
we see that the orientation of
$F_{b''}$
induced by
$\mu_{bb''}$
coincides with that induced by
$\mu_{b'b''}\circ\mu_{bb'}$
for each
$b'$, $b''\in B$.
Hence, the claim.
\end{proof}

We assume that a simultaneous orientation
$O$
of
$\C$-lines
in
$X_\om$
is fixed, and we call it the {\em fiber orientation}.
Now we are able to produce nontrivial M\"obius
automorphisms of
$X$.
Using the fiber orientation
$O$
we define for every
$s\in\R$
the map 
$\ga=\ga_s:X_\om\to X_\om$
which acts on every fiber
$F$
on
$\pi_\om$
as the shift by
$|s|$
in the direction determined in the obvious way by the sign of
$s$
and
$O$.
The map 
$\ga$
is called a {\em vertical shift}.

\begin{pro}\label{pro:vshift_isometry} Every vertical shift
$\ga:X_\om\to X_\om$
is an isometry.
\end{pro}

\begin{proof} This immediately follows from definition
of a vertical shift, Lemma~\ref{lem:distance_clines}
and Proposition~\ref{pro:explicit_distance}.
\end{proof}

\subsection{Homotheties and shifts}

Recall that every M\"obius map 
$\ga:X\to X$
which fixes the point
$\om\in X$
is a homothety of
$X_\om$,
that is,
$|\ga(x)\ga(y)|_\om=\la|xy|_\om$
for some
$\la>0$
and every
$x$, $y\in X_\om$.

\begin{lem}\label{lem:involution_extended} For every
$\C$-circle $F\sub X$,
every M\"obius involution
$\ga:F\to F$
without fixed points extends to a M\"obius map
$\ov\ga:X\to X$.
\end{lem}

\begin{proof} We identify
$F=\di Y$,
where
$Y=\frac{1}{2}\hyp^2$
is the hyperbolic plane of constant curvature
$-4$.
For a fixed
$\om\in F$
any vertical shift
$\al:X_\om\to X_\om$
restricted to
$F$
is induced by a parabolic rotation 
$\wt\al:Y\to Y$,
which is an isometry without fixed points in
$Y$
having the unique fixed point
$\om\in\di Y$.
One easily sees that parabolic rotations generate the 
isometry group of
$Y$
preserving orientation.

The involution
$\ga$
is induced by a central symmetry of
$Y$.
Thus 
$\ga$
can be represented by a composition of vertical
shifts with appropriate fixed points on
$F$.
Hence
$\ga$
extends to a M\"obius map 
$\ov\ga:X\to X$.
\end{proof}

\begin{lem}\label{lem:homothety_existence} Given distinct
$o$, $\om\in X$, $\la>0$,
there is a homothety
$\ov\ga:X_\om\to X_\om$
with coefficient
$\la$, 
$\ov\ga(o)=o$,
preserving the fiber orientation
$O$.
\end{lem}

\begin{proof} Let
$F$
be the 
$\C$-circle
determined by
$o$, $\om$.
Using representation
$F=\di Y$
for 
$Y=\frac{1}{2}\hyp^2$
we write 
$\ga=\al\circ\be$
for any orientation preserving homothety
$\ga:F\to F$
with fixed
$o$, $\om$,
where
$\al$, $\be:F\to F$
are M\"obius involutions without fixed points.
By Lemma~\ref{lem:involution_extended},
$\ga$
extends to a M\"obius
$\ov\ga:X\to X$. 
\end{proof}

Assuming that
$\om\in X$
is fixed we denote
$B=B_\om$.
For an isometry
$\al:X_\om\to X_\om$
preserving the fiber orientation 
$O$,
we use notation
$\rot\al$
for the rotational part of the induced isometry
$\ov\al:B\to B$.
An isometry
$\ga:X_\om\to X_\om$
is called a {\em shift,} if
$\ga$
preserves 
$O$
and
$\rot\ga=\id$. 

\begin{lem}\label{lem:refined_shifts} For any
$x$, $x'\in X_\om$
there is a shift
$\eta=\eta_{xx'}:X_\om\to X_\om$
with
$\eta(x)=x'$.
\end{lem}

\begin{proof} Let
$\bO=\bO(B)$
be the isometry group of 
$B$
preserving a fixed point
$\ov o\in B$.
For
$A\in\bO$
we put
$|A|=\sup_v|vA(v)|$,
where the supremum is taken over all unit vectors
$v\in B$.
Note that
$\bO$
is compact and that
$|A|\le 2$
for every
$A\in\bO$.
We shall use the following well known
(and obvious) fact. For every
$\de>0$
there is
$K(\de)=K(B,\de)\in\N$
such that for any
$A\in\bO$
there is an integer
$m$, $1\le m\le K(\de)$,
with
$|A^m|<\de$.

By Lemma~\ref{lem:homothety_existence}, for every
$\ep>0$,
there are homotheties
$\phi$, $\psi:X_\om\to X_\om$
with coefficient
$\la=1/\ep$, $\phi(x)=x$, $\psi(x')=x'$.
Next, there is an integer
$m$, $1\le m\le K(\ep/K(\ep))$,
such that
$|\rot\phi^m|<\ep/K(\ep)$.
Applying the same argument to
$\psi^m$,
we find an integer
$r$, $1\le r\le K(\ep)$,
such that
$|\rot\psi^{mr}|<\ep$.
Then also
$|\rot\phi^{mr}|<\ep$,
and for
$\phi_\ep=\phi^{mr}$, $\psi_\ep=\psi^{mr}$
we have
$\phi_\ep(x)=x$, $|\rot\phi_\ep|<\ep$
and
$\psi_\ep(x')=x'$, $|\rot\psi_\ep|<\ep$.
For the isometry
$\eta_\ep=\psi_\ep^{-1}\circ\phi_\ep$,
we have
$|\rot\eta_\ep|<2\ep$, $|\eta_\ep(x)x'|=|xx'|/\la_\ep$,
where
$\la_\ep=\la^{mr}\to\infty$
as
$\ep\to 0$.
Thus
$\eta_\ep\to\eta$
as
$\ep\to 0$
with
$\eta(x)=x'$, $\rot\eta=\id$.
The isometry
$\eta$
preserves the fiber orientation
$O$
because the homotheties
$\phi$, $\psi$
do.
\end{proof}

\subsection{Lifting isometry}
\label{subsect:lifting_isometry}

We fix 
$\om\in X$
and use notations
$B=B_\om$, $\pi=\pi_\om:X_\om\to B$.
Let
$P\sub B$
be a {\em pointed} oriented parallelogram,
i.e., we assume that an orientation and a vertex
$o$
of
$P$
are fixed. We have a map
$\tau_P:F\to F$,
where
$F\sub X_\om$
is the fiber of
$\pi$
over
$o$, $F=\pi^{-1}(o)$.
Namely, given
$x\in F$, 
by axiom~($\rm E_\R$) there is a unique
$\R$-line
in
$X_\om$
through
$x$
that projects down by
$\pi$
to the first (according to the orientation) side of
$P$
containing
$o$.
In that way, we lift the sides of
$P$
to
$X_\om$
in the cyclic order according to the orientation and starting
with
$o$
which is initially lifted to
$x$.
Then
$\tau_P(x)\in F$
is the resulting lift of the parallelogram sides.

\begin{lem}\label{lem:triangle_lift} The map
$\tau_P:F\to F$
is an isometry that preserves orientation and, therefore,
it acts on
$F$
as a vertical shift.
\end{lem}

\begin{proof} The map
$\tau_P$
is obtained as a composition of four
$\C$-line
isometries of type
$\mu_{bb'}$,
see sect.~\ref{subsect:vertical_shifts}. Any isometry
$\mu_{bb'}$
preserves orientation,
see Lemma~\ref{lem:fiber_orient}. Thus
$\tau_P:F\to F$
is an isometry preserving orientation.
\end{proof}

There is a unique vertical shift
$\ga:X_\om\to X_\om$
with
$\ga|F=\tau_P$.
Furthermore, every shift
$\eta:X_\om\to X_\om$
commutes with
$\ga$,
thus the extension
$\ga$
of
$\tau_P$
coincides with that of
$\tau_{P'}$
for any 
$P'$
obtained from
$P$
by a shift of the base
$B$.
We use the same notation for the extension
$\tau_P:X_\om\to X_\om$
and call it a {\em lifting isometry}.

Since the group of vertical shifts is commutative, we have
$\tau_P\circ\tau_{P'}=\tau_{P'}\circ\tau_P$
for any (pointed oriented) parallelograms and even for any closed
oriented polygons
$P$, $P'\sub B$.

Let
$Q\sub B$
be a closed, oriented polygon. Adding a segment
$qq'\sub B$
between points
$q$, $q'\in Q$
we obtain closed, oriented polygons
$P$, $P'$
such that
$Q\cup qq'=P\cup P'$,
the orientations of
$P$, $P'$
coincide with that of
$Q$
along
$Q$,
and the segment
$qq'=P\cap P'$
receives from
$P$, $P'$
opposite orientations. In this case we use notation
$Q=P\cup P'$.

\begin{lem}\label{lem:adding_lifts} In the notation above we have
$\tau_Q=\tau_{P'}\circ\tau_P$.
\end{lem}

\begin{proof} We fix
$q\in Q\cap P\cap P'$
as the base point. Moving from
$q$
along
$Q$
in the direction prescribed by the orientation of
$Q$,
we also move along one of
$P$, $P'$
according to the induced orientation. We assume 
without loss of generality that this is the polygon
$P$.
In that way, we first lift
$P$
to
$X_\om$
starting with some point
$o\in F$,
where
$F$
is the fiber of the projection
$\pi:X_\om\to B$
over
$q$,
such that the side
$q'q\sub P$
is the last one while lifting
$P$.
Now, we lift
$P'$
to
$X_\om$
starting with
$o'=\tau_P(o)\in F$
moving first along the side
$qq'\sub P'$.
Then clearly the resulting lift of
$P'$
gives
$\tau_Q(o)=\tau_{P'}(o')\in F$.
Thus
$\tau_Q=\tau_{P'}\circ\tau_P$.
\end{proof}

For a vertical shift
$\ga:X_\om\to X_\om$
we denote
$|\ga|=|x\ga(x)|_\om$
the displacement of
$\ga$.
This is independent of
$x\in X_\om$.
If
$\ga$, $\ga'$
are vertical shifts in the same direction, then
$|\ga\circ\ga'|^2=|\ga|^2+|\ga'|^2$.

\begin{lem}\label{lem:equal_shifts_triangles} Given
a representation 
$P=T\cup T'$
of an oriented parallelogram
$P\sub B$
by oriented triangles
$T$, $T'$
with induced from
$P$
orientations, whose common side
$qq'$
is a diagonal of
$P$,
we have
$\tau_T=\tau_{T'}$.
\end{lem}

\begin{proof} For every
$n\in\N$
we consider the subdivision
$P=\cup_{a\in A}P_a$
of
$P$
into
$|A|=n^2$
congruent parallelograms
$P_a$
with sides parallel to those of
$P$.
We assume that the orientation of each
$P_a$
is induced by that of
$P$.
The parallelograms
$P_a$, $P_{a'}$
are obtained from each other by a shift of the base
$B$,
thus
$\tau_a=\tau_{a'}$
for each
$a$, $a'\in A$,
where
$\tau_a=\tau_{P_a}$.
 
By Lemma~\ref{lem:adding_lifts} we have
$\prod_{a\in A}\tau_a=\tau_P$,
hence
$|\tau_P|^2=n^2|\tau_a|^2$
for every 
$a\in A$.
Therefore
$|\tau_a|^2=|\tau_P|^2/n^2\to 0$
as
$n\to\infty$.
We subdivide the set 
$A$
into three disjoint subsets
$A=C\cup C'\cup D$,
where
$a\in D$
if and only if the interior of
$P_a$
intersects the diagonal
$qq'$
of
$P$,
and 
$a\in C$
if
$P_a\sub T$, $a\in C'$
if
$P_a\sub T'$.
Then
$|C|=|C'|=\frac{n(n-1)}{2}$, $|D|=n$.
We conclude that
$\prod_{a\in C}\tau_a=\prod_{a\in C'}\tau_a$
and
$$\left|\prod_{a\in D}\tau_a\right|^2=n|\tau_P|^2/n^2\to 0$$
as
$n\to\infty$.
By Lemma~\ref{lem:adding_lifts},
$\tau_P=\tau_T\circ\tau_{T'}$.
It follows that
$$\tau_T=\lim_n\prod_{a\in C}\tau_a=\lim_n\prod_{a\in C'}\tau_a=\tau_{T'}.$$
\end{proof}

The following Lemma will be used in sect.~\ref{subsect:COS_suspension},
see the proof of Proposition~\ref{pro:sphere_isometric_cos}.

\begin{lem}\label{lem:area_lift_triange} Let
$T=vyz\sub B$
be an oriented triangle. Then for the triangle
$P=xyz$
with
$x\in[vz]$
we have
$$|\tau_P|^2=\frac{|xz|}{|vz|}|\tau_T|^2.$$
\end{lem}

\begin{proof} Arguing as in Lemma~\ref{lem:adding_lifts} for
the oriented parallelogram
$Q=vyzw$ 
(for which 
$vz$
is a diagonal) we find that
$\tau_P=\tau_{P'}$
for 
$P=xyz$, $P'=vyx$
in the case
$x$
is the midpoint of
$[vz]$.
Therefore
$|\tau_P|^2=\frac{1}{2}|\tau_T|^2=\frac{|xz|}{|vz|}|\tau_T|^2$
in this case. Next we obtain by induction the required formula
for the case 
$|xz|/|vz|$
is a dyadic number. Then the general case follows by
continuity.
\end{proof}

\begin{lem}\label{lem:homothety_lift_parallelogram} Let
$P\sub B$
be an oriented parallelogram,
$\la P\sub B$
the parallelogram obtained from
$P$
by a homothety
$h:B\to B$
with coefficient
$\la>0$,
$h(o)=o$, $h(v)=\la v$.
Then
$|\tau_{\la P}|=\la|\tau_p|$.
\end{lem}

\begin{proof} We can assume that the parallelograms
$P$, $\la P$
have
$o$
as a common vertex, and that
$\la$
is rational. For a general
$\la$
one needs to use approximation. We choose
$n\in\N$
such that
$\la n\in\N$
and subdivide
$P$
into 
$n^2$
congruent parallelograms,
$P=\cup_{a\in A}P_a$, $|A|=n^2$,
as in Lemma~\ref{lem:equal_shifts_triangles}. This
also gives a subdivision of
$\la P$
into
$\la^2n^2$
parallelograms congruent to ones of the first subdivision,
$\la P=\cup_{a\in A'}P_a$, $|A'|=\la^2n^2$.
Then
$\tau_P=\prod_{a\in A}\tau_a$,
$|\tau_a|^2=|\tau_P|^2/n^2$
and
$$|\tau_{\la P}|^2=\la^2n^2|\tau_a|^2=\la^2|\tau_P|^2.$$
\end{proof}

\subsection{Reflections with respect to $\C$-circles}
\label{subsect:reflection_ccircle}

In sect.~\ref{sect:ccircle_involutions} we have defined 
for every
$\C$-circle $F\sub X$
the reflection
$\phi_F:X\to X$
whose fixed point set is 
$F$
and
$v=\phi_F(u)$
is conjugate to
$u$
pole of
$F$
for every
$u\in X\sm F$.

\begin{pro}\label{pro:moebius_ccircle_involution} For every
$\C$-circle $F\sub X$
the reflection
$\phi_F:X\to X$
is M\"obius.
\end{pro}

\begin{proof}
We fix 
$\om\in F$
and consider
$F$
as a
$\C$-line
in 
$X_\om$.
By definition,
$\phi_F:X_\om\to X_\om$
preserves every
$\R$-line $\si\sub X_\om$
intersecting
$F$
and acts on
$\si$
as the reflection with respect to
$\si\cap F$,
in particular,
$\phi_F|\si$
is isometric. It follows from Lemma~\ref{lem:mu_isometric}
that
$\phi_F$
is isometric on every
$\C$-line
in
$X_\om$,
thus
$\phi_F$
induces the central symmetry
$\ov\phi=\ov\phi_F:B_\om\to B_\om$
with respect to
$\ov o=\pi_\om(F)\in B_\om$.

Given
$x$, $x'\in X_\om$
we show that
$|yy'|=|xx'|$
for
$y=\phi_F(x)$, $y'=\phi_F(x')$.
Let
$F_x$, $F_y\sub X_\om$
be the 
$\C$-lines
through
$x$, $y$
respectively. We denote
$z=\mu_{F_x,\om}(x')\in F_x$, $z'=\mu_{F_y,\om}(y')\in F_y$.
By Proposition~\ref{pro:explicit_distance} we have
$$|xx'|^4=|x'z|^4+|zx|^4\quad\text{and}\quad |yy'|^4=|y'z'|^4+|z'y|^4.$$
Furthermore,
$|x'z|=|y'z'|$
because
$|x'z|=|\ov x'\ov z|$, $|y'z'|=|\ov y'\ov z'|$
and
$\ov\phi(\ov x')=\ov y'$, $\ov\phi(\ov z)=\ov z'$,
where ``bar'' means the projection by
$\pi_\om$.

The triangles
$\ov o\,\ov x'\ov z$
and
$\ov o\,\ov y'\ov z'$
in
$B$
are symmetric to each other by
$\ov\phi$
and have the same orientation. It follows from
Lemma~\ref{lem:equal_shifts_triangles} that
any lift of the closed polygon
$P=\ov x'\ov y'\ov z'\ov z\,\ov x'\sub B$
closes up in
$X_\om$,
that is,
$\tau_P=\id$.
Hence
$\mu_{F_y,\om}(z)=z'$
and therefore
$|xz|=|yz'|$.
We conclude that
$|yy'|=|xx'|$.
\end{proof}

\subsection{Pure homotheties}
\label{subsect:pure_homothety}

For an
$\R$-line $\si\sub X_\om$
we define the {\em\semi-plane}
$R=R_\si\sub X_\om$
as
$R=\pi_\om^{-1}(\pi_\om(\si))$.
Note that
$R$
has two foliations: one by
$\C$-lines
and another by
$\R$-lines.

Given a
$\C$-line
$F\sub X_\om$
and an 
$\R$-line $\si\sub X_\om$
with
$o=F\cap\si$,
let 
$R\sub X_\om$
be the \semi-plane {\em spanned} by
$F$
and
$\si$.
Every point
$x\in R$
is uniquely determined by its projections the {\em vertical} to
$F$, $x_F=\mu_{F,\om}(x)$,
and the {\em horizontal} to
$\si$, $x_\si=\si\cap F_x$,
where
$F_x$
is the 
$\C$-line
through
$x$.

For
$o\in X_\om$
and
$\la>0$
we define a map 
$h=h_{o,\la}:X_\om\to X_\om$
as follows. We put
$h(o)=o$
and require that 
$h$
preserves the 
$\C$-line $F$
through
$o$
and every
$\R$-line $\si$
through
$o$
acting on
$F$
and 
$\si$
as the homotheties with coefficient
$\la$.
Finally,
$h$
preserves every \semi-plane
$R$
containing  
$F$
and acts on
$R$
by
$h(x_F,x_\si)=(h(x_F),h(x_\si))$,
where
$R$
is spanned by
$F$
and the 
$\R$-line $\si$
through
$o$.

\begin{pro}\label{pro:pure_homothety} The map 
$h:X_\om\to X_\om$
defined above is a homothety with coefficient
$\la$, $|h(x)h(y)|=\la|xy|$
for every
$x$, $y\in X_\om$,
in particular,
$h$
is M\"obius.
\end{pro}

\begin{proof} It follows from definition and 
Proposition~\ref{pro:explicit_distance} that the restriction
$h|R$
is the required homothety for every \semi-plane
$R\sub X_\om$
containing the 
$\C$-line $F$
through
$o$.
Thus the induced map 
$\ov h:B_\om\to B_\om$
is the homothety with coefficient
$\la$.
 
Let
$F_x$
be the
$\C$-line
through
$x$, $\si$
the 
$\R$-line
through
$o$
that intersect 
$F_x$.
Similarly, let 
$F_y$
be the
$\C$-line
through
$y$, $\ga$
the 
$\R$-line
through
$o$
that intersect 
$F_y$.
We denote 
$z=\mu_{F_x,\om}(y)\in F_x$
the projection of
$y$
to
$F_x$.
Then by Proposition~\ref{pro:explicit_distance} we have
$|xy|^4=|xz|^4+|zy|^4$.

For the closed polygon
$P=\ov o\,\ov y\,\ov z\,\ov o\sub B_\om$,
where ``bar'' means the projection by
$\pi_\om$,
we have
$\tau_P(y_F)=z_F$,
where
$y=(y_F,y_\ga)$, $z=(z_F,z_\si)$
are vertical and horizontal coordinates in respective
\semi-planes.

Applying the map 
$h$
we obtain
$x'=h(x)$, $y'=h(y)$, $z'=h(z)$, $y_F'=h(y_F)$, $z_F'=h(z_F)$.
Furthermore, the polygon
$\la P=\ov h(P)=\ov o\,\ov y'\ov z'\ov o$
is obtained from
$P$
by the homothety
$\ov h:B_\om\to B_\om$, $\ov h(\ov o)=\ov o$,
in particular
$|\ov y'\ov z'|=\la|\ov y\,\ov z|=\la|yz|$.
Using the fact that 
$h\circ\tau_a=\tau_{\la a}\circ h$
for any vertical shift
$\tau_a:F\to F$
with
$|\tau_a|=a\ge 0$
and Lemmas~\ref{lem:equal_shifts_triangles},
\ref{lem:homothety_lift_parallelogram} we obtain
$\tau_{\la P}(y_F')=z_F'$.
It means that 
$z'=\mu_{F_{x'},\om}(y')$.
Therefore,
$|y'z'|=\la|yz|$
and
$|x'y'|^4=|x'z'|^4+|z'y'|^4$
again by Proposition~\ref{pro:explicit_distance}.
Since
$|x'z'|=\la|xz|$,
we obtain
$|x'y'|=\la|xy|$.
\end{proof}

By definition, the homothety
$h:X_\om\to X_\om$, $h(o)=o$,
preserves every
$\R$-line $\si\sub X_\om$
through
$o$.
Every homothety with this property is said to be
{\em pure}.

\section{Orthogonal complements to a $\C$-circle}
\label{sect:fiber_filling_map}

\subsection{Definition and properties}

Let
$F\sub X$
be a
$\C$-circle.
Every
$u\in X\sm F$
determines an involution
$\eta_u:F\to F$
without fixed points, which is M\"obius by Corollary~\ref{cor:phi_u_moebius}.
In other words, we have a map 
$\cF:X\sm F\to J$, $u\mapsto\eta_u$,
where
$J=J_F$
is the set of M\"obius involutions 
of
$F$
without fixed points. We study fibers of this map, 
$\cF^{-1}(\eta)=:(F,\eta)^\perp$.
The set 
$(F,\eta)^\perp=\set{u\in X\sm F}{$\eta_u=\eta$}$
is called the {\em orthogonal complement} to
$F$
at
$\eta$. 

Given distinct
$x$, $y\in F$
let
$S_{x,y}\sub X$
be the set covered by all
$\R$-circles
in 
$X$
through
$x$, $y$.
By Lemma~\ref{lem:filling_sphere_rcirles},
this set can be described as the sphere in
$X$
between
$o$, $\om\in F$
such that
$(x,o,y,\om)\in\harm_F$.
Then for every
$u\in S_{x,y}\sm\{x,y\}$
we have
$\eta_u(x)=y$.
Thus
\begin{equation}\label{eq:filling_fiber_represent}
(F,\eta)^\perp=\bigcap_{x\in F}S_{x,\eta(x)} 
\end{equation}

\begin{lem}\label{lem:filling_fiber} Given a
$\C$-circle $F\sub X$
and a M\"obius involution
$\eta:F\to F$
without fixed points, the orthogonal complement 
$A=(F,\eta)^\perp$
to
$F$
at
$\eta$
can be represented as
$$A=S_{x,y}\cap S_{o,\om}$$
for any
$x$, $o\in F$, $o\neq x,y=\eta(x)$,
where
$\om=\eta(o)$.
\end{lem}

\begin{proof} We have
$A\sub S_{x,y}\cap S_{o,\om}$
by Eq.~(\ref{eq:filling_fiber_represent}).
On the other hand, for every
$u\in S_{x,y}\cap S_{o,\om}$
we have
$\eta_u=\eta$
along the 4-tuple
$(x,o,y,\om)$.
Thus
$\eta_u=\eta$
because any M\"obius
$\eta:F\to F$
is uniquely determined by values 
at three distinct points. Hence
$u\in A$. 
\end{proof}

Note that the set 
$(F,\eta)^\perp$
contains with every
$u\in (F,\eta)^\perp$
the conjugate pole
$v=\phi_F(u)$
of
$F$.
Thus
$\phi_F:(F,\eta)^\perp\to (F,\eta)^\perp$
is an involution without fixed points. 
By Proposition~\ref{pro:moebius_ccircle_involution} 
this involution is M\"obius.

\begin{lem}\label{lem:orthogonal_covering_ccircles} For each pair
$u,v\in (F,\eta)^\perp$
of conjugate poles of
$F$
the
$\C$-circle $F'$
through
$u$, $v$
is contained in
$(F,\eta)^\perp$.
\end{lem}

\begin{proof} We take distinct
$x,o\in F$, $o\neq x,y$,
such that
$(x,o,y,\om)\in\harm_F$,
where
$y=\eta_u(x)$, $\om=\eta_u(o)$.
There are 
$\R$-circles $\si,\ga\sub X$
such that
$(x,u,y,v)\in\harm_\si$, $(o,u,\om,v)\in\harm_\ga$.
By axiom 
$\rm O_\C$
the
$\C$-circle $F'$
through
$u,v$
is contained in 
$S\cap S'$,
where
$S$, $S'\sub X$
are spheres between
$x,y$
and
$o,\om$
respectively that contain
$u,v$.
Since
$(x,o,y,\om)\in\harm_F$, 
by axiom~$\rm O_\R$
we have
$\ga\sub \wt S$
and
$\si\sub \wt S'$,
where
$\wt S$, $\wt S'$
are spheres between
$x,y$
through
$o,\om$
and between
$o,\om$
through
$x,y$
respectively. Hence
$\wt S=S$, $\wt S'=S'$.
We conclude that
$S=S_{o,\om}$, $S'=S_{x,y}$. 
Then
$S\cap S'=(F,\eta)^\perp$
by Lemma~\ref{lem:filling_fiber}, hence
$F'\sub (F,\eta)^\perp$.
\end{proof}

\begin{lem}\label{lem:orthogonal_foliated_ccircles} Given a
$\C$-circle $F\sub X$
and a M\"obius involution
$\eta:F\to F$
without fixed points, the orthogonal complement
$A=(F,\eta)^\perp$
is foliated by
$\C$-circles $F'$
through pairs
$u,v\in A$
of conjugate poles of
$F$,
and
$\phi_F:A\to A$
preserves every fiber of this fibration.
\end{lem}

\begin{proof} By Lemma~\ref{lem:orthogonal_covering_ccircles},
$A$
is covered by 
$\C$-circles $F'$.
We have
$\phi_F(F')=F'$
because
$\phi_F$
permutes conjugate poles of
$F$.
Thus by axiom~$\rm E_\C$
distinct 
$\C$-circles
of the covering are disjoint, that is,
the 
$\C$-circles $F'$
form a fibration of
$A$.
\end{proof}

The fibration of
$(F,\eta)^\perp$
by
$\C$-circles
described in Lemma~\ref{lem:orthogonal_foliated_ccircles}
is called {\em canonical}. We do not claim that every
$\C$-circle
in
$(F,\eta)^\perp$
is a fiber of the canonical fibration, this is actually not true
in general.

\begin{lem}\label{lem:fiber_involution} Given a M\"obius involution
$\eta:F\to F$
without fixed points of a 
$\C$-circle $F$,
for every fiber
$F'$
of the canonical fibration of
$A=(F,\eta)^\perp$
the reflection
$\phi'=\phi_{F'}:X\to X$
preserves
$A$
and its canonical fibration.
\end{lem}

\begin{proof} For every
$u\in A$
and
$x\in F$
there an
$\R$-circle $\si\sub X$
such that
$(x,u,y,v)\in\harm_\si$,
where
$y=\eta(x)$, $v=\phi_F(u)$.
Taking
$u\in F'$
we observe that by definition
$\phi'$
permutes
$x$, $y$,
thus
$\phi'$
preserves
$F$
and
$\phi'|F=\eta$.
Since
$\phi'$
is M\"obius, we have
$$\phi'(A)=(\phi'(F),\phi'\circ\eta\circ\phi'^{-1})^\perp
  =(F,\eta)^\perp=A.$$
For an arbitrary 
$u\in A$
we have
$\phi'(x,u,y,v)=(y,u',x,v')\in\harm_{\si'}$,
where
$\si'=\phi'(\si)$,
and
$u',v'\in A$
are conjugate poles of
$F$.
Thus
$\phi'$
moves the fiber of the canonical fibration through
$u,v$
to the fiber through
$u',v'$.
\end{proof}

\subsection{Mutually orthogonal $\C$-circles}
\label{subsect: mutually_orthogonal_ccircles}

We say that distinct
$\C$-circles $F$, $F'\sub X$
are {\em mutually orthogonal} to each other, 
$F\perp F'$,
if
$\phi_F(F')=F'$
and
$\phi_{F'}(F)=F$.
Note that then
$F$, $F'$
are disjoint because by Lemma~\ref{lem:intersect_rcircle_ccircle} a
$\C$-circle
and an
$\R$-circle 
have in common at most two points. If
$\phi_F(F')=F'$,
then
$\phi_{F'}(F)=F$
automatically by definition of
$\phi_F$.
If
$\eta=\phi_{F'}|F$
and
$\eta'=\phi_F|F'$,
then
$F\sub (F',\eta')^{\perp}$
and
$F'\sub (F,\eta)^{\perp}$.
We also note that if
$F\perp F'$,
then
$\phi_F$
acts on
$F'$
as a M\"obius involution without fixed points.

\begin{lem}\label{lem:involutions_commute} Assume a
$\C$-circle $F\sub X$
is invariant under a M\"obius
$\phi':X\to X$.
Then
$\phi=\phi_F$
commutes with
$\phi'$, $\phi\circ\phi'=\phi'\circ\phi$.
In particular,
$\phi_F\circ\phi_{F'}=\phi_{F'}\circ\phi_F$
for mutually orthogonal 
$\C$-circles $F$, $F'$.
\end{lem}

\begin{proof} Since
$F$
is the fixed point set for 
$\phi$, 
we have
$\phi\circ\phi'=\phi'\circ\phi$
along
$F$.
Thus to prove the equality
$\phi\circ\phi'(x)=\phi'\circ\phi(x)$
for an arbitrary
$x\in X$,
we can assume that
$x\not\in F$.

For an arbitrary
$\om\in F$
there is a (uniquely determined)
$\R$-circle
$\si\sub X$
such that
$(x,z,y,\om)\in\harm_\si$,
where
$y=\phi(x)$, $z=\eta_x(\om)\in F_\om\cap\si$.
Then for the 
$\R$-circle $\si'=\phi'(\si)$
we have
$(x',z',y',\om')=\phi'(x,z,y,\om)\in\harm_{\si'}$,
where
$\om',z'=\eta_{x'}(\om')\in F$.
This means that
$y'=\phi(x')$,
that is,
$\phi'\circ\phi(x)=\phi\circ\phi'(x)$.
\end{proof}

A collection
$\set{F_\la}{$\la\in\La$}$
of
$\C$-circles
is said to be {\em orthogonal}, if
$F_\la\perp F_{\la'}$
for each distinct
$\la$, $\la'\in\La$.
We denote by
$\phi_\la=\phi_{F_\la}$
the respective M\"obius involutions. By Lemma~\ref{lem:involutions_commute},
$\phi_\la\circ\phi_{\la'}=\phi_{\la'}\circ\phi_\la$
for each
$\la$, $\la'\in\La$.

\begin{lem}\label{lem:twoofthree_coincide} Let
$\set{F_\la}{$\la\in\La$}$
be an orthogonal collection of
$\C$-circles. Then for distinct
$\la,\la',\la''\in\La$
we have
$\phi_\la|F_{\la''}=\phi_{\la'}|F_{\la''}$,
in particular,
$F_{\la''}$
lies in the fixed point set of
$\phi_\la\circ\phi_{\la'}$.
\end{lem}

\begin{proof} By definition, 
$\phi_\la$, $\phi_{\la'}$
preserve
$F_{\la''}$
and therefore act on
$F_{\la''}$
as M\"obius involutions without fixed points. Since they commute, the composition
$\phi_\la\circ\phi_{\la'}$
has a finite order. It follows that
$\phi_\la|F_{\la''}=\phi_{\la'}|F_{\la''}$ 
because otherwise
$\phi_\la\circ\phi_{\la'}$
would be of infinite order.
\end{proof}

\subsection{Intersection of orthogonal complements}

\begin{pro}\label{pro:empty_3D} Assume
$$(F,\eta)^\perp\cap(F',\eta')^\perp=\es$$
for mutually orthogonal
$\C$-circles $F,F'\sub X$,
where
$\eta=\phi_{F'}|F$, $\eta'=\phi_F|F'$.
Then
$\dim X=3$.
\end{pro}

\begin{proof} We let
$\dim X=k+1$
with 
$k\ge 0$.
Then
$\dim A=\dim A'=k-1$
for 
$A=(F,\eta)^\perp$, $A'=(F',\eta')^\perp$.
By Lemma~\ref{lem:orthogonal_foliated_ccircles},
$A$
is foliated by
$\C$-circles,
thus
$k$
is even. If 
$k=0$,
then
$X=\di M$
for 
$M=\C\hyp^1$.
We can assume that
$k\ge 2$.
Note that the codimension of
$A$
equals two and that
$F$
is not contractible in
$X\sm A$.
Thus the assumption
$A\cap A'=\es$
implies by transversality argument that
$2(k-1)<k+1$.
Hence
$k=2$
and
$\dim X=3$.
\end{proof}

\begin{pro}\label{pro:intersection_orthogonal_complements} Let
$F,F'\sub X$
be mutually orthogonal
$\C$-circles, $F\perp F'$,
$A=(F,\eta)^\perp$, $A'=(F',\eta')^\perp$,
where
$\eta=\phi'|F$, $\eta'=\phi|F'$
and
$\phi=\phi_F$, $\phi'=\phi_{F'}$.
Then for every
$u\in A\cap A'$
we have
$\phi(u)=\phi'(u)$.
\end{pro}

\begin{proof} We denote
$v=\phi(u)$, $v'=\phi'(u)$.
We fix
$o\in F$, $x\in F'$
and put
$\om=\phi'(o)\in F$, $y=\phi(x)\in F'$.
Since
$F'\sub A$,
there is an
$\R$-circle $\si\sub X$
such that
$(x,o,y,\om)\in\harm_\si$.
Since
$u\in A$,
there is an
$\R$-circle $\ga\sub X$
such that
$(u,o,v,\om)\in\harm_\ga$.

\begin{lem}\label{lem:F_symm} In the metric of
$X_\om$
we have
$|vy|_\om=|ux|_\om$, $|vx|_\om=|yu|_\om$, $|ox|_\om=|ou|_\om=|ov|_\om$. 
\end{lem}

\begin{proof} The 
$\C$-circle $F$
is the fixed point set of the M\"obius
$\phi:X\to X$.
Thus
$\phi$
acts on
$X_\om$
as an isometry. Using that
$\phi(u)=v$
and
$\phi(x)=y$,
we obtain
$|vy|_\om=|ux|_\om$, $|vx|_\om=|yu|_\om$.
Recall that
$u,v,x,y\in A$
and 
$A$
lies in a sphere between
$o$, $\om$.
Thus
$|ox|_\om=|ou|_\om=|ov|_\om$.
\end{proof}

Since
$u\in A'$,
there is an
$\R$-circle $\ga'\sub X$
such that
$(u,x,v',y)\in\harm_{\ga'}$.

\begin{lem}\label{lem:Fprim_symm} In the metric of
$X_x$
we have
$|v'\om|_x=|ou|_x$, $|v'o|_x=|u\om|_x$, $|v'y|_x=|y\om|_x=|uy|_x$. 
\end{lem}

\begin{proof} The 
$\C$-circle $F'$
is the fixed point set of the M\"obius
$\phi':X\to X$.
Thus
$\phi'$
acts on
$X_x$
as an isometry. Using that
$\phi'(u)=v'$
and
$\phi'(o)=\om$,
we obtain
$|v'\om|_x=|ou|_x$, $|v'o|_x=|u\om|_x$.
Recall that
$u,v',o,\om\in A'$
and 
$A'$
lies in a sphere between
$x$, $y$.
Thus
$|v'y|_x=|y\om|_x=|uy|_x$.
\end{proof}

Using Lemma~\ref{lem:Fprim_symm}, we obtain
$$|v'y|_\om=\frac{|v'y|_x}{|v'\om|_x\cdot|y\om|_x}=\frac{1}{|v'\om|_x}=|v'x|_\om,$$
and using Lemma~\ref{lem:F_symm}, we obtain
$$|v'x|_\om=\frac{1}{|v'\om|_x}=\frac{1}{|ou|_x}
 =\frac{|ox|_\om\cdot|ux|_\om}{|ou|_\om}=|ux|_\om.$$
That is,
$|v'y|_\om=|v'x|_\om=|ux|_\om$.

Next we show that
$|ov'|_\om=|ov|_\om$.
Using Lemma~\ref{lem:Fprim_symm}, we obtain
$$|ov'|_\om=\frac{|ov'|_x}{|o\om|_x\cdot|v'\om|_x}
  =\frac{|u\om|_x}{|o\om|_x\cdot|ou|_x}
  =\frac{|ox|_\om}{|ux|_\om}\cdot\frac{|ox|_\om\cdot|ux|_\om}{|ou|_\om}
  =\frac{|ox|_\om^2}{|ou|_\om}.$$
By Lemma~\ref{lem:F_symm},
$|ox|_\om=|ou|_\om=|ov|_\om$,
hence
$|ov'|_\om=|ov|_\om$.

Since
$(u,x,v',y)\in\harm_{\ga'}$,
we have
$$|uv'|_\om\cdot|xy|_\om=|ux|_\om\cdot|v'y|_\om+|uy|_\om\cdot|xv'|_\om
  =2|ux|_\om\cdot|v'y|_\om=2|ux|_\om^2.$$
Applying the Ptolemy inequality to the 4-tuple
$(u,x,v,y)$,
we obtain
$$|uv|_\om\cdot|xy|_\om\le |ux|_\om\cdot|vy|_\om+|vx|_\om\cdot|uy|_\om
  =|ux|_\om^2+|uy|_\om^2.$$
By Lemma~\ref{lem:Fprim_symm},
$$|uy|_\om=\frac{|uy|_x}{|u\om|_x\cdot|y\om|_x}=\frac{1}{|u\om|_x}=|ux|_\om.$$
Thus
$$|uv|_\om\cdot|xy|_\om\le 2|ux|_\om^2=|uv'|_\om\cdot|xy|_\om.$$
We conclude that
$|uv|_\om\le|uv'|_\om$.

On the other hand,
$|uv|_\om=|uo|_\om+|ov|_\om$
and
$|uv'|_\om\le|uo|_\om+|ov'|_\om=|uo|_\om+|ov|_\om=|uv|_\om$.
Hence
$|uv'|_\om=|uo|_\om+|ov'|_\om$.
Therefore, the 4-tuple
$(u,o,v',\om)$
satisfies the Ptolemy equality
$$|uv'|\cdot|o\om|=|uo|\cdot|v'\om|+|u\om|\cdot|ov'|$$
and three of its entries
$u,o,\om$
lie on the 
$\R$-circle $\ga$.
By Proposition~\ref{pro:property_u},
$v'\in\ga$
and hence
$v'=v$.
\end{proof}

\begin{pro}\label{pro:intersection_foliation} Let
$F,F'\sub X$
be mutually orthogonal
$\C$-circles, $F\perp F'$,
$A=(F,\eta)^\perp$, $A'=(F',\eta')^\perp$,
where
$\eta=\phi'|F$, $\eta'=\phi|F'$
and
$\phi=\phi_F$, $\phi'=\phi_{F'}$.
Then the intersection
$A\cap A'$
is the fixed point set of 
$\psi=\phi'\circ\phi$,
and it carries a fibration by
$\C$-circles,
which coincides with the restriction of the canonical
fibrations of
$A$, $A'$
to
$A\cap A'$.
\end{pro}

\begin{proof} Assume
$u\in A\cap A'$.
There is a 
$\C$-circle $G\sub A$
of the canonical fibration of
$A$
passing through
$u$
and
$\phi(u)$,
and there is a 
$\C$-circle $G'\sub A'$
of the canonical fibration of
$A'$
passing through
$u$
and
$\phi'(u)$.
By Proposition~\ref{pro:intersection_orthogonal_complements},
$\phi(u)=\phi'(u)$.
Hence
$G=G'\sub A\cap A'$,
i.e.
$A\cap A'$
carries a fibration by
$\C$-circles,
which coincides with the restrictions of the canonical
fibrations of
$A$, $A'$
to
$A\cap A'$.
Furthermore, since
$\phi$, $\phi'$
are involutions, we have
$\psi(u)=\phi'\circ\phi(u)=u$,
i.e.
$A\cap A'$
lies in the fixed point set 
$\fix\psi$
of
$\psi$.

Assume
$\psi(u)=u$
for some 
$u\in X$.
Then
$u\not\in F\cup F'$
and
$v=\phi(u)=\phi'(u)\neq u$.
Thus there is a unique
$\C$-circle $F''$
through
$u$, $v$.
Then
$\phi(F'')=F''=\phi'(F'')$,
hence
$F,F',F''$
is an orthogonal collection of
$\C$-circles.
By Lemma~\ref{lem:twoofthree_coincide},
$\phi''=\phi_{F''}$
acts on
$F$ ($F'$)
as
$\phi'$ ($\phi$)
does, 
$\phi''|F=\phi'|F$, $\phi''|F'=\phi|F'$,
and
$F''\sub\fix(\phi'\circ\phi)=\fix\psi$.
Therefore
$F''\sub A\cap A'$
is a 
$\C$-circle 
of the canonical fibrations of
$A$, $A'$.
Thus
$A\cap A'=\fix\psi$. 
\end{proof}

\subsection{An induction argument}
\label{subsect:induced_argument}

\begin{lem}\label{lem:fixed_point_axioms} Assume a subset
$A\sub X$
foliated by
$\C$-circles
is the fixed point set of a M\"obius
$\psi:X\to X$.
Then
$A$
satisfies axioms (E), (O). 
\end{lem}

\begin{proof} We only need to check the existence axioms (E).
Given distinct
$a$, $a'\in A$,
there is a unique
$\C$-circle $F\sub X$
through
$a$, $a'$.
Since
$\psi(a)=a$, $\psi(a')=a'$,
we have
$\psi(F)=F$.
Then in the space
$X_a$,
the M\"obius
$\psi:X_a\to X_a$
acts as a homothety preserving
$a'$.
On the other hand,
$X_a$
is foliated by
$\C$-lines,
one of which,
$F'$,
lies by the assumption in
$\fix\psi$.
Hence,
$\psi:X_a\to X_a$
is an isometry pointwise preserving
$F'$
and
$a'$.
This excludes a possibility that
$\psi$
acts on
$F_a$
as the reflection at
$a'$.
Therefore,
$F\sub\fix\psi=A$,
which is axiom~($\rm E_\C$). 

Given a
$\C$-circle $F\sub A$, $\om\in F$
and
$u\in A\sm F$,
there is a unique 
$\R$-circle $\si\sub X$
through
$\om$, $u$
that hits
$F_\om$.
Thus at least three distinct points of
$\si$
lies in
$\fix\psi$.
We conclude that
$\psi$
pointwise preserves
$\si$,
i.e.
$\si\sub A$,
which is axiom~($\rm E_\R$). 
\end{proof}

\begin{cor}\label{cor:axioms_orthogonal_complements} Let
$F$, $F'\sub X$
be mutually orthogonal
$\C$-circles, $F\perp F'$,
$A=(F,\eta)^\perp$, $A'=(F',\eta')^\perp$,
where
$\eta=\phi'|F$, $\eta'=\phi|F'$
and
$\phi=\phi_F$, $\phi'=\phi_{F'}$.
Then the intersection
$A\cap A'$
satisfies axioms (E) and (O).
\end{cor}

\begin{proof} Apply Lemma~\ref{lem:fixed_point_axioms}
to
$A\cap A'$
which is by Proposition~\ref{pro:intersection_foliation}
the fixed point set of the M\"obius
$\psi=\phi'\circ\phi$
foliated by
$\C$-circles.
\end{proof}

\section{M\"obius join}
\label{sect:moebius_join}

\subsection{Canonical subspaces orthogonal to a $\C$-circle}

Let
$(F,\eta)$
be a 
$\C$-circle
in
$X$
with a M\"obius involution
$\eta:F\to F$
without fixed points.
Let
$F'\sub (F,\eta)^\bot$
be a nonempty subspace that satisfies
axioms (E), (O), and is invariant under the reflection
$\phi_F:X\to X$, $\phi_F(F')=F'$.
In this case we say that
$F'$
is a {\em canonical} subspace orthogonal to
$F$
at
$\eta$,
COS for brevity. Note that
$F'$
carries a canonical fibration by
$\C$-circles
induced by
$\phi_F$,
where the fiber through
$x\in F'$
is the uniquely determined
$\C$-circle
through
$x$, $\phi_F(x)$.
We use notation
$\cF=\cF_{F'}$
for this fibration, and
$H\in\cF$
means that
$H$
is a fiber of 
$\cF$.

\begin{lem}\label{lem:moebius_join_rcircle} For every
$u\in F$, $x\in F'$
there are uniquely determined
$v\in F$, $y\in F'$
such that 
$(u,x,v,y)\in\harm_\si$
for an
$\R$-circle $\si\sub X$.
\end{lem}

\begin{proof} By definition of
$(F,\eta)^\bot$, 
there is a uniquely determined 
$\R$-circle $\si\sub X$
through
$u$, $x$
that hits
$F$
at
$v=\eta(u)$.
Then there is a unique
$y=\phi_F(x)\in\si$
such that
$(u,x,v,y)\in\harm_\si$.
Again,
$y\in F'$
by definition of
$F'$.
\end{proof}

We define the {\em (M\"obius) join}
$F\ast F'$
as the union of
$\R$-circles $\si\sub X$
such that
$(u,x,v,y)\in\harm_\si$
with
$\{u,v\}=\si\cap F$, $\{x,y\}=\si\cap F'$,
where
$v=\eta(u)$, $y=\phi_F(x)$.
Every such a circle is called a {\em standard}
$\R$-circle in
$F\ast F'$.

\begin{lem}\label{lem:moebius_join_intersection} For any
different standard
$\R$-circles $\si$, $\si'\sub F\ast F'$
we have
$\si\cap\si'\sub F\cup F'$. 
\end{lem}

\begin{proof} We take
$o\in\si\cap F$, $o'=\si'\cap F$,
and put
$\om=\eta(o)\in\si\cap F$, $\om'=\eta(o')\in\si'\cap F$.
Then by Axiom~($\rm O_\R$), for any
$u'$, $v'\in F$
such that
$(u',o',v',\om')\in\harm_F$
the 
$\R$-circle $\si'$
lies in a sphere
$S'\sub X$
between
$u'$, $v'$.

Assume
$o'\in\{o,\om\}$.
Without loss of generality,
$o'=o$.
Then
$\om'=\om$
and
$\si\cap\si'=\{o,\om\}$
since otherwise
$\si=\si'$
by Lemma~\ref{lem:three_points_rcircles}.
Thus we can assume that
$o'\neq o,\om$.

We take
$v'=\om$
and
$u'\in F$
such that
$(u',o',v',\om')\in\harm_F$.
In the space
$X_\om$, $\si$
is an
$\R$-line
through
$o$,
and
$F$
is a
$\C$-line
through
$o$,
while
$S'$
is the metric sphere centered at the midpoint
$u'\in F$
between
$o'$, $\om'$, $|u'o'|=|u'\om'|=:r$.
The distance function
$d_{u'}:\si\to\R$, $d_{u'}(x)=|u'x|$,
is convex and by Corollary~\ref{cor:bisector_cline}(i) 
it is symmetric with respect to
$o$.
Thus
$d_{u'}$
takes the value 
$r$
at most two times, that is,
$\si$
intersects
$S'$
at most two times (actually exactly two times because
the pair
$(o,\om)$
separates the pair
$(o',\om')$
on
$F$
by properties of
$\eta$).

On the other hand,
$\si$
intersects
$F'$
twice and
$F'\sub S'$
because 
$(F,\eta)^\bot\sub S'$.
Thus
$\si\cap\si'\sub F\cup F'$.
\end{proof}

\begin{pro}\label{pro:dim_moeb_join} Assume
$(F,\eta)$
is a 
$\C$-circle
in
$X$
with a free of fixed point M\"obius involution
$\eta:F\to F$, $F'\sub X$
a COS to
$F$
at
$\eta$.
If
$\dim X=\dim F'+2$,
then
$X=F\ast F'$.
\end{pro}

\begin{proof} We show that every point
$u\in X$
lies on a standard
$\R$-circle
in
$F\ast F'$.
Given
$\om\in F$,
we put
$o=\eta(\om)\in F$
and use notation
$|xy|=|xy|_\om$
for the distance between
$x$, $y$
in
$X_\om$.
Then
$F$
is a 
$\C$-line
in 
$X_\om$,
and every standard
$\R$-circle $\ga\sub F\ast F'$
through
$o$, $\om$
is an
$\R$-line 
in
$X_\om$.
Moreover, every
$\R$-line $\ga\sub X_\om$
through
$o$
intersects
$F'$
because
$\dim X=\dim F'+2$,
furthermore
$(o,u,\om,v)\in\harm_\ga$
for 
$\{u,v\}=\ga\cap F'$
and thus
$\ga\sub F\ast F'$
is a standard
$\R$-line.
That is, for every
$x\in F$
every
$\R$-circle $\ga\sub X$
through
$x$, $y=\eta(x)\in F$
is a standard one in
$F\ast F'$. 
Let
$w\in F$
be the midpoint between
$x$, $y$, $|xw|=|wy|=r$.
By Lemma~\ref{lem:filling_sphere_rcirles}
the sphere 
$S_r(w)\sub X_\om$
centered at
$w$
of radius
$r$
is covered by
$\R$-circles
through
$x$, $y$.
Thus it suffices to show that for every
$u\in X\sm F$
there is
$x\in F$
such that
$u\in S_r(w)$,
where
$w\in F$
is the midpoint between
$x$, $y=\eta(x)$, $r=|xw|=|wy|$.

Let
$z=\mu_{F,\om}(u)$
be the projection of
$u$
to
$F$
in
$X_\om$,
$a=|zu|$, $b=|zo|$.
Let
$\rho>0$
be the radius of the sphere in
$X_\om$
centered at
$o$
that contains
$F'$.
For
$x\in F$
the condition
$y=\eta(x)$
is equivalent to
\begin{equation}\label{eq:prod}
|xo|\cdot|oy|=\rho^2. 
\end{equation}

Assuming that
$z$
lies on
$F$
between
$o$
and
$x$
in
$X_\om$,
we also have
$|xz|^2=|xo|^2-b^2$,
$|yz|^2=|yo|^2+b^2$.
By Lemma~\ref{lem:wharm_basics},
$a^4=|xz|^2\cdot|yz|^2$,
which gives another equation for 
$|xo|$, $|yo|$,
\begin{equation}\label{eq:square_diff}
|xo|^2-|yo|^2=\frac{a^4+b^4-\rho^4}{b^2}. 
\end{equation}
The equations~(\ref{eq:prod}), (\ref{eq:square_diff})
have a positive solution
$(|xo|,|yo|)$,
which allows us to find 
$x$, $y=\eta(x)\in F$
and the midpoint
$w\in F$
between
$x$, $y$, 
that is,
$r:=|wx|=|wy|$
with
$|wu|=r$.
Indeed, we have
$|xz|^2+|yz|^2=2r^2$, $|zw|^2=r^2-|xz|^2$.
Using
$a^4=|xz|^2\cdot|yz|^2$
and Proposition~\ref{pro:explicit_distance},
we obtain
$$|wu|^4=|zw|^4+a^4=r^4+|xz|^2(|xz|^2+|yz|^2-2r^2)=r^4.$$
This shows that
$F\ast F'=X$.
\end{proof}

\section{M\"obius join equivalence}
\label{sect:moebius_join_equivalence}

\begin{thm}\label{thm:moebius_join_canonical} Let
$(F,\eta)$
be a
$\C$-circle in
$X$
with a fixed point free M\"obius involution
$\eta:F\to F$, $F'\sub X$
a canonical subspace orthogonal to
$F$
at
$\eta$.
Assume that
$F'$
is M\"obius equivalent to
$\di\C\hyp^k$, $k\ge 1$,
taken with the canonical M\"obius structure. Then the join
$F\ast F'$
is M\"obius equivalent to
$\di\C\hyp^{k+1}$.
\end{thm}

We start the proof with

\begin{lem}\label{lem:involution_equivariant_moebius_circles} Let
$F$, $G$
be M\"obius spaces which are equivalent to
$Y=\di\C\hyp^k$, $k\ge 1$, 
taken with the canonical M\"obius structure. Given
M\"obius involutions
$\eta:F\to F$, $\eta':G\to G$
without fixed points, there is a M\"obius equivalence
$g:F\to G$
that is equivariant with respect to
$\eta$, $\eta'$,
$$\eta'\circ g=g\circ\eta.$$
\end{lem}

\begin{proof} If we identify
$F$
with
$Y$,
then there is
$a\in M=\C\hyp^k$
such that the central symmetry
$s_a:M\to M$
induces the involution
$\eta$
on
$F$, $\di s_a=\eta$.
Similarly,
$\eta'=\di s_{a'}$
for 
$a'\in M$.
Thus any isometry
$f:M\to M$
with
$f(a)=a'$
induces a required M\"obius equivalence
$g=\di f$.
\end{proof}

\begin{lem}\label{lem:central_symmetry_decomposition} Let
$E\sub M=\C\hyp^{k+1}$
be a complex hyperbolic plane and
$a\in E$.
Let
$E'\sub M$
be the orthogonal complement to
$E$
at
$a$.
Denote by
$G=\di E$, $G'=\di E'\sub Y=\di M$.
Then
$\phi_G|G'=\psi|G'$,
where
$\psi:Y\to Y$
is the M\"obius involution without fixed points
induced by the central symmetry
$s_a:M\to M$
at
$a$, $\psi=\di s_a$.
\end{lem}

\begin{proof} We use the fact that every M\"obius
$Y\to Y$
is induced by an isometry
$M\to M$.
Recall that 
$G$
is the fixed point set of the reflection
$\phi_G:Y\to Y$.
Then
$E$
is the fixed point set for the isometry
$\zeta:M\to M$
with
$\di\zeta=\phi_G$,
and
$\zeta$
acts on
$E'$
as
$s_a$
does,
$\zeta|E'=s_a|E'$.
Hence
$\phi_G|G'=\psi|G'$.
\end{proof}

\subsection{Constructing a map between M\"obius joins}
\label{subsect:proof_moebius_join_canonical}

We fix a complex hyperbolic plane
$E\sub M=\C\hyp^{k+1}$
and
$a\in E$.
The orthogonal complement 
$E'\sub M$
to
$E$
at 
$a$
is isometric to
$\C\hyp^k$.
We denote by
$G=\di E$, $G'=\di E'$.

Recall that the boundary at infinity 
$Y=\di M$
taken with the canonical M\"obius structure
satisfies axioms (E), (O) (see Proposition~\ref{pro:axioms_model}).
Thus all the notions involved in the definition of
the M\"obius join
$G\ast G'$
are well defined for
$Y$.

The central symmetry
$s_a:M\to M$
induces the M\"obius involution
$\psi=\di s_a:Y\to Y$
without fixed points, for which
$G$, $G'$
are invariant,
$\psi(G)=G$, $\psi(G')=G'$.
Furthermore,
$G'\sub Y$
is a COS to
$G$
at
$\psi$
and
$\dim Y=\dim G'+2$.
Then by Proposition~\ref{pro:dim_moeb_join},
$Y=G\ast G'$.

By Lemma~\ref{lem:involution_equivariant_moebius_circles},
there is a M\"obius equivalence
$g:G\to F$,
which is equivariant with respect to
$\psi$
and
$\eta$, $g\circ\psi|G=\eta\circ g$.
Note that
$\phi_F:X\to X$
acts on
$F'$
as a M\"obius involution without fixed points.
By the assumption,
$F'\sub X$
is M\"obius equivalent to
$G'\sub Y$.
By Lemma~\ref{lem:involution_equivariant_moebius_circles} again,
there is a M\"obius equivalence
$g':G'\to F'$
which is equivariant with respect to 
$\psi$
and
$\psi'=\phi_F|F'$, $g'\circ\psi|G'=\psi'\circ g'$.
We define
$\psi':F\cup F'\to F\cup F'$
by
$\psi'|F=\eta$, $\psi'|F'=\phi_F|F'$,
and
$f:G\cup G'\to F\cup F'$
by
$f|G=g$, $f|G'=g'$.
Then 
$f$
is equivariant with respect to
$\psi$
and
$\psi'$, $\psi'\circ f=f\circ\psi$.
Furthermore,
$f$
maps the canonical fibration of
$G'$
by
$\C$-circles 
to that of
$F'$.

The intersection
$\si\cap(F\cup F')$
of every standard
$\R$-circle $\si\sub F\ast F'$
with
$F\cup F'$
is invariant under
$\psi'$.
Then
$\psi'$
uniquely extends to a M\"obius
$\si\to\si$,
for which we use the same notation
$\psi'$.
By Lemma~\ref{lem:moebius_join_intersection},
different standard
$\R$-circles
in
$F\ast F'$
may have common points only in
$F\cup F'$.
Thus we have a well defined involution
$\psi':F\ast F'\to F\ast F'$
without fixed points, which is M\"obius along
$F$, $F'$
and every standard
$\R$-circle 
in
$F\ast F'$.

For every standard
$\R$-circle $\si\sub G\ast G'$,
we have the map
$f:\si\cap(G\cup G')\to F\ast F'\sub X$,
which is equivariant with respect to 
$\psi$
and 
$\psi'$.
The map 
$f$
uniquely extends to a M\"obius
$f:\si\to F\ast F'$.
By Lemma~\ref{lem:moebius_join_intersection} different standard
$\R$-circles 
in
$G\ast G'$, $F\ast F'$
may have common points only in
$G\cup G'$, $F\cup F'$
respectively, thus this gives a well defined bijection
$f:G\ast G'\to F\ast F'$
which is M\"obius along
$G$, $G'$
and any standard
$\R$-circle
in
$Y$.
Moreover,
$f$
is equivariant with respect to
$\psi$, $\psi'$.

We show that
$f$
is M\"obius. We fix
$\om\in G$,
put
$o=\psi(\om)\in G$, $\om'=f(\om)$, $o'=f(o)\in F$
and consider
$Y_\om$, $X_{\om'}$
with metrics normalized so that
$f|G:G\to F$
is an isometry. 

It suffices to check that 
$f:Y_\om\to (F\ast F')_{\om'}$
is an isometry,
$|u'v'|_{\om'}=|uv|_\om$
for each
$u$, $v\in Y_\om$, $u'=f(u)$, $v'=f(v')\in (F\ast F')_{\om'}$.
For brevity, we use notation
$f_\om:=f:Y_\om\to (F\ast F')_{\om'}$
regarding
$f$
as a map between respective metric spaces.

During the proof we will also consider the maps
$f_u:Y_u\to X_{u'}$,
where
$u'=f(u)$.
We view
$f_u$
as a map between metric spaces, where on
$Y_u$
the metric is defined as the metric inversion (\ref{eq:metric_inversion})
of
$|\ |_\om$
and on
$X_{u'}$
the metric inversion of
$|\ |_{\om'}$.

Note that then
$f|G_u$
is isometric, if
$u\in G$.

\subsection{Isometricity along standard objects}

\begin{lem}\label{lem:standard_isometry_rlines} The map 
$f_\om$
is isometric on every standard
$\R$-line
in
$Y_\om$
through
$o$. 
\end{lem}

\begin{proof} Recall that
$Y$
is a compact Ptolemy space with axioms (E), (O). Thus
$G'$
lies in a sphere between
$o$, $\om$,
that is, in a metric sphere in
$Y_\om$
centered at
$o$,
while
$F'$
lies in a metric sphere in
$X_{\om'}$
centered at
$o'$,
see sect.~\ref{sect:fiber_filling_map}.
The respective radii are called the {\em radii}
of
$G'$, $F'$
in
$Y_\om$, $X_{\om'}$
respectively.

There is
$u\in G$
such that
$(u,o,v,\om)\in\harm_G$
for 
$v=\psi(u)$.
Any standard
$\R$-circle $\si\sub Y$
through
$u$, $v$
lies in a metric sphere in
$Y_\om$
centered at
$o$.
Then
$\si'=f(\si)$
lies in a metric sphere in
$X_{\om'}$
centered at
$o'=f(o)$.
Since
$f|G$
preserves distances, the metric spheres in
$Y_\om$, $X_{\om'}$
centered at
$o$, $o'$
containing
$\si$, $\si'$
respectively have equal radii. Hence the radius of
$G'$
in
$Y_\om$
is the same as the radius of
$F'$
in
$X_{\om'}$, $|o'x'|_{\om'}=|ox|_\om=r$
for every
$x\in G'$, $x'\in F'$.

Any standard
$\R$-circle $\si\sub Y$
through
$o$, $\om$
hits
$G'$
at
$x$, $y=\psi(x)$
and it is an
$\R$-line
in
$Y_\om$.
Thus
$f|\si:\si\to\si'$
is a homothety with respect to the metrics of
$Y_\om$, $X_{\om'}$.
Since
$|o'x'|_{\om'}=r=|ox|_\om$
for 
$x'=f(x)$,
we observe that
$f|\si$
is an isometry, that is,
$f_\om$
preserves distances along any standard
$\R$-line
through
$o$.
\end{proof}

\begin{lem}\label{lem:metric_inversion_line} Assume for 
$u\in G$
the map
$f_u$
is isometric along
$\R$-lines
(not necessarily standard) through some
$v\in G_u$
and preserves the distance between some
$x$, $y$
on those lines,
$|x'y'|_{u'}=|xy|_u$,
where
$x'=f(x)$, $y'=f(y)$, $u'=f(u)$.
Then
$|x'y'|_{\om'}=|xy|_\om$. 
\end{lem}

\begin{proof} We can assume that
$u\neq\om$
since otherwise there is nothing to prove. Let
$\si\sub Y_u$
be the  
$\R$-line
through
$v$, $x$.
Then
$\si'=f(\si)$
is an 
$\R$-line in
$(F\ast F')_{u'}$
through
$v'=f(v)\in F$.
Recall that,
$G$
is a
$\C$-line 
in
$Y_u$.
Then by Proposition~\ref{pro:explicit_distance} the distance
$|x\om|_u$
is uniquely determined by
$|vx|_u$, $|v\om|_u$
for every
$x\in\si$, $|x\om|_u^4=|vx|_u^4+|v\om|_u^4$.
Since 
$f_u$
is isometric along
$\si$
and along
$G_u$, 
we have
$|v'x'|_{u'}=|vx|_u$, $|v'\om'|_{u'}=|v\om|_u$.
Hence
$|x'\om'|_{u'}=|x\om|_u$,
and similarly
$|y'\om'|_{u'}=|y\om|_u$.
Since
$$|xy|_\om=\frac{|xy|_u}{|x\om|_u\cdot|y\om|_u}$$
and a similar formula holds for 
$|x'y'|_{\om'}$,
we obtain
$|x'y'|_{\om'}=|xy|_\om$.
\end{proof}

\begin{cor}\label{cor:standard_isometry_rcircles} The map 
$f_\om$
is isometric along every standard
$\R$-circle $\si\sub Y$.
\end{cor}

\begin{proof} We let
$\{u,v\}=\si\cap G$.
Then
$\si$
is an 
$\R$-line
in
$Y_u$,
hence by Lemma~\ref{lem:standard_isometry_rlines},
$f_u$
is isometric along
$\si$.
By Lemma~\ref{lem:metric_inversion_line},
$f_\om$
is isometric along 
$\si$.
\end{proof}

\begin{lem}\label{lem:standard_isometry_G} For every
$u\in G'$
the map 
$f_u$,
is isometric on
$G$. 
\end{lem}

\begin{proof} Given
$x$, $y\in G$,
there are standard
$\R$-circles $\si$, $\ga\sub Y$
through
$u$
with
$x\in\si$, $y\in\ga$.
By the assumption,
$|x'y'|_{\om'}=|xy|_\om$
for 
$x'=f(x)$, $y'=f(y)$.
By Corollary~\ref{cor:standard_isometry_rcircles},
$|x'u'|_{\om'}=|xu|_\om$
and
$|y'u'|_{\om'}=|yu|_\om$.
Applying the metric inversions with respect to
$u$, $u'$
we obtain
$|x'y'|_{u'}=|xy|_u$. 
\end{proof}

\begin{lem}\label{lem:standard_isometry_ccircle} The map 
$f_\om$
is isometric on
$G'$.
\end{lem}

\begin{proof} For
$x,y\in G'$
we put
$z=\psi(x)$, $w=\psi(y)\in G'$.
There are the standard
$\R$-circles $\si,\ga\sub Y$
such that
$(x,o,z,\om)\in\harm_\si$, $(y,o,w,\om)\in\harm_\ga$,
in particular,
$\si$, $\ga$
are 
$\R$-lines
in 
$Y_\om$.
Hence
$|xz|_\om=|yw|_\om=2r$,
where
$r>0$
is the radius of
$G'$
in 
$Y_\om$.
For
$(x',y',z',w')=f(x,y,z,w)\sub F'$
we have
$(x',o',z',\om')\in\harm_{\si'}$, $(y',o',w',\om')\in\harm_{\ga'}$
for the 
$\R$-circles $\si'=f(\si)$, $\ga'=f(\ga)$.
Hence
$|x'z'|_{\om'}=|y'w'|_{\om'}=2r$.
Since
$f|G':G'\to F'$
is M\"obius,
$\crt(x',y',z',w')=\crt(x,y,z,w)$,
and we conclude
$$|x'y'|_{\om'}\cdot|z'w'|_{\om'}=|xy|_\om\cdot|zw|_\om,\quad
  |x'w'|_{\om'}\cdot|y'z'|_{\om'}=|xw|_\om\cdot|yz|_\om.$$
On the other hand,
$z'=\psi'(x')$, $w'=\psi'(y')$
because
$f$
is equivariant with respect to
$\psi$, $\psi'$.
Thus
$z'=\phi_F(x')$, $w'=\phi_F(y')$.
Furthermore,
$z=\phi_G(x)$, $w=\phi_G(y)$
by Lemma~\ref{lem:central_symmetry_decomposition}. Recall that
$G$ ($F$)
is the fixed point set of
$\phi_G$ ($\phi_F$),
thus
$\phi_G$ ($\phi_F$)
is an isometry of
$Y_\om$ ($X_{\om'}$)
and hence
$|zw|_\om=|xy|_\om$, $|z'w'|_{\om'}=|x'y'|_{\om'}$.
Therefore,
$|x'y'|_{\om'}=|xy|_\om$
for any
$x$, $y\in G'$,
i.e.
$f_\om|G'$
is an isometry. 
\end{proof}

\subsection{$\R$- and COS-foliations of a suspension}
\label{subsect:rcfoliations}

Let
$H\sub Y$
be a COS to a 
$\C$-circle $K\sub Y$.
For every
$u\in K$, $x\in H$
there is a uniquely determined
$\R$-circle $\si\sub Y$
through
$u$, $x$
that meets
$K$
at another point
$v$.
We denote by
$S_uH=S_{u,v}H$
the set covered by
$\R$-circles
in
$Y$
through
$u$, $v$
which meet 
$H$.
Note that
$\si\cap H=\{x,y\}$, $y=\phi_K(x)$,
for every such an
$\R$-circle $\si$.
Topologically,
$S_uH$
is a {\em suspension} over 
$H$.
The points
$u$, $v$
are called the {\em poles} of
$S_uH$.
The set 
$S_uH\sm\{u,v\}$
is foliated by
$\R$-circles
through
$u,v$.
Slightly abusing terminology, this foliation
is called the 
$\R$-{\em foliation}
of
$S_uH$.
In the metric of
$Y_u$,
the circles of the 
$\R$-foliation
are
$\R$-lines
through
$v$,
and
$K$
is a
$\C$-line.

Let
$h=h_\la:Y_u\to Y_u$
be a pure homothety (see sect.~\ref{subsect:pure_homothety}) with coefficient
$\la>0$
centered at
$v$, $h(v)=v$.
The homothety
$h$
is induced by an isometry
$\ga:M\to M$,
which is a transvection along the geodesic
$uv\sub M$, $h=\di\ga$.
Then
$h$
preserves every
$\R$-line $\si\sub Y_u$
through
$v$,
acting on 
$\si$
as the homothety with coefficient
$\la$,
which preserves orientations of
$\si$.
The image
$H_\la=h(H)\sub Y_u$
is the boundary at infinity of
$\wt E_\la=\ga(\wt E)$,
where
$\wt E\sub M$
is the subspace with
$\di\wt E=H$,
and it is also a COS to 
$K$.
In that way, the COS's
$H_\la$
form a foliation of
$S_uH\sm\{u,v\}$,
which is called the 
{\em COS-foliation}
of
$S_uH$.
In the case
$H$
is a 
$\C$-circle,
we say about
$\C$-foliation
of
$S_uH$.

\begin{lem}\label{lem:cfoliation_isometric} Assume
$f_u$
is isometric along every
$R$-line $\si$
of the 
$\R$-foliation 
of
$S_uH$
and along
$H$.
Then
$f_u$
is isometric along every COS
$H_\la$
of the 
$\C$-foliation. 
\end{lem}

\begin{proof} Note that 
$f_u|\si$
is equivariant with respect to
$h$, $h'$,
where
$h'=h_\la':X_{u'}\to X_{u'}$
is the pure homothety (see sect.~\ref{subsect:pure_homothety})
with coefficient
$\la$
with the fixed point
$v'=f(v)$.
It follows that
$f|H_\la:H_\la\to f(H)_\la=h'\circ f(H)$
is an isometry with respect to the metrics of
$Y_u$, $X_{u'}$
because
$f|H_\la=h'\circ f\circ h^{-1}|H_\la$.
\end{proof}

\begin{lem}\label{lem:suspension_isometric} Let COS
$H$
to
$K$
be a 
$\C$-circle. 
Under the assumption
of Lemma~\ref{lem:cfoliation_isometric}, the map 
$f_u$
is isometric along the suspension
$S_uH$.
\end{lem}

\begin{proof} By the assumption,
$f_u$
is isometric along every 
$\R$-line
in
$S_uH$
through
$v$.
Thus taking
$x$, $y\in S_uH$
we can assume that neither
$x$
nor 
$y$
coincides with
$u$
or
$v$.
Then there are uniquely determined an
$\R$-line $\si\sub S_uH$
of the 
$\R$-foliation
through
$x$
and a
$\C$-circle
$H_\la\sub S_uH$
of the 
$\C$-foliation 
through
$y$.
The intersection
$\si\cap H_\la$
consists of two points, say,
$z$, $w$.
Taking the metric inversion with respect to
$w$
we obtain
$$|xy|_w=\frac{|xy|_u}{|xw|_u\cdot|yw|_u}$$
and similar expressions for 
$|xz|_w$, $|yz|_w$.
In the space
$X$
we have respectively
$$|x'y'|_{w'}=\frac{|x'y'|_{u'}}{|x'w'|_{u'}\cdot|y'w'|_{u'}}$$
and similar expressions for 
$|x'z'|_{w'}$, $|y'z'|_{w'}$,
where ``prime'' means the image under
$f$.

By the assumption and Lemma~\ref{lem:cfoliation_isometric},
$f_u$
is isometric along
$\si$
and
$H_\la$.
Hence
$|x'w'|_{u'}=|xw|_u$, $|y'w'|_{u'}=|yw|_u$,
and
$|x'z'|_{u'}=|xz|_u$, $|y'z'|_{u'}=|yz|_u$.
On the other hand,
$$|xy|_w^4=|xz|_w^4+|yz|_w^4,\quad |x'y'|_{w'}^4=|x'z'|_{w'}^4+|y'z'|_{w'}^4$$
by Proposition~\ref{pro:explicit_distance}, and we conclude that
$|x'y'|_{w'}=|xy|_w$.
Therefore
$|x'y'|_{u'}=|xy|_u$. 
\end{proof}

\begin{lem}\label{lem:sphere_isometric_omega} For every
$u\in G$,
every
$v\in G'$,
the map 
$f_u$
is isometric along the suspension
$S_vG=S_{v,w}G$, $w=\phi_G(v)$.
\end{lem}

\begin{proof} Every circle 
$\si$
of the 
$\R$-foliation
of
$S_vG$
is standard in
$Y$,
and
$\si$
intersects
$G$
at two points
$p$, $q=\psi(p)$.
Then by Lemma~\ref{lem:standard_isometry_rlines} applied to
$\om=p$, $f_p$
is isometric along
$\si$.
Thus
$f_v$
is isometric along
$\si$
because
$v\in\si$.
By Lemma~\ref{lem:standard_isometry_G},
$f_v$
is isometric along
$G$.
Using Lemmas~\ref{lem:cfoliation_isometric} and \ref{lem:suspension_isometric},
we obtain that
$f_v$
is isometric along
$S_vG$,
in particular,
$f$
is M\"obius along
$S_vG$.
Note that
$u\in S_vG$
because
$G\sub S_vG$.
Thus for any
$x$, $z\in S_vG$
the 4-tuple
$(v,x,z,u)$
lies in
$S_vG$.
Hence
$\crt(v',x',z',u')=\crt(v,x,z,u)$,
where ``prime'' means the image under
$f$.
On the other hand,
$\crt(v,x,z,u)=(|vx|_u:|vz|_u:|xz|_u)$,
and
$|v'x'|_{u'}=|vx|_u$
by Corollary~\ref{cor:standard_isometry_rcircles} because
$v,x$
lie on a standard 
$\R$-circle.
Therefore
$|x'z'|_{u'}=|xz|_u$.
\end{proof}

\subsection{Isometricity along fibers of the $\R$-foliation}

\begin{lem}\label{lem:positive_root} The polynomial
$g(s)=s^4+c(s+b)^4-d$
with positive
$b,c$
such that
$cb^4-d<0$
has a unique positive root.
\end{lem}

\begin{proof} We have
$\frac{d^2g}{ds^2}=12s^2+12c(s+b)^2>0$
for all 
$s\in\R$,
and
$\frac{dg}{ds}(0)=4cb^3>0$.
Thus
$g$
is convex and a unique minimum point
$\wt s$
of
$g$
is negative,
$\wt s<0$.
Since
$g(\wt s)<g(0)=cb^4-d<0$,
there is a unique positive root
$s_0$
of
$g$.
\end{proof}

\begin{pro}\label{pro:rfoliation_isometric} Let
$S_uH\sub Y$
be a suspension over a
$\C$-circle $H$
with the poles
$u$, $v=\phi_H(u)$,
and let
$K$
be the 
$\C$-circle
through
$u$, $v$.
Assume
$f_u$
is isometric along 
$K$
and along every
$\C$-circle $H_\la$
of the 
$\C$-foliation of
$S_uH$,
and the following assumptions hold 
\begin{itemize}
 \item [(i)] $f\circ\mu_{K,u}(x)=\mu_{K',u'}\circ f(x)$
and
$|v'x'|_{u'}=|vx|_u$
for every
$x\in S_uH$,
where
$\mu_{K,u}:Y_u\to K$
and
$\mu_{K',u'}:X_{u'}\to K'=f(K)$
are projections to the respective
$\C$-lines,
 \item [(ii)] for every
$x\in S_uH$
and every
$\C$-circle $H_\la$
of the 
$\C$-foliation
there is
$y\in H_\la$
such that
$|x'y'|_{u'}=|xy|_u$,
\end{itemize}
where ``prime'' means the image under
$f$.
Then
$f_u$
is isometric along every
$\R$-line $\si\sub S_uH$
of the 
$\R$-foliation.
\end{pro}

\begin{proof} We fix an arclength parametrization
$\si:\R\to Y_u$
of
$\si$
with 
$\si(0)=v$.
We show that
$f\circ\si:\R\to (F\ast F')_{u'}$
is an arclength parametrization of an
$\R$-line $\si'\sub (F\ast F')_{u'}$
through
$v'=f(v)$.

By assumption (i) we know that
$f\circ\mu_{K,u}(\si(t))=\mu_{K',u'}\circ f(\si(t))=v'$
and
$|f\circ\si(t)v'|_{u'}=|\si(t)v|_u=|t|$
for every
$t\in\R$.
 
Given
$t$, $t'\in\R$,
we assume without loss of generality that
$t'>0$
and
$0<|t|<t'$.
We put
$x=\si(t')$, $z=\si(t)$, $p=\si(-t)$,
and assume furthermore that
$|xp|_u>|zp|_u$.
Note that in the case
$t>0$
the last assumption follows from
$|t|<t'$
because
$|xp|_u=t'+t>2t=|zp|_u$.

We let
$H_t\sub S_uH$
be the 
$\C$-circle
of the 
$\C$-foliation
with
$\si(t)\in H_t$,
and note that
$x\in H_{t'}$, $z,p\in H_t$.

\begin{figure}[htbp]
\centering
\psfrag{p}{$p$}
\psfrag{u}{$v$}
\psfrag{y}{$z$}
\psfrag{z}{$x$}
\psfrag{x}{$y$}
\psfrag{Gt}{$H_t$}
\psfrag{Gt'}{$H_{t'}$}
\includegraphics[width=0.4\columnwidth]{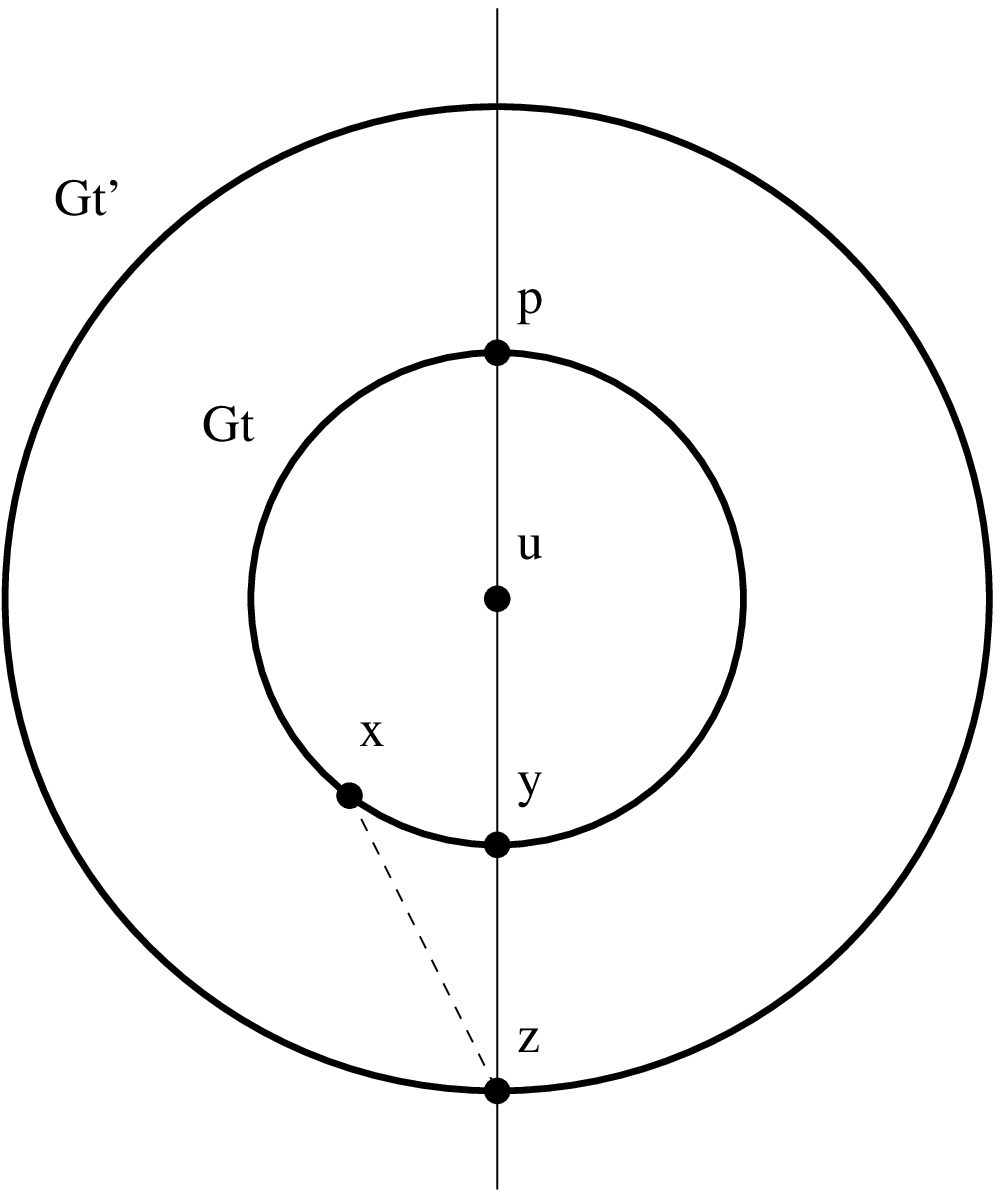}
\end{figure}

By assumption (ii), there is
$y\in H_t$
such that
$|x'y'|_{u'}=|xy|_u$.

\begin{rem}\label{rem:nonflat} Since
$f_u$
is isometric along
$H_t$,
the triangle
$xyz$
has two sides
$|xy|_u$, $|yz|_u$
preserved by
$f_u$.
Our aim is to show that the third distance
$|xz|_u$
is also preserved by
$f_u$.
In the space
$Y_p$
the distance
$|xz|_p$
is uniquely determined by
$|xy|_p$, $|yz|_p$, $|xz|_p^4+|yz|_p^4=|xy|_p^4$
by Proposition~\ref{pro:explicit_distance}. We can pass to
$Y_u$
by a metric inversion. The change of metrics involves
the distances
$|yp|_u$, $|zp|_u$
which are preserved by
$f_u$,
and the distance
$|xp|_u=|xz|_u+|zp|_u$.
In that way, we obtain an equation of 4th degree for
$|xz|_u$
to which we apply Lemma~\ref{lem:positive_root}.
\end{rem}

For brevity, we use notation
$|xy|=|xy|_u$
for 
$x,y\in Y_u$.
The metric inversion with respect to
$p$
gives
$$|xy|_p=\frac{|xy|}{|xp|\cdot|yp|},\quad 
  |xz|_p=\frac{|xz|}{|xp|\cdot|zp|},\quad
  |yz|_p=\frac{|yz|}{|yp|\cdot|zp|}.$$
Thus
\begin{equation}\label{eq:rline_isometric1}
\frac{|xy|^4}{|xp|^4\cdot|yp|^4}=\frac{|xz|^4}{|xp|^4\cdot|zp|^4}
+\frac{|yz|^4}{|yp|^4\cdot|zp|^4},
\end{equation}
and we obtain
$$|xz|^4+\frac{|yz|^4}{|yp|^4}(|xz|+|zp|)^4
 =\frac{|xy|^4\cdot|zp|^4}{|yp|^4},$$
which we write as
\begin{equation}\label{eq:rline_isometric2}
s^4+c(s+b)^4=d,
\end{equation}
where
$s=|xz|$, $b=|zp|$, $c=\frac{|yz|^4}{|yp|^4}$, 
$d=\frac{|xy|^4\cdot|zp|^4}{|yp|^4}$.
The coefficients
$b,c,d$
are positive and preserved by
$f$.
We show that
$cb^4-d<0$.
Using (\ref{eq:rline_isometric1}), we have
$$|xy|^4-|yz|^4=|yz|^4\left(\frac{|xp|^4}{|zp|^4}-1\right)
  +\frac{|xz|^4\cdot|yp|^4}{|zp|^4}>0$$
because
$|xp|>|zp|$
by the assumption. Hence
$$cb^4-d=\frac{|zp|^4}{|yp|^4}(|yz|^4-|xy|^4)<0.$$
By Lemma~\ref{lem:positive_root}, the equation (\ref{eq:rline_isometric2})
has a unique positive solution
$r_0$.
Thus
$r_0=|xz|$.
Since the same equation with the same coefficients 
$b,c,d$
holds in the space
$(F\ast F')_{u'}\sub X_{u'}$
for 
$|x'z'|=|x'z'|_{u'}$,
we obtain
$|x'z'|=|xz|$,
that is,
$f_u$
preserves the distance
$|xz|=|xz|_u$.

Hence in the case
$t>0$
we have
$$|v'x'|_{u'}=|vx|=|vz|+|zx|=|v'z'|_{u'}+|z'x'|_{u'}.$$
By assumption~(i), the points
$v'$, $z'$, $u'$
lie on an
$\R$-circle $\si'\sub X$,
and by the above, the Ptolemy equality holds for 
$\crt(v',z',x',u')$.
Then
$x'\in\si'$
by Proposition~\ref{pro:property_u}. This shows that
$f_u$
is isometric on the ray 
$\si([0,\infty))\sub\si$,
and similarly, 
$f_u$
is isometric on the opposite ray
$\si([0,-\infty)\sub\si$.

Assuming
$t<0$
with
$0<|t|<t'$
we observe that the equality
$|x'z'|_{u'}=|xz|$
implies
$|x'z'|_{u'}=|x'v'|_{u'}+|v'z'|_{u'}$
because
$|x'v'|_{u'}=|xv|$, $|v'z'|_{u'}=|vz|$
and
$|xz|=|xv|+|vz|$.
Then by Lemma~\ref{lem:concatenation}, the concatenation of the rays
$f\circ\si([o,\infty))\cup f\circ\si([0,-\infty))$
is an 
$\R$-line
in
$X_{u'}$.
Thus 
$f_u$
is isometric on
$\si$.
\end{proof}

\subsection{Isometricity along suspensions over $\C$-circles}

\begin{lem}\label{lem:iso_commutes_projections} For every
$u\in G$
the map 
$f$
commutes with the projections
$\mu_{G,u}:Y_u\to G$, $\mu_{F,u'}:(F\ast F')_{u'}\to F$
(in
$X_{u'}$),
$u'=f(u)$,
$$\mu_{F,u'}\circ f(w)=f\circ\mu_{G,u}(w)$$
for every
$w\in Y_u$. 
Moreover,
$|w'z'|_{u'}=|wz|_u$,
where
$z=\mu_{G,u}(w)$, $w'=f(w)$, $z'=\mu_{F,u'}(w')=f(z)$.
\end{lem}

\begin{proof} We can assume that
$w\not\in G$.
There is a standard 
$\R$-circle $\si\sub Y$
through
$w$.
We let
$\{x,y\}=\si\cap G$, $v=\phi_G(w)\in\si$.
Then
$(x,w,y,v)\in\harm_\si$
and
$(x',w',y',v')\in\harm_{\si'}$
for 
$\si'=f(\si)$, $(x',w',v',y')=f(x,w,y,v)$.
By Corollary~\ref{cor:standard_isometry_rcircles},
$f_u$
is isometric along
$\si$,
thus we have
$|w'v'|_{u'}=|wv|_u$, $|x'w'|_{u'}=|xw|_u$.
Therefore,
$$|z''w'|_{u'}=\frac{1}{2}|w'v'|_{u'}=\frac{1}{2}|wv|_u=|zw|_u.$$
for
$z''=\mu_{F,u'}(w')$,
and using Proposition~\ref{pro:explicit_distance} we obtain
$|zx|_u=|x'z''|_{u'}$.
Thus
$z''=z'=f(z)$.
\end{proof}

Recall that
$G'\sub Y$
is a COS to
$G$
at
$\psi$,
and
$G'$
carries the canonical fibration 
$\cF_{G'}$
by
$\C$-circles,
where the 
$\C$-circle $H\sub G'$
of the fibration through
$x\in G'$
is determined by
$\phi_G(x)\in H$.

For distinct
$u$, $v\in G$
the sphere
$S_{u,v}\sub Y$
formed by all 
$\R$-circles 
in
$Y$
through
$u$, $v$
is foliated (except
$u,v$)
by COS's to
$G$.
One can visualize this picture by taking the geodesic
$\ga=uv\sub E$
and regarding the orthogonal projection
$\pr_E:M\to E$.
Then
$S_{u,v}=\di\pr_E^{-1}(\ga)$,
while the fibers
$G_b=\di\pr_E^{-1}(b)$, $b\in\ga$,
are COS's to
$G$,
which foliate 
$S_{u,v}\sm\{u,v\}$.
In the case
$v=\psi(u)$,
i.e., the geodesic
$uv$
passes through
$a$,
the 
$\R$-circles
through
$u$, $v$
are standard, and the sphere
$S_{u,v}$
is called {\em standard}. Every standard sphere
is a suspension over
$G'$.
If 
$S_{u,v}$
is not standard, then for every
$b\in uv\sub E$
the fiber
$G_b$
is the intersection
$G_b=S_{u,v}\cap S_w$
for the uniquely determined standard sphere
$S_w=S_{w,\psi(w)}$, $w\in G$:
one takes the geodesic
$\ga_b\sub E$
through
$a$, $b$,
then
$S_w=\di\pr_E^{-1}(\ga_b)$,
and
$w,\psi(w)\in G$
are the ends at infinity of
$\ga_b$.

Every COS
$G_b$
to
$G$, $b\in uv$,
carries the canonical fibration 
$\cF_b=\cF_{G_b}$
by
$\C$-circles, 
where the 
$\C$-circle $H\sub G_b$
through
$x\in G_b$
is determined by
$\phi_G(x)\in H$.
Note that the 
$\C$-circles $G$, $H$
are mutually orthogonal,
$G\bot H$,
and
$v=\phi_H(u)$.
The suspension
$S_{u,v}H=S_uH$
is a 2-dimensional sphere in
$S_{u,v}$.

\begin{lem}\label{lem:unique_representative_cfiber} Given distinct
$u,v\in G$ 
with 
$v\neq\psi(u)$
and a
$\C$-circle $H\in\cF_{G'}$,
there is a uniquely determined suspension
$S_u\wh H=S_{u,v}\wh H$, $v=\phi_{\wh H}(u)$, 
over a 
$\C$-circle $\wh H\sub Y$
such that every fiber
$\wh H_c$
of the 
$\C$-foliation
of
$S_u\wh H$
is represented as
$\wh H_c=S_u\wh H\cap S_wH$
for some
$w\in G$.
\end{lem}

\begin{proof} Let
$uv\sub E$
be the geodesic with end points at infinity
$u,v$.
By the assumption,
$a\not\in uv$.
We take
$b\in uv$,
consider the geodesic
$\ga_b\sub E$
through
$a$, $b$,
put
$\{x,y\}=\di\ga_b\sub G$,
and let
$S_xH=S_{x,y}H$
be the M\"obius suspension over
$H$
with poles
$x,y$.
The suspension
$S_xH$
intersects the COS
$G'$
over 
$H$.
We put
$\wh H=S_xH\cap G_b$,
the fiber of the 
$\C$-foliation of
$S_uH$
over
$b$.

\begin{figure}[htbp]
\centering
\psfrag{a}{$a$}
\psfrag{b}{$b$}
\psfrag{c}{$c$}
\psfrag{u}{$u$}
\psfrag{v}{$v$}
\psfrag{x}{$x$}
\psfrag{y}{$y$}
\psfrag{z}{$z$}
\psfrag{w}{$w$}
\includegraphics[width=0.4\columnwidth]{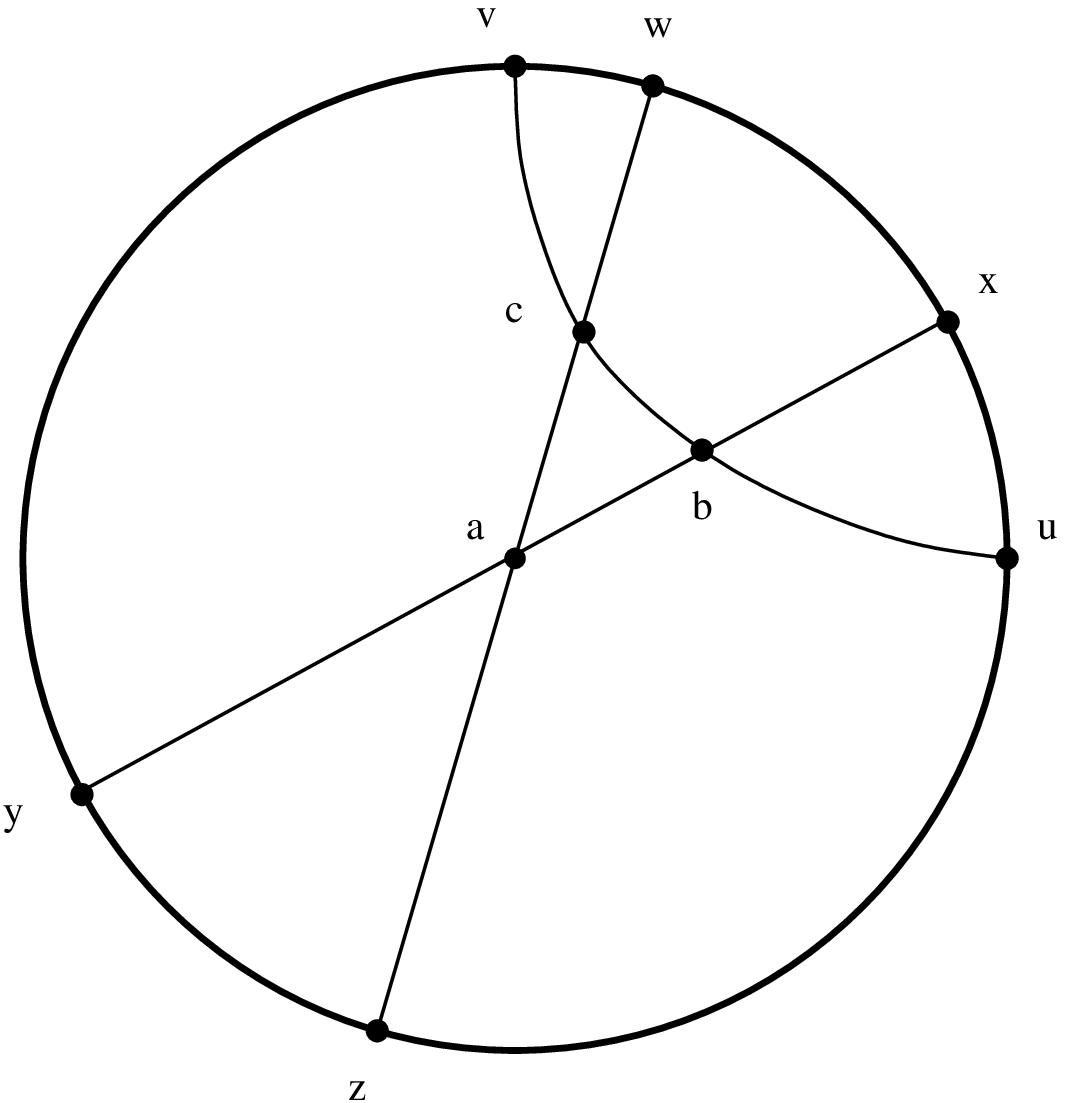}
\end{figure}

For an arbitrary
$c\in uv$
let
$\ga_c\sub E$
be the geodesic through
$a$, $c$, $\{w,z\}=\di\ga_c\sub G$,
and let
$S_wH$
be the suspension over
$H$
with poles
$w,z$.

Every transvection along a geodesic
$\ga\sub M$
induces a M\"obius automorphism of
$Y$,
which is a pure homothety of
$Y_u$
for any end
$u\in\di\ga$.
For brevity, we speak about the action of transvections on
$Y$.

The fiber
$\wh H_c=S_u\wh H\cap G_c$
of the 
$\C$-foliation
of
$S_u\wh H$
over
$c$
is obtained from
$\wh H$
by the transvection along
$uv$,
which moves
$b$
into
$c$,
and
$\wh H$
is obtained from
$H$
by the transvection along
$\ga_b$,
which moves
$a$
into
$b$.
Thus
$\wh H_c$
is obtained from
$H$
by the composition of these transvections.

On the other hand, by Proposition~\ref{pro:holonomy_normal_bundle},
the holonomy of the normal bundle
$E^\bot$
along any loop in
$E$
preserves any complex hyperbolic plane in
$E^\bot$.
It follows that
$\wh H_c$
is obtained from
$H$
by the transvection along
$\ga_c$
that moves
$a$
into
$c$.
Therefore,
$\wh H_c=S_u\wh H\cap S_wH$.
\end{proof}

\begin{lem}\label{lem:nonstandard_suspension_isometric} For every
distinct
$u$, $v\in G$
the map 
$f_\om$
is isometric along any suspension
$S_u\wh H=S_{u,v}\wh H$, $v=\phi_{\wh H}(u)$,
over a 
$\C$-circle $\wh H$
that is orthogonal to
$G$.
\end{lem}

\begin{proof} As above we use a ``prime'' notation for images under
$f$,
and the notation
$f_u:=f:Y_u\to(F\ast F')_{u'}$, $u\in Y$.

We first note that for every
$u\in G$,
the map
$f_\om$
is isometric along any {\em standard} suspension
$S_uH=S_{u,v}H$, 
where
$H\sub G'$
is a 
$\C$-circle
of the canonical fibration and
$v=\psi(u)=\phi_H(u)$.
Indeed, the 
$\R$-foliation
of
$S_uH$
consists of standard
$\R$-circles, 
and by Lemma~\ref{lem:standard_isometry_rlines},
$f_u$
is isometric along every of them. By
Lemma~\ref{lem:standard_isometry_ccircle},
$f_u$
is isometric on
$H$.
Thus by Lemma~\ref{lem:suspension_isometric},
$f_u$
is isometric on
$S_uH$.
Then by Lemma~\ref{lem:metric_inversion_line},
$f_\om$
is isometric on
$S_uH$.

Now, we show that
$f_\om$
is isometric along every suspension
$S_u\wh H=S_{u,v}\wh H$, $u\in G$,
where
$v=\phi_{\wh H}(u)\neq\psi(u)$.
We check that the assumptions of Proposition~\ref{pro:rfoliation_isometric}
are satisfied. In this case,
$K=G$,
and
$f_u$
is isometric along
$G$
because
$f_\om$
is assumed to be isometric along
$G$
and
$u\in G$.
By the first part of the argument,
$f_u$
is isometric on every
$\C$-circle
of the 
$\C$-foliation 
of any standard suspension
$S_wH$, $w\in G$.
Thus by Lemma~\ref{lem:unique_representative_cfiber},
$f_u$
is isometric on every
$\C$-circle
of the 
$\C$-foliation
of
$S_u\wh H$.
Property~(i) follows from Lemma~\ref{lem:iso_commutes_projections}.

To check property~(ii) of Proposition~\ref{pro:rfoliation_isometric},
we take 
$x\in S_u\wh H$
and a
$\C$-circle
$\wh H_\la$
of the 
$\C$-foliation of
$S_u\wh H$.
There is a 
$\C$-circle $\wh H_{\la'}$
of the 
$\C$-foliation
with 
$x\in\wh H_{\la'}$.
By Lemma~\ref{lem:unique_representative_cfiber},
$\wh H_{\la'}=S_u\wh H\cap S_{w'}H$
for some
$w'\in G$,
and there is an
$\R$-circle $\si'\sub S_{w'}H$
of the 
$\R$-foliation
with
$x\in\si'$.
The suspension
$S_{w'}H$
is standard, and the intersection
$\si'\cap H$
consists of two points. We take one of them,
$q\in\si'\cap H$.
Again,
$\wh H_\la=S_u\wh H\cap S_wH$
for some
$w\in G$,
and there is a (standard)
$\R$-circle $\si\sub S_wH$
of the 
$\R$-foliation
with
$q\in\si$.
The intersection
$\si\cap\wh H_\la$
consists of two points, and we take one of them,
$y\in\si\cap\wh H_\la$.
Note that
$\si$, $\si'$
are standard
$\R$-circles
with
$q\in\si\cap\si'$.
Thus the suspension
$S_qG$
over
$G$
with poles
$q$, $\phi_G(q)$
contains
$x$, $y$.
By Lemma~\ref{lem:sphere_isometric_omega},
$f_u$
is isometric on
$S_qG$,
hence
$|x'y'|_{u'}=|xy|_u$.
That is, the assumptions of Proposition~\ref{pro:rfoliation_isometric}
are satisfied. By that Proposition
$f_u$
is isometric along every
$\R$-circle
of the 
$\R$-foliation
of
$S_u\wh H$.
Applying Lemma~\ref{lem:suspension_isometric}, we see that
$f_u$
is isometric along
$S_u\wh H$.
Then by Lemma~\ref{lem:metric_inversion_line},
$f_\om$
is isometric along
$S_u\wh H$.
\end{proof}

\begin{pro}\label{pro:join_isometric} For every
$\C$-circle $H\in\cF_{G'}$
the map 
$f_\om$
is isometric along the M\"obius join
$G\ast H\sub Y$.
\end{pro}

\begin{proof} Given
$x$, $y\in G\ast H$,
we show that there is a suspension
$S_u\wh H=S_{u,v}\wh H$, $u,v=\phi_{\wh H}(u)\in G$
with
$x,y\in S_u\wh H$.

There are standard suspensions
$S_pH$, $S_qH\sub G\ast H$,
$p,q\in G$, 
over 
$H$
with
$x\in S_pH$, $y\in S_qH$.
We can assume that
$\ov x=\pr_E(x)$, $\ov y=\pr_E(y)\in E$
are distinct since otherwise we can take
$S_pH=S_qH$
as the required suspension.
Then there is a unique geodesic
$\ga\sub E$
through
$\ov x$, $\ov y$.
Let
$u$, $v\in G$
be the ends at infinity of
$\ga$.
We can assume that
$v\neq\psi(u)$,
since otherwise again
$S_pH=S_qH$.
Then by Lemma~\ref{lem:unique_representative_cfiber},
there is a uniquely determined suspension
$S_u\wh H$
such that every fiber
$\wh H_c$
of the 
$\C$-foliation
of
$S_u\wh H$
is represented as
$\wh H_c=S_u\wh H\cap S_wH$
for some
$w\in G$.
Then
$x\in S_u\wh H\cap S_pH$, $y\in S_u\wh H\cap S_qH$,
in particular,
$x,y\in S_u\wh H$.

By Lemma~\ref{lem:nonstandard_suspension_isometric},
$|x'y'|_{\om'}=|xy|_\om$,
that is,
$f_\om$
is isometric along
$G\ast H$.
\end{proof}

\begin{rem}\label{rem:case3D} Proposition~\ref{pro:join_isometric}
implies Theorem~\ref{thm:moebius_join_canonical} and
Theorem~\ref{thm:complex_hyperbolic} in the case
$\dim X=3$
because
$Y=G\ast H$
by Proposition~\ref{pro:dim_moeb_join} in that case. 
Moreover, it implies Theorem~\ref{thm:moebius_join_canonical} 
in the case
$F$, $F'$
are mutually orthogonal 
$\C$-circles
in 
$X$.
\end{rem}

\subsection{Properties of the base revisited}

Here we make a digression to prove an important fact,
Proposition~\ref{pro:base_euclidean}.

Let
$M=\C\hyp^2$, $Y=\di M$.
We fix 
$o\in M$
and consider a complex reflection
$h:M\to M$
with respect to a complex hyperbolic plane 
$E\sub M$
through
$o$.
It acts on
$Y$
as a reflection with respect to the 
$\C$-circle 
$F=\di E\sub Y$.
The tangent space
$V=T_oM$
is a 2-dimension complex vector space with a hermitian
form generated by the Riemannian metric, and
$g=dh:V\to V$
is a unitary involution whose fixed point set is the complex line
$T_oE\sub V$.
We formalize this situation as follows.

Let
$V=\C^2$
be the 2-dimensional complex vector space with the standard
hermitian form. A unitary involution
$g:V\to V$
is said to be {\em reflection} if
$\dim_\C\im(1-g)=1$.
Note that
$\det g=-1$.
Thus any product
$g\cdot g'$
of unitary involutive reflections lies in
$S\cU(2)$.
We need the following simple 

\begin{lem}\label{lem:transitive_unitary_reflections} Let
$\cF\sub\cU(2)$
be the set of unitary involutive reflections
$V\to V$.
Then the subgroup
$G\sub S\cU(2)$
generated by products
$g\cdot g'$
with
$g,g'\in\cF$
is transitive on the unit sphere
$S\sub V$,
and for 
$g\in\cF$
there is a dense subset
$\wt\cF\sub\cF$
such that
$g'\cdot g$
has a finite order for every
$g'\in\wt\cF$.
\end{lem}

\begin{proof} We identify
$\cF$
with the set of complex lines in
$V$
and fix 
$g\in\cF$.
Then we have a decomposition
$V=\C\oplus\C$
such that 
$\C\oplus\{0\}=\fix g$
is the fixed point set of
$g$,
and 
$g$
acts on
$V$
as 
$g(a,b)=(a,-b)$.
Note that 
$g$
preserves any real 2-subspace
$L\sub V$
spanned by
$(a,0)$, $(0,b)$
with
$a,b\in\C$, $|a|=|b|=1$,
acting on
$L$
as the reflection with respect to 
$L\cap\fix g$.

Each pair of opposite points in
$L\cap S$
is represented as
$L\cap g'\cap S$
for some
$g'\in\cF$, $g'\cdot g$ 
acts on
$L$
by a rotation, and the subgroup in
$G$
preserving
$L$
acts transitively on
$L\cap S$. 
If
$(g'\cdot g)^k$
is identical on
$L$,
then it is identical,
$(g'\cdot g)^k=\id$. 
Thus there is a dense subset
$\wt\cF\sub\cF$
such that
$g'\circ g$
has a finite order for every
$g'\in\wt\cF$.

Given 
$u$, $v\in S$,
$u=(a_u,b_u)$, $v=(a_v,b_v)$,
we can find
$g_u$, $g_v\in G$
such that
$g_u(u)=(a_u,0)$, $g_v(v)=(0,b_v)$. 
Let
$L\sub V$ 
be the real 2-subspace
$L\sub V$
spanned by
$(a_u,0)$, $(0,b_v)$.
Then there is
$h\in G$
with
$h((a_u,0))=(0,b_v)$.
Hence
$g_v^{-1}\circ h\circ g_u(u)=v$.
\end{proof}

\begin{pro}\label{pro:base_euclidean} For every
$\om\in X$
the canonical metric on the base
$B_\om$
of the projection
$\pi_\om:X_\om\to B_\om$
is an Euclidean one, and the dimension of 
$B_\om$
is even.
\end{pro}

\begin{proof} We fix a
$\C$-circle $F\sub X$
through
$\om$
and a M\"obius involution
$\eta:F\to F$
without fixed points. We show that the group
$G$
of isometries
$X_\om\to X_\om$
preserving
$o=\eta(\om)$
acts on
$(F,\eta)^\perp$
transitively. 

There is
$x\in F$, $x\neq\om,o$,
such that
$(x,o,y,\om)\in\harm_F$,
where
$y=\eta(x)$.
By Lemma~\ref{lem:filling_fiber} the orthogonal
complement
$A=(F,\eta)^\perp\sub X$
can be represented as
$A=S_{x,y}\cap S_{o,\om}$,
where
$S_{x,y}$
is the sphere in
$X$
between
$o,\om$
through
$x,y$
and
$S_{o,\om}$
is the sphere in
$X$
between
$x,y$
through
$o,\om$.
In the space
$X_\om$
the sphere
$S_{o,\om}$
consists of
$\R$-lines
through
$o$,
and
$A=\set{a\in S_{o,\om}}{$|oa|=r$}$
for some
$r>0$,
where
$|oa|=|oa|_\om$.
For simplicity of notation we assume that
$r=1$.

By Lemma~\ref{lem:orthogonal_foliated_ccircles}
there is a canonical fibration 
$\cF$
of
$A$
by
$\C$-circles
each of which is invariant under the reflection
$\phi_F:X_\om\to X_\om$
with respect to
$F$,
in particular,
$\dim A\ge 1$
is odd. Since
$\pi_\om|A:A\to\pi_\om(A)$
is homeomorphism, the unit sphere
$\ov A=\pi_\om(A)\sub B=B_\om$
centered at
$\ov o=\pi_\om(o)$
has an odd dimension
$\ge 1$.
It follows that
$\dim B$
is even. 

Every
$H\in\cF$
and
$F$
are mutually orthogonal,
$F\perp H$.
Proposition~\ref{pro:join_isometric} implies that 
the M\"obius join
$F\ast H\sub X$
is M\"obius equivalent to the standard model space
$Y=\di\C\hyp^2$.
If
$\dim A=1$
then
$A=H$, $X=F\ast H$
is M\"obius equivalent to 
$Y=\di\C\hyp^2$,
and there is nothing to prove. Thus we assume that
$\dim A\ge 3$.

The reflection
$\phi_H:X\to X$
with respect to
$H$
preserves
$F$
and
$\phi_H|F=\eta$,
furthermore, by Lemma~\ref{lem:fiber_involution}, 
$\phi_H$
preserves
$A$
and its fibration 
$\cF$
for every
$H\in\cF$.
Thus the composition
$\phi_{H'}\circ\phi_H:X\to X$
preserves
$A$
and its fibration
$\cF$, 
and it acts on
$F$
identically for each 
$\C$-circles $H$, $H'\in\cF$.
In particular,
$\phi_{H'}\circ\phi_H:X_\om\to X_\om$
is an isometry.

We fix a
$\C$-circle $H\in\cF$.
For every
$H'\in\cF$
that is orthogonal to
$H$, $H'\sub(H,\phi_F|H)^\perp\cap A$,
the M\"obius join
$Z=H\ast H'\sub X$
is M\"obius equivalent to 
$Y$
by Proposition~\ref{pro:join_isometric}, hence
$Z$
satisfies axioms (E) and (O). We show that
$Z\sub A$.
Note that
$Z$
is invariant under the reflection
$\phi_F$, $\phi_F(Z)=Z$,
because
$H$, $H'$
are. Thus
$Z$
carries a fibration
$\cF_Z$
by
$\C$-circles
invariant under
$\phi_F$,
in particular, every
$\C$-circle $K\in\cF_Z$
is orthogonal to
$F$, $K\perp F$,
and
$\phi_K(F)=F$.

To see that
$\phi_K:X\to X$
preserves
$Z$,
we note that there is a (unique)
$\C$-circle $K'\in\cF_Z$
that is orthogonal to
$K$
at
$\phi_F|K$,
actually,
$K'=(K,\phi_F|K)^\perp$
in
$Z$.
Then
$\phi_K(K')=K'$,
and since
$\phi_K$
acts on
$K$
identically, we see that
$\phi_K:X\to X$
preserves
$Z$, $\phi_K(Z)=Z$,
because
$Z$
can be represented as the M\"obius join
$Z=K\ast K'$.
Using conjugation via a M\"obius isomorphism
$Y\to Z$
and the fact that
$Y=\di\C\hyp^2$,
we can consider
$\phi_K$
as a unitary involutive reflection of
$V=\C^2$.
Applying Lemma~\ref{lem:transitive_unitary_reflections},
we find a dense subset
$\wt\cF_Z\sub\cF_Z$
such that the composition
$g=\phi_K\circ\phi_H$,
when restricted to
$Z$,
has a finite order,
$(g|Z)^k=\id_Z$, for every
$\C$-circle $K\in\wt\cF_Z$
and some
$k\in\N$
depending on
$K$.
Note that
$g|F:F\to F$
being a composition of two M\"obius involutions
$\phi_K|F$, $\phi_H|F$
without fixed points has an infinite order or is identical.
In the first case there are two fixed points
$u$, $v\in F$
for 
$g|F$,
and 
$g$
acts on
$X_u$
as a nontrivial homothety. This contradicts the fact that
$(g|Z)^k=\id_Z$.
Therefore,
$g|F=\id_F$
and thus
$\phi_K|F=\phi_H|F=\eta$.
Consequently,
$K\sub A$.
By continuity, this holds for every
$K\in\cF_Z$,
hence
$Z\sub A$.

For every
$K,K'\in\cF_Z$
the product
$g=\phi_K\circ\phi_{K'}:X\to X$
is identical on
$F$,
hence
$g:X_\om\to X_\om$
is an isometry. By Lemma~\ref{lem:transitive_unitary_reflections}
isometries of this type act on
$Z$
transitively. Varying 
$H\in\cF$, $H'\in(H,\phi_F|H)^\perp\cap A$,
we see that the M\"obius joins
$Z=H\ast H'$
cover
$A$.
Thus the group 
$G$
acts on
$A$
transitively. It follows that the group of isometries of the base
$B$
preserving
$\ov o$
acts on the unit sphere 
$\ov A\sub B$
centered at
$\ov o$
transitively. Using the standard argument with L\"ovner
ellipsoid, we obtain that the unit sphere
$\ov A$
is an ellipsoid and thus the metric of
$B$
is Euclidean.
\end{proof}

\subsection{Suspension over a COS}
\label{subsect:COS_suspension}

We turn back to the proof of Theorem~\ref{thm:moebius_join_canonical}
and to notations of sect.~\ref{subsect:proof_moebius_join_canonical}.

Given distinct
$u$, $v\in Y$,
let 
$K\sub Y$
be the 
$\C$-circle
through
$u$, $v$.
The sphere
$S_{u,v}\sub Y$
formed by all 
$\R$-circles
in
$Y$
through
$u$, $v$
carries 
$\R$-
and COS-foliations, see sect.~\ref{subsect:rcfoliations}. 
Note that every fiber
$H$
of the COS-foliation of
$S_{u,v}$
satisfies axioms (E) and (O) because
$H$
is the boundary at infinity of an orthogonal complement in
$M=\C\hyp^{k+1}$
to the complex hyperbolic plane
$E\sub M$
with
$\di E=K$.

Now we are able to prove the following generalization of
Lemma~\ref{lem:suspension_isometric}.

\begin{pro}\label{pro:sphere_isometric_cos} Given distinct
$u$, $v\in Y$
assume
$f_u$
is isometric along every
$\R$-circle
of the 
$\R$-foliation
of
$S=S_{u,v}$
and every fiber of the COS-foliation of
$S$.
Then
$f_u$
is isometric along
$S$.
\end{pro}

\begin{proof} We take
$x$, $y\in S$
and show that
$|x'y'|_{u'}=|xy|_u$,
where we use ``prime'' notation for images under
$f$.

Every fiber
$H$
of the COS-foliation of
$S$
lies in a sphere between
$u$, $v$, 
that is, there is
$r=r(H)>0$
such that
$|vh|_u=r$
for every
$h\in H$.
In the space
$Y_u$
every
$\R$-circle $\si$
of the 
$\R$-foliation
of
$S$
is an
$\R$-line 
through
$v$.
Since
$f_u$
is isometric along
$\si$,
we can assume that neither
$x$
nor 
$y$
coincides with
$u$
or
$v$.
Then there are uniquely determined an
$\R$-line $\si\sub S$
of the 
$\R$-foliation
through
$x$
and a fiber
$H\sub S$
of the COS-foliation through
$y$,
and we also assume that
$x\not\in H$, $y\not\in\si$,
since otherwise there is nothing to prove. The intersection
$\si\cap H$
consists of two points, say,
$z$, $w$
such that the pair
$z,w$
separate the pair
$u,v$
on
$\si$.
The points
$v,z,u,w$
subdivide 
$\si$
into four arcs, and we assume without loss of generality
that
$x$
lies on the arc 
$vz\sub\si$.

Since
$H$
satisfies axioms (E) and (O), the 
$\C$-circle $K\sub Y$
through
$w$, $y$
lies in
$H$,
and we assume that
$z\not\in K$
because otherwise Lemma~\ref{lem:suspension_isometric}
applies. Since
$K$
is a 
$\C$-line 
in
$Y_w$,
we have
$\ov y\neq\ov z$,
where ``bar'' means the image under
$\pi_w:Y_w\to B_w$.

It follows from Lemma~\ref{lem:bisector} that 
$H$
together with
$K$
lies in the bisector in
$Y_w$
between
$u$, $v$.
Thus
$|uz|_w=|vz|_w$, $|uk|_w=|vk|_w$
for every
$k\in K$
and therefore
$|\ov u\,\ov y|=|\ov v\,\ov y|$.
Note that
$\si$
is an
$\R$-line 
also in
$Y_w$,
thus
$\ov u\,\ov v=\ov\si$, $\ov x\in\ov u\,\ov v$
and
$\ov z\in\ov u\,\ov v$
is the midpoint between
$\ov u$, $\ov v$.
The metric on the base
$B_w$
is Euclidean, thus the line
$\ov z\,\ov y\sub B_w$
through
$\ov z$, $\ov y$
is orthogonal to the line
$\ov u\,\ov v$.
Therefore
$|\ov x\,\ov y|^2=|\ov x\,\ov z|^2+|\ov z\,\ov y|^2$.

In the space
$X$,
we have the sphere
$S_{u',v'}$
with poles
$u'$, $v'$,
the image
$S'=f(S)\sub S_{u',v'}$,
the 
$\R$-circle $\si'=f(\si)\sub S'$
with 
$v',x',z',u',w'\in\si'$
(in this cyclic order) and
$H'=f(H)$
with
$y',z',w'\in H'$.
By our assumption on
$f_u$
we have
$|h'v'|_{u'}=r$
for every
$h'\in H'$,
and
$K'=f(K)\sub H'$
is a
$\C$-circle
through
$w'$, $y'$.
Thus
$H'$
lies in a sphere in
$X$
between
$u'$, $v'$,
hence
$H'$
together with the 
$\C$-line
$K'$
lies in the bisector in
$X_{w'}$
between
$u'$
and
$v'$.

By Proposition~\ref{pro:base_euclidean}, the base
$B_{w'}$
of the canonical foliation
$\pi_{w'}:X_{w'}\to B_{w'}$
is also Euclidean. Thus we obtain as above
$|\ov x'\ov y'|^2=|\ov x'\ov z'|^2+|\ov z'\ov y'|^2$.
Note that
$f_w$
is isometric along
$\si$,
hence
$|\ov x\,\ov z|=|xz|_w=|x'z'|_{w'}=|\ov x'\ov z'|$.
Since
$f_u$
is isometric along
$H$,
we see that
$f_w$
is also isometric along 
$H$
applying the metric inversion with respect to
$w$,
thus
$|\ov z\,\ov y|=|\ov z'\ov y'|$.
It follows that
$|\ov x'\ov y'|=|\ov x\,\ov y|$.

Using that
$|v'h'|_{u'}=r=|vh|_u$
for every
$h\in H$
and applying the metric inversion with respect to
$w$,
we obtain 
$|vh|_w=\frac{1}{|hw|_u}=|uh|_w$,
$|v'h'|_{w'}=\frac{1}{|h'w'|_{u'}}=|u'h'|_{w'}$.
Hence the map 
$f_w$
preserves the distances of
$u$, $v$
to points of
$K$, $|v'h'|_{w'}=|u'h'|_{w'}=|vh|_w=|uh|_w$,
in particular,
$|v'y'|_{w'}=|vy|_w$.

Let
$\wt v=\mu_{K,w}(v)$, $\wt x=\mu_{K,w}(x)$, $\wt z=\mu_{K,w}(z)\in K$
be the projections in
$Y_w$
of
$v,x,z$
to the 
$\C$-line $K$.
Since the map 
$f_w$
preserves the distances of
$v$
to points of
$K$,
we see using the distance formula that
$f(\wt v)=\mu_{K',w'}(v')=:\wt v'$,
i.e. the projections
$\mu_{K,w}:Y_w\to K$, $\mu_{K',w'}:X_{w'}\to K'$
commute with
$f$
at
$v$.
Similarly,
$f(\wt z)=\mu_{K',w'}(z')=:\wt z'$.

On the other hand, the point
$\wt v$
can be obtained from
$\wt z$
by lifting isometry (see sect.~\ref{subsect:lifting_isometry})
of the oriented triangle
$T=\ov v\,\ov y\,\ov z\sub B_w$,
$\wt v=\tau_T(\wt z)$.
Analogously, in the space
$X_{w'}$
the point
$\wt v'$
can be obtained from
$\wt z'$
by lifting isometry of the oriented triangle
$T'=\ov v'\ov y'\ov z'\sub B_{w'}$,
$\wt v'=\tau_{T'}(\wt z')$.
Thus
$|\tau_{T'}|=|\wt z'\wt v'|_{w'}=|\wt z\,\wt v|_w=|\tau_T|$.

Using Lemma~\ref{lem:area_lift_triange} we obtain
$$|\tau_P|^2=\frac{|xz|}{|vz|}|\tau_T|^2=\frac{|x'z'|}{|v'z'|}|\tau_{T'}|^2
  =|\tau_{P'}|^2$$
for the oriented triangles
$P=\ov x\,\ov y\,\ov z\sub B_w$,
$P'=\ov x'\ov y'\ov z'\sub B_{w'}$.
We assume that a fiber orientation of
$X_w\to B_w$
is fixed and that the fiber orientation of
$X_{w'}\to B_{w'}$
is induced by
$f_w$.
Then
$f(\wt x)=f\circ\tau_P(\wt z)=\tau_{P'}\circ f(\wt z)$
and we conclude that the projections
$\mu_{K,w}$, $\mu_{K',w'}$
commute with
$f$
at
$x$, $f(\wt x)=\wt x'$
for 
$\wt x'=\mu_{K',w'}(x')$.
Hence
$$|x'y'|_{w'}^4=|\ov x'\ov y'|^4+|y'\wt x'|_{w'}^4=|\ov x\,\ov y|^4+|y\wt x|_w^4
  =|xy|_w^4$$
by Proposition~\ref{pro:explicit_distance}. Applying
the metric inversion with respect to
$u$
and using that
$|u'x'|_{w'}=|ux|_w$, $|u'y'|_{w'}=|uy|_w$,
we finally obtain
$|x'y'|_{u'}=|xy|_u$. 
\end{proof}

\subsection{Proof of Theorems~\ref{thm:moebius_join_canonical}
and \ref{thm:complex_hyperbolic}}

\begin{proof}[Proof of Theorem~\ref{thm:moebius_join_canonical}]
We show that the map 
$f_\om=f:Y_\om\to(F\ast F')_{\om'}$
is isometric.

First, we note that for every
$u\in G$
the map 
$f_u$
is isometric along the standard sphere
$S_u=S_{u,v}$, $v=\psi(u)$.
Indeed,
$S_u=S_uG'$
is a suspension over
$G'$,
and
$f_u$
is isometric along the fibers of the 
$\R$-foliation of
$S_uG'$
because they are standard
$\R$-lines
in
$Y_u$,
and along
$G'$
by Lemma~\ref{lem:standard_isometry_ccircle}. Thus
by Lemma~\ref{lem:cfoliation_isometric},
$f_u$
is isometric along every fiber of the COS-foliation of
$S_uG'$.
By Proposition~\ref{pro:sphere_isometric_cos},
$f_u$
is isometric along
$S_u=S_uG'$.
It follows from Lemma~\ref{lem:metric_inversion_line} that
$f_\om$
is isometric along
$S_u$.

Given
$x$, $y\in Y$
we can assume that the points
$\ov x=\pr_E(x)$, $\ov y=\pr_E(y)\in E$
are distinct and neither of them coincides with
$o$,
since otherwise
$x$, $y$
lie in a standard sphere
$S_u$
for some
$u\in G$,
hence
$|x'y'|_{\om'}=|xy|_\om$.
Then there is a unique geodesic
$\ga\sub E$
through
$\ov x$, $\ov y$.
Let
$u$, $v\in G$
be the ends at infinity of
$\ga$, $v\neq\psi(u)$
by our assumption. Then
$x,y\in S_{u,v}=\pr_E^{-1}(\ga)$.

The map 
$f_u$
is isometric along the fibers of the COS-foliation
of the sphere
$S_{u,v}$
because every such a fiber is represented as the fiber
$S_{u,v}\cap S_w$
of the COS-foliation of a standard sphere
$S_w$
for some
$w\in G$.
Furthermore, 
$S_{u,v}=S_{u,v}\wh G$
is a suspension over a fiber
$\wh G$
of the COS-foliation of
$S_{u,v}$,
and every fiber
$\si$
of the 
$\R$-foliation of
$S_{u,v}$
is a fiber of the 
$\R$-foliation
of
$S_{u,v}\wh H$
for a
$\C$-circle $\wh H$
of the canonical fibration of
$\wh G$.
Thus as in Lemma~\ref{lem:nonstandard_suspension_isometric} we see that
$f_u$
is isometric along 
$\si$.
By Proposition~\ref{pro:sphere_isometric_cos},
$f_u$
is isometric along
$S_{u,v}$.
By Lemma~\ref{lem:metric_inversion_line} again,
$f_\om$
is isometric along
$S_{u,v}$.
Hence
$|x'y'|_{\om'}=|xy|_\om$.
This completes the proof of Theorem~\ref{thm:moebius_join_canonical}.
\end{proof}

\begin{proof}[Proof of Theorem~\ref{thm:complex_hyperbolic}]
It follows from Proposition~\ref{pro:base_euclidean} that
$\dim X$
is odd, 
$\dim X=2k-1$, $k\ge 1$.
If
$k=1$,
then
$X=\di\C\hyp^1$
is a 
$\C$-circle.
In the case
$k=2$,
Theorem~\ref{thm:complex_hyperbolic} is already proved, see 
Remark~\ref{rem:case3D}. Thus we assume that
$k\ge 3$
and argue by induction over dimension. We take any
mutually orthogonal
$\C$-circles $F$, $F'\sub X$.
By Proposition~\ref{pro:empty_3D} their respective
orthogonal complements
$A$, $A'$
have nonempty intersection
$X'=A\cap A'$.
By Corollary~\ref{cor:axioms_orthogonal_complements},
$X'$
satisfies axioms (E) and (O), hence
$X'=\di\C\hyp^{k-2}$
by the inductive assumption. By Theorem~\ref{thm:moebius_join_canonical}
the join
$X''=F'\ast X'$
is M\"obius equivalent to
$\di\C\hyp^{k-1}$.
Since
$F'$, $X'\sub A$,
we see as in the proof of Proposition~\ref{pro:base_euclidean} that
$X''\sub A$,
and therefore 
$X''$
is a COS to
$F$.
Applying Theorem~\ref{thm:moebius_join_canonical} once again to
$F\ast X''$,
we obtain that 
$X=F\ast X''$
is M\"obius equivalent to
$\di\C\hyp^k$.
\end{proof}


\bigskip
\begin{tabbing}

Sergei Buyalo,\hskip11em\relax \= Viktor Schroeder,\\

St. Petersburg Dept. of Steklov \>
Institut f\"ur Mathematik, Universit\"at \\

Math. Institute RAS, Fontanka 27, \>
Z\"urich, Winterthurer Strasse 190, \\

191023 St. Petersburg, Russia,\>  CH-8057 Z\"urich, Switzerland\\
St. Petersburg State University, \> {\tt vschroed@math.unizh.ch}\\
{\tt sbuyalo@pdmi.ras.ru}\\

\end{tabbing}


\begin{thebibliography}{ABC}

\bibitem[BS1]{BS1} S.~Buyalo, V.~Schroeder, M\"obius structures and Ptolemy
spaces: boundary at infinity of complex hyperbolic spaces, arXive:1012.1699, 
2010.

\bibitem[BS2]{BS2} S.~Buyalo, V.~Schroeder, M\"obius characterization
of the boundary at infinity of rank one symmetric spaces, Geometriae Dedicata, 
2013, DOI:10.1007/s10711-013-9906-6

\bibitem[BS3]{BS3} S.~Buyalo, V.~Schroeder, Elements of asymptotic geometry,
EMS Monographs in Mathematics, 2007, 209 pages.

\bibitem[GKM]{GKM} D.~Gromoll, W.~Klingenberg, W.~Meyer, Riemannsche 
Geometrie im Grossen, Lecture Notes in Mathematics, Vol. 55. 
Springer-Verlag, Berlin-New York, 1975. vi+287 pp.


\bibitem[FS1]{FS1} T.~Foertsch, V.~Schroeder, Metric M\"obius geometry and 
a characterization of spheres,  Manuscripta Math. 140 (2013), no. 3-4, 613--620. 

\bibitem[FS2]{FS2} T.~Foertsch, V.~Schroeder, Hyperbolicity,
$\CAT(-1)$-spaces and Ptolemy inequality, Math. Ann. 350 (2011), no. 2, 339--356.

\bibitem[HL]{HL} P.~Hitzelberger, A.~Lytchak, Spaces with many affine functions,
Proc. Amer. Math. Soc.  135  (2007),  no. 7, 2263--2271.


\end{thebibliography}
\end{document}